\documentclass[a4paper,11pt]{amsart}
\usepackage{amssymb}
\usepackage{amscd}
\usepackage{comment}
\usepackage{amsmath,amsthm}
\usepackage[colorlinks=true]{hyperref}
\usepackage{enumerate}
\usepackage{booktabs,multirow}
\usepackage{tikz}
\usepackage{rotating}
\usetikzlibrary{patterns}
\usetikzlibrary{decorations.pathreplacing}
\usetikzlibrary{calc,through}
\usetikzlibrary{trees,arrows,matrix,fit}

\allowdisplaybreaks[1]
\numberwithin{figure}{section}
\numberwithin{equation}{section}

\title{Noncommutative crossing partitions}
\author[K.~Shigechi]{Keiichi~Shigechi}
\email{k1.shigechi AT gmail.com}
\date{\today}

\newcommand\tikzpic[2]{
\raisebox{#1\totalheight}{
\begin{tikzpicture}
#2
\end{tikzpicture}
}}
\newtheorem{theorem}[figure]{Theorem}
\newtheorem{example}[figure]{Example}
\newtheorem{lemma}[figure]{Lemma}
\newtheorem{defn}[figure]{Definition}
\newtheorem{prop}[figure]{Proposition}
\newtheorem{cor}[figure]{Corollary}

\newtheorem{remark}[figure]{Remark}
\begin{document}
\begin{abstract}
We define and study noncommutative crossing partitions which are a generalization of 
non-crossing partitions.
By introducing a new cover relation on binary trees, we show that the partially 
ordered set of noncommutative crossing partitions is a graded lattice.
This new lattice contains the Kreweras lattice, the lattice of non-crossing partitions,
as a sublattice.
We calculate the M\"obius function, the number of maximal chains and 
the number of $k$-chains in this new lattice by constructing an explicit $EL$-labeling on the lattice.
By use of the $EL$-labeling, we recover the classical results on the Kreweras lattice.
We characterize two endomorphism on the Kreweras lattice, the Kreweras complement map and the involution
defined by Simion and Ullman, in terms of the maps on the noncommutative crossing partitions. 
We also establish relations among three combinatorial objects: labeled $k+1$-ary trees, $k$-chains 
in the lattice, and $k$-Dyck tilings.
\end{abstract}

\maketitle

\section{Introduction}
A {\it non-crossing partition} of the set $[n]:=\{1,2,\ldots,n\}$ is a 
partition $\pi$ of $[n]$ such that
if four integers satisfy $a<b<c<d$, a block $B_{1}$ contains $a$ and $c$ 
and another block $B_{2}$ contains $b$ and $d$, then 
$B_1$ and $B_2$ coincides with each other.
In 1972, G.~Kreweras systematically studied the non-crossing 
partition lattice, which we call the Kreweras lattice \cite{Kre72} (see also \cite{Pou72} by Y.~Poupard). 
It was shown that the M\"obius function is equal to the Catalan number up to a sign, 
and the number of intervals is equal to the Fuss--Catalan number.
In \cite{Edel80,Edel82}, P.~H.~Edelman constructed a bijection between $k+1$-ary trees and $k$-chains in the
Kreweras lattice. There, the number of maximal chains in the lattice is equal to the number of parking functions, 
or equivalently the number of rooted labeled forests \cite[Corollary 3.3]{Edel80}.
The combinatorial aspects of $k$-chains are extensively studied in \cite{Edel80,Edel82,EdelSim94}. 
R.~P.~Stanley established a bijection between maximal chains in the Kreweras lattice and parking functions 
in \cite{Sta97} by constructing an explicit $EL$-labeling on the Kreweras lattice.
In addition, non-crossing partitions appear in various context:
algebraic combinatorics \cite{Arm09,Bia97,Rei97,Sim94,Sim00}, geometric group theory \cite{Arm09,Rei97} and 
noncommutative probability theory \cite{NicSpe97,Spe94,Spe97}.

We will study {\it noncommutative crossing partitions}, which 
are a natural generalization of the notion of non-crossing partitions.
In the theory of non-crossing partitions, a partition satisfies the following 
two properties: the order of blocks are irrelevant and two blocks are non-crossing.
In the case of a noncommutative crossing partition, the order of blocks in a partition 
plays an important role and two blocks may be crossing.
By introducing an appropriate cover relation on the set $\mathtt{NCCP}(n)$ of noncommutative 
crossing partitions, we construct a new lattice which contains the Kreweras lattice, 
the lattice of non-crossing partitions, as a sublattice.
We denote by $\mathtt{NCP}(n)$ the set of non-crossing partitions.
The cover relation $\subseteq$ on the Kreweras lattice is defined by use of a refinement of the blocks.
On the other hand, the cover relation $\lessdot$ on $\mathtt{NCCP}(n)$ 
introduced in this paper is defined on the binary trees, 
which are bijective to the noncommutative crossing partitions.  
To embed the Kreweras lattice into the lattice of $\mathtt{NCCP}(n)$, we consider 
the subset $\mathtt{NCCP}(n;312)$, which is the set of $312$-avoiding noncommutative 
crossing partitions.
We establish the lattice isomorphism between two lattices $(\mathtt{NCP}(n),\subseteq)$ and 
$(\mathtt{NCCP}(n;312),\le)$.
Under this isomorphism, the two cover relations $\lessdot$ and $\subset$ are compatible 
with each other.
Since the cover relation $\lessdot$ is defined in $\mathtt{NCCP}(n)$, 
it can be regarded as a natural generalization of the cover relation $\subseteq$.

In \cite{Kre72}, Kreweras defined an endomorphism on the lattice $(\mathtt{NCP}(n),\subseteq)$, 
which we call Kreweras complement map.
Similarly, in \cite{SimUll91}, Simion and Ullman defined an involution similar to the Kreweras complement map.
The latter map can be obtained by modifying the Kreweras complement map.
Both maps are defined on $\mathtt{NCP}(n)$. 
By use of the isomorphism between $\mathtt{NCP}(n)$ and $\mathtt{NCCP}(n;312)$, 
we study the relation between these two maps and two other maps on $\mathtt{NCCP}(n)$.
By characterizing these two maps, Kreweras complement map and the involution, on $(\mathtt{NCCP}(n;312),\le)$, 
we show that they are compatible with certain maps from $\mathtt{NCCP}(n;312)$ to $\mathtt{NCP}(n)$.

By the results of Bj\"orner \cite{Bjo80} and Edelman and Simion \cite{EdelSim94}, it is known 
that the lattice $(\mathtt{NCP}(n),\subseteq)$ admits an $EL$-labeling.
If a poset admits an $EL$-labeling, then this poset is lexicographically shellable and 
it implies Cohen-Macaulay \cite{Bjo80,BjoGarSta82,BjoWac83}.
As in the case of the Kreweras lattice, we show that the lattice $\mathcal{L}_{NCCP}(n):=(\mathtt{NCCP}(n),\le)$ 
is bounded and graded. 
Further, through an explicit construction of an $EL$-labeling, the lattice $\mathcal{L}_{NCCP}(n)$ 
is lexicographically shellable.
This $EL$-labeling has an advantage such that it is defined through the definition of a cover relation,
and directly related to a generalization of parking functions, labeled parking functions.
Here, a labeled parking function is a pair of a parking function and a permutation (see Section \ref{sec:lpk}).
By using this explicit $EL$-labeling on $\mathcal{L}_{NCCP}(n)$, 
we calculate the M\"obius function (Theroem \ref{thrm:moebius}), 
the number of maximal chains (Theorem \ref{thrm:mchainNCCP}), and 
the number of $k$-chains (Theorem \ref{thrm:interval}).
Since we introduce a pair of a parking function and a permutation as a refinement of 
the $EL$-labeling, we study the multiplicities of a permutation associated to 
the calculation of the M\"obius function (Section \ref{sec:mult}).
Table \ref{table:summary} is a summary of combinatorial properties of two lattices, the Kreweras lattice 	
and $\mathcal{L}_{NCCP}(n)$.
\begin{table}[ht]
\begin{tabular}{c|c|c}
 & Kreweras lattice $\mathtt{NCP}(n)$ & $\mathcal{L}_{NCCP}(n)$ \\ \hline
(signless) M\"obius function & Catalan number &  $(2n-3)!!$ \\ \hline 
Number of maximal chains &  $n^{n-2}$ & $((n-1)!)^2$ \\ \hline
Number of $k-1$-chains & Fuss--Catalan number & $\prod_{j=0}^{n-1}((k-1)j+1)$ \\ 
\end{tabular}
\vspace*{10pt}
\caption{Summary of combinatorial properties of the Kreweras lattice and $\mathcal{L}_{NCCP}(n)$}
\label{table:summary}
\end{table}

The difference of combinatorial structures between the Kreweras lattice and $\mathcal{L}_{NCCP}(n)$ 
is that the former is characterized by combinatorial objects without labels, however, 
the latter is with labels.
Thus, by ignoring the labels on the combinatorial objects on $\mathcal{L}_{NCCP}(n)$, 
we can recover the results for the Kreweras lattice obtained in 
\cite{Edel80,Edel82,Kre72,Sim00}: the M\"obius function (Theorem \ref{cor:KreMoebius}), 
the number of maximal chains (Theorem \ref{thrm:mchainKre}), 
and the number of $k$-chains (Theorem \ref{thrm:interval}).

The study of chains in the Kreweras lattice in \cite{Edel80,Edel82} reveals a relation 
between $k+1$-ary trees and $k$-chains.
The cardinality of $k$-chains is given by the Fuss--Catalan number, which 
is a generalization of the Catalan number.
We study $k$-chains in the lattice $\mathcal{L}_{NCCP}(n)$ and connect 
them with labeled $k+1$-ary trees.
On the other hand, labeled $k+1$-ary trees are bijective to 
another combinatorial object, $k$-Dyck tilings (generalized Dyck tilings).
A $k$-Dyck tiling is a generalization of Dyck tilings studied 
in \cite{JosVerKim16,KeyWil12,KimMesPanWil14,Shi21,ShiZinJus12}.
Thus, we make a new bridge between $k$-chains and $k$-Dyck tilings.

The set of $k$-Dyck tilings are divided into two classes: the first one is $k$-Dyck tilings 
consisting of trivial tiles, and the other is $k$-Dyck tilings with non-trivial 
tiles. If an element in a $k$-chain $c$ is not in $\mathtt{NCCP}(n;312)$, this 
chain corresponds to a $k$-Dyck tiling with non-trivial tiles.
Conversely, if all element in the chain $c$ are in $\mathtt{NCCP}(n;312)$,
$c$ corresponds to a $k$-Dyck tiling consisting of only trivial tiles.
This correspondence is equivalent, by the isomorphism between 
$\mathtt{NCCP}(n;312)$ and $\mathtt{NCP}(n)$, to the correspondence between 
a $k$-chain in the lattice of non-crossing partitions and a 
$k$-Dyck path obtained in \cite{Edel80,Edel82}.

Schematically in Figure \ref{fig:cco}, we have the correspondences among 
combinatorial objects which appear in this paper.
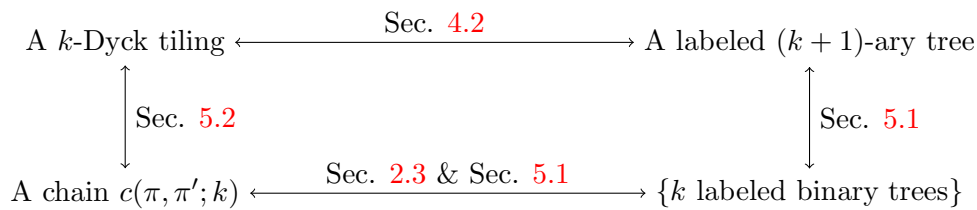
\begin{figure}[ht]
\begin{align*}
\begin{split}
\begin{tikzpicture}
\node (0) at (0,0) {A labeled $(k+1)$-ary tree};
\node (1) at (-9,0) {A $k$-Dyck tiling};
\node (2) at (-9,-2){A chain $c(\pi,\pi';k)$};
\node (3) at (0,-2){$\{ k \text{ labeled binary trees}\}$};
\draw[<->,anchor=south] (0.west) to node{Sec. \ref{sec:GDyckT}}(1.east);
\draw[<->,anchor=west] (0.south) to node{Sec. \ref{sec:deckbt}}(3.north);
\draw[<->,anchor=south] (2.east) to node{Sec. \ref{sec:lbt} $\&$ Sec. \ref{sec:deckbt}}(3.west);
\draw[<->,anchor=west] (1.south) to node{Sec. \ref{sec:interlat}}(2.north);
\end{tikzpicture}
\end{split}
\end{align*}
\caption{Correspondences among combinatorial objects.}
\label{fig:cco}
\end{figure}
These correspondences are a natural generalization of 
the correspondences studied in \cite{Edel80,Edel82,EdelSim94}.

The paper is organized as follows.
In Section \ref{sec:NCCP}, we introduce the notion of noncommutative 
crossing partitions and study the relation to labeled binary trees.
We also give a recurrence relation for the number of noncommutative crossing partitions
of a certain type. This number is expressed as the number of paths in a lattice.
In Section \ref{sec:LatNCCP}, we define a cover relation on 
labeled binary trees. 
This gives a view of the poset $(\mathtt{NCCP}(n),\le)$ as a lattice.
We study the lattice structure and show that the Kreweras lattice 
is a sublattice of this new lattice.
As for the Kreweras lattice, we study two endomorphisms: the Kreweras complement
map in \cite{Kre72} and an involution studied in \cite{SimUll91}.
By constructing an explicit $EL$-labeling on the lattice $\mathcal{L}_{NCCP}(n)$, 
we calculate the M\"obius function, and the number of maximal chains.
In Section \ref{sec:kDyck}, we briefly review the relation between 
labeled $k$-ary trees and $k$-Dyck tilings.
In Section \ref{sec:int}, we give a decomposition of a labeled $k$-ary 
trees in terms of $k-1$ labeled binary trees.
Then, we consider the chains of length $k$ in $\mathcal{L}_{NCCP}(n)$ 
and make a bridge to $k$-Dyck tilings introduced in Section \ref{sec:kDyck}.

\section{Noncommutative crossing partitions and labeled trees}
\label{sec:NCCP}
\subsection{Noncommutative crossing partitions}
We first introduce the notion of noncommutative crossing partitions.
They are a generalization of non-crossing partitions and possess similar properties.
Especially, the set of non-crossing partitions is contained as a subset.

\begin{defn}
A noncommutative crossing partition $\pi:=(\pi_{1},\pi_{2},\ldots,\pi_{l})$ of length $l$ in $[n]:=\{1,2,\ldots,n\}$  
is a set of crossing partitions (or blocks) $\pi_{i}$ with $1\le i\le l$ which satisfies the following three conditions:
\begin{enumerate}
\item Each integer $p\in[n]$ appears exactly once in $\pi$.
\item Each $\pi_{i}$, $1\le i\le l$, is an increasing integer sequence from left to right.
\item We have $\max(\pi_{i})>\min(\pi_{i+1})$ for $1\le i\le l-1$.
\end{enumerate}
We denote by $\mathtt{NCCP}(n)$ the set of noncommutative crossing partition in $[n]$.
\end{defn}

\begin{remark}
Consider two elements in $\mathtt{NCCP}(4)$, $\pi=24/13$ and $\pi'=13/24$.
If we ignore the order of blocks, we have two blocks: $13$ and $24$.
These are the simplest examples of noncommutativity and crossing partitions, 
since we distinguish $\pi$ from $\pi'$.
\end{remark}

For simplicity, we write $\pi:=(\pi_1,\ldots,\pi_{l})$
as $\pi_1/\pi_{2}/\cdots/\pi_{l}$.
For example, $\pi=(24,3,1)$ in $\{1,2,3,4\}$ is written as $24/3/1$.

We define two classes of noncommutative crossing partitions:
non-crossing partitions and canonical partitions.
\begin{defn}
Let $\pi:=(\pi_{1},\ldots,\pi_{l})\in\mathtt{NCCP}(n)$.
Suppose $1\le a<b<c<d\le n$.
Then, we call $\pi$ a non-crossing partition if there is no $i\neq j$ such that 
$a,c\in\pi_{i}$ and $b,d\in\pi_{j}$.
\end{defn}

\begin{defn}
A partition $\pi:=(\pi_{1},\ldots,\pi_{l})\in\mathtt{NCCP}(n)$ is called canonical 
if it satisfies 
\begin{align*}
\min(\pi_{l})<\min(\pi_{l-1})<\ldots<\min(\pi_{1}).
\end{align*}
\end{defn}

We have two crossing partitions for $n=4$: $24/13$ and $13/24$.
Other partitions for $n=4$ are non-crossing.
Similarly, the partition $24/13$ is canonical, but $13/24$ is not.

\begin{defn}
We denote by $\mathtt{NCP}(n)$ the set of 
non-crossing canonical partitions in $[n]$.
\end{defn}

\begin{remark}
The set $\mathtt{NCP}(n)$ is nothing but the set of non-crossing partitions 
in $[n]$ studied in the literature (see for exapmle \cite{Kre72,Pou72,Sim00} 
and references therein).
\end{remark}

Given a noncommutative crossing partition $\pi:=\pi_1/\pi_2/\cdots/\pi_{l}$, 
we write each crossing partition as $\pi_{i}:=\pi_{i,1}\ldots \pi_{i,p}$ where 
$p$ is the length of $\pi_{i}$.  

\begin{defn}
A crossing partition $\pi\in\mathtt{NCCP}(n)$ is $312$-avoiding
if and only if 
there is no integers $1\le i<j\le k\le l$ and $p,q,r$ such that 
\begin{align*}
\pi_{j,q}<\pi_{k,r}<\pi_{i,p}.
\end{align*}
We denote by $\mathtt{NCCP}(n;312)$ is the set of $312$-avoiding 
partitions in $\mathtt{NCCP}(n)$.
\end{defn}

\subsection{Enumerations of noncommutative crossing partitions}
We first show that a noncommutative crossing partitions is bijective 
to a permutation.
We give several enumerative results regarding to partitions.

\begin{prop}
The number of noncommutative crossing partitions of $[n]$ 
is $n!$, i.e., $|\mathtt{NCCP}(n)|=n!$.
\end{prop}
\begin{proof}
We have a natural bijection between a noncommutative crossing partition 
and a permutation.
Given a $\pi:=(\pi_1,\ldots,\pi_{l})\in\mathtt{NCCP}(n)$, we have a unique permutations $w:=w_{1}\ldots w_{n}$ in $[n]$
by concatenating $\pi_{i}$, $1\le i\le l$, from left to right.
This map from $w$ to $\mu$ is obviously invertible, since we have a unique noncommutative crossing partition 
by inserting ``$/$" between $w_{i}$ and $w_{i+1}$ if $w_{i}>w_{i+1}$.

From these observations, the number $\mathtt{NCCP}(n)$ is equal to the number of permutations in $[n]$, which 
completes the proof.
\end{proof}

The following two results are classical results on non-crossing partitions \cite{Kre72,Sim00} 
and the number of permutations with a pattern avoidance.
\begin{prop}
\label{prop:cardNCP}
The number of non-crossing canonical partitions of $[n]$
is the $n$-th Catalan number, i.e., 
\begin{align*}
|\mathtt{NCP}(n)|=\genfrac{}{}{}{}{1}{n+1}\genfrac{(}{)}{0pt}{}{2n}{n}.
\end{align*}
\end{prop}
\begin{proof}
Let $A_{n}:=|\mathtt{NCP}(n)|$ and $\pi\in\mathtt{NCP}(n)$ in one-line notation. 
We divide $\pi$ into two pieces $\pi_{l}$ and $\pi_{r}$ where 
$\pi_l$ (resp. $\pi_r$) is left (resp. right) to $n$ in $\pi$.
It is obvious that $\pi_l$ and $\pi_r$ are in $\mathtt{NCP}(|\pi_{l}|)$ and 
$\mathtt{NCP}(|\pi_r|)$ where $|\pi|$ is the size of $\pi$.
Thus, $A_{n}$ satisfies the recurrence relation
\begin{align*}
A_{n}=\sum_{k=0}^{n-1}A_{k}A_{n-1-k},
\end{align*}
which is nothing but the recurrence equation for the Catalan numbers.
\end{proof}

\begin{prop}
The number of $312$-avoiding noncommutative crossing partitions of $[n]$ 
is given by the $n$-th Catalan number.
\end{prop}
\begin{proof}
Let $\pi\in\mathtt{NCCP}(n;312)$. 
As in the case of Proposition \ref{prop:cardNCP}, we divide $\pi$ into 
two pieces $\pi_l$ and $\pi_r$ where $\pi_{l}$ (resp. $\pi_{r}$) is left (resp. right)
to $n$ in $\pi$.
Then, by the same argument in the proof of Proposition \ref{prop:cardNCP}, 
we have $|\mathtt{NCCP}(n;312)|=|\mathtt{NCP}(n)|$.
This completes the proof.
\end{proof}

Let $A(n,l)$ be the number of noncommutative crossing partition of length $l$ in $[n]$.
Then, $A(n,l)$ satisfies the recurrence equation
\begin{align*}
A(n,l)=l\ A(n-1,l)+(n-l+1)A(n-1,l-1).
\end{align*} 
The integers $A(n,l)$ are known as the Eulerian numbers and 
satisfy the formula
\begin{align*}
A(n,l)=\sum_{j=0}^{l}(-1)^{j}\genfrac{(}{)}{0pt}{}{n+1}{j}(l-j)^{n},
\end{align*}
where $n\ge1$ and $1\le l\le n$.
These integers $A(n,l)$ appear in the sequence A008292 in OEIS \cite{Slo}.

Similarly, $B(n,l)$ be the number of non-crossing canonical partitions of length $l$ in $[n]$.
Then, $B(n,l)$ is given by the formula
\begin{align*}
B(n,l)=\genfrac{}{}{}{}{1}{n}\genfrac{(}{)}{0pt}{}{n}{l}\genfrac{(}{)}{0pt}{}{n}{l-1},
\end{align*}
where $n\ge1$ and $1\le l\le n$.
The integers $B(n,l)$ are known as the Narayana numbers and 
appear in the sequence A001263 in OEIS \cite{Slo}.

We say that $\pi:=(\pi_1,\ldots,\pi_{l})\in\mathtt{NCCP}(n)$ is type $(s_1,\ldots,s_l)$ 
if each $\pi_{i}$ consists of $s_{i}$ letters.
For example, $4/23/15$ is type $(1,2,2)$. 
Let $\mathbf{s}:=(s_1,\ldots,s_{l})$ be a sequence of positive integers, which 
satisfy $\sum_{i=1}^{l}s_{i}=n$.
We denote by $T(\mathbf{s})$ the number of elements in $\mathtt{NCCP}(n)$ of type 
$\mathbf{s}$.

\begin{defn}
\label{defn:ssetS}
Given $\mathbf{s}:=(s_1,\ldots,s_{l})$, we define two sets of positive integers:
\begin{align*}
S_{1}(\mathbf{s})&:=\{i\le l-1 | s_{i}>1 \}\cup\{l\}, \\
S_{2}(\mathbf{s})&:=\{i\le l-1| s_{i+1}\ge2 \}.
\end{align*}
\end{defn}

\begin{lemma}
\label{lemma:T}
The number $T(\mathbf{s})$ satisfies the recurrence equation
\begin{align}
\label{eq:recT}
T(\mathbf{s})=
\sum_{i\in S_{1}(\mathbf{s})}T(s_1,\ldots, s_{i}-1,\ldots,s_{l})
+\sum_{i\in S_{2}(\mathbf{s})}T(s_1,\ldots,s_{i-1}, s_{i}+s_{i+1}-1, s_{i+2},\ldots,s_{l}),
\end{align}
with the initial condition $T(1)=1$.
\end{lemma}
\begin{proof}
Suppose $\pi$ is a partition of type $\mathbf{s}$ and $n=\sum_{i}s_{i}$.
First, $\pi$ is obtained from a partition $\pi'$ of type $\mathbf{s'}$ with $\sum_{i}s'_{i}=n-1$
such that $\mathbf{s'}=(s_1,\ldots,s_{i-1},s_{i}-1,s_{i+1},\ldots,s_{l})$ and $i\in S_{1}(\mathbf{s})$.
In fact, we increase the integers in $\pi'$ by one, insert $1$ at position $i$ and obtain $\pi$.
Secondly, suppose we have an increasing sequence of size $m$ $\mathbf{a}:=a_1<a_2<\ldots<a_{m}$ with $a_1>2$. 
If we insert one between $a_{i}$ and $a_{i+1}$, $1\le i$, we have two increasing sequences 
$\mathbf{a'}:=a_1<a_2\ldots<a_{i}$ and $\mathbf{a''}:=1<a_{i+1}<\ldots<a_{m}$.
From this, the partition $\pi$ can be obtained from $\pi'$ of type $\mathbf{s'}$ with $\sum_{i}s'_{i}=n-1$
such that $\mathbf{s'}=(s_1,\ldots,s_{i-1},s_{i}+s_{i+1}-1,s_{i+2},\ldots,s_{l})$.
From these observations, we have Eq. (\ref{eq:recT}).
\end{proof}

For example, we have 
\begin{align*}
&T(4)=1, \\
&T(3,1)=T(1,3)=3,\quad T(2,2)=5, \\
&T(2,1,1)=T(1,1,2)=3,\quad T(1,2,1)=5, \\
&T(1,1,1)=1.
\end{align*}
The integers $T(\mathbf{s})$ appear in the sequence A335845 in OEIS \cite{Slo}.

The number $T(\mathbf{s})$ possesses a symmetric property.
\begin{lemma}
Let $\mathbf{s}^{\mathrm{rev}}:=(s_l,s_{l-1},\ldots,s_{1})$.
Then, we have $T(\mathbf{s}^{\mathrm{rev}})=T(\mathbf{s})$. 
\end{lemma}
\begin{proof}
Suppose that $\mathbf{s}:=(1^{l})$.
It is obvious that we have $T(\mathbf{s}^{\mathrm{rev}})=T(\mathbf{s})$ 
since $\mathbf{s}^{\mathrm{rev}}=\mathbf{s}$.

We define the set $\widetilde{S}$ of indices in $[l]$ such that if $s_{i}\ge2$, then $i\in \widetilde{S}$.
We define an involution $\overline{i}:=l+1-i$ for $i\in[l]$. 
We denote by $\overline{P}$ the set $\{\overline{i} | i\in P\}$ and 
by $P\pm1$ the set $\{i\pm1| i\in P\}$ for some set $P$.
We write $\mathbf{s}\xrightarrow{i}\mathbf{s'}$ if 
$i\in S_{2}(\mathbf{s})$ and $s':=(s_1,\ldots,s_{i-1},s_{i}+s_{i+1}-1,s_{i+2},\ldots,s_{l})$.

We consider four cases: (1) $1,l\in\widetilde{S}$,  (2) $1\in\widetilde{S}$ and $l\notin\widetilde{S}$, 
(3) $1\notin\widetilde{S}$ and $l\in\widetilde{S}$ and (4) $1,l\notin\widetilde{S}$.

\paragraph{Case (1)}
Since $1,l\in\widetilde{S}$, we have $S_{1}(\mathbf{s})=\widetilde{S}=\overline{S_{1}(\mathbf{s}^{\mathrm{rev}})}$.
By definition of $S_{2}(\mathbf{s})$, 
we have $S_{2}(\mathbf{s})=\widetilde{S}-1\setminus\{0\}$ and 
$S_{2}(\mathbf{s}^{\mathrm{rev}})=\overline{\widetilde{S}}-1\setminus\{0\}$.
It is easy to see that $|S_{2}(\mathbf{s}^{\mathrm{rev}})|=|S_{2}(\mathbf{s})|=|\widetilde{S}|-1$.
Suppose that we have $\mathbf{s}\xrightarrow{i}\mathbf{s'}$.
Then, there exists a unique $j$ such that 
$\mathbf{s}^{\mathrm{rev}}\xrightarrow{j}\mathbf{s'}^{\mathrm{rev}}$.
From these observations, we have $T(\mathbf{s}^{\mathrm{rev}})=T(\mathrm{s})$.

\paragraph{Case (2)}
By definition of $S_{1}(\mathbf{s})$, 
we have $S_{1}(\mathbf{s})=\widetilde{S}\cup\{l\}$, $S_{1}(\mathbf{s}^{\mathrm{rev}})=\overline{\widetilde{S}}$, and 
$|S_{1}(\mathbf{s})|=|S_{1}(\mathbf{s}^{\mathrm{rev}})|+1$. 
Therefore, the reverse sequence of $(s_1,\ldots,s_{l-1})$ can not be constructed 
from $\mathbf{s}^{\mathrm{rev}}$ and $S_{1}(\mathbf{s}^{\mathrm{rev}})$ 
since $1\notin S_{1}(\mathbf{s}^{\mathrm{rev}})$.
Since $1\in\widetilde{S}$ and $l\notin\widetilde{S}$, 
$S_{2}(\mathbf{s})=\widetilde{S}-1\setminus\{0\}$ and 
$S_{2}(\mathbf{s}^{\mathrm{rev}})=\overline{\widetilde{S}+1}$.
We have $|S_{2}(\mathbf{s}^{\mathrm{rev}})|=|S_{2}(\mathbf{s})|+1$.
Note that the sequence $\mathbf{s'}:=(s_1,\ldots,s_{l-1})$ can not 
be constructed from $S_{2}(\mathbf{s})$, however, the reverse sequence
$\mathbf{s'}^{\mathrm{rev}}$ can be constructed from $S_{2}(\mathbf{s}^{\mathrm{rev}})$.
We have 
$|S_{1}(\mathbf{s})|+|S_{2}(\mathbf{s})|=|S_{1}(\mathbf{s}^{\mathrm{rev}})|+|S_{2}(\mathbf{s}^{\mathrm{rev}})|$
and there is a one-to-one correspondence between 
$i$ and $j$ such that  
$i\in S_{1}(\mathbf{s})$ or $S_{2}(\mathbf{s})$ and 
$j\in S_{1}(\mathbf{s}^{\mathrm{rev}})$ or $S_{2}(\mathbf{s}^{\mathrm{rev}})$.
Especially, we have a unique $j\in S_2(\mathbf{s}^{\mathrm{rev}})$ which 
corresponds to $l\in S_{1}(\mathbf{s})$.
Therefore, we have $T(\mathbf{s})=T(\mathbf{s}^{\mathrm{rev}})$.

\paragraph{Case (3)}
This case is essentially the same as Case (2) since one can replace $\mathbf{s}$ by $\mathbf{s}^{\mathrm{rev}}$ 
in Case (2), which is equivalent to consider Case (3).

\paragraph{Case (4)}
This case is essentially the same as Case (1). 
The difference is that we have $S_{2}(\mathbf{s})=\widetilde{S}-1$ and 
$S_{2}(\mathbf{s}^{\mathrm{rev}})=\overline{\widetilde{S}+1}$.

In all cases, we have $T(\mathbf{s})=T(\mathbf{s}^{\mathrm{rev}})$, which completes the proof.
\end{proof}

We define a cover relation $\lessdot$ on sequences $\mathbf{s}$ of positive integers:
\begin{align*}
\mathbf{s}\lessdot\mathbf{s'}, 
\end{align*}
if and only if 
\begin{align*}
\mathbf{s'}=(s_1,\ldots,s_{i-1},s_{i}-1, s_{i+1},\ldots,s_{l}), \quad i\in S_{1}(\mathbf{s}),
\end{align*}
or 
\begin{align*}
\mathbf{s'}=(s_1,\ldots,s_{i-1},s_{i}+s_{i+1}-1, s_{i+2},\ldots,s_{l}), \quad i\in S_{2}(\mathbf{s}),
\end{align*}
where $\mathbf{s}=(s_1,\ldots,s_{l})$ and the sets $S_{i}(\mathbf{s})$, $i=1,2$, are defined in Definition \ref{defn:ssetS}.
Note that the sum of elements in $\mathbf{s}'$ is one less than that in $\mathbf{s}$.

By construction, it is obvious that we have a graded bounded lattice with the greatest element $(1)$ 
and the minimum element $\mathbf{s}$.
We have a natural grading by $n:=\sum_{i=1}^{l}s_{i}$.
We denote by $\mathcal{L}(\mathbf{s})$ the lattice obtained as above.
The next proposition gives an formula of $T(\mathbf{s})$ in terms of 
paths in the Hasse diagram of the lattice $\mathcal{L}(\mathbf{s})$.
\begin{prop}
The number $T(\mathbf{s})$ is given by 
\begin{align}
\label{eq:Tinpaths}
T(\mathbf{s})=\#\{\text{ paths from } \mathbf{s} \text{ to } (1) \text{ in } \mathcal{L}(\mathbf{s})\},
\end{align}
\end{prop}
\begin{proof}
When $s=(1)$, we have $T(\mathbf{s})=1$ and Eq. (\ref{eq:Tinpaths}) holds trivially. 	
Suppose that Eq. (\ref{eq:Tinpaths}) holds up to $n=\sum_{i}s_{i}$.
Then, from Lemma \ref{lemma:T} and the cover relation, we have 
\begin{align*}
T(\mathbf{s})&=\sum_{\mathbf{s'}:\mathbf{s}\lessdot\mathbf{s'}}T(\mathbf{s'}), \\
&=\sum_{\mathbf{s'}:\mathbf{s}\lessdot\mathbf{s'}}\#\{\text{ paths from } \mathbf{s'} \text { to } (1) \}, \\
&=\#\{\text{ paths from } \mathbf{s} \text{ to } (1) \text{ in } \mathcal{L}(\mathbf{s})\},
\end{align*}
which completes the proof.
\end{proof}

\begin{example}
Figure \ref{fig:typeS} is an example of the lattice 
$\mathcal{L}(\mathbf{s})$ with $\mathbf{s}=(1,3,1)$.
The number $T(1,3,1)=11$ and the number of paths from 
$(1,3,1)$ to $(1)$ is also $11$ in Figure \ref{fig:typeS}.
\begin{figure}[ht]
\begin{tikzpicture}
\node (131) at (0,0){$(1,3,1)$};
\node (31) at (-2,1.5){$(3,1)$};
\node (13) at (0,1.5){$(1,3)$};
\node (121) at (2,1.5){$(1,2,1)$};
\node (3) at (-3,3){$(3)$};
\node (21) at (-1,3){$(2,1)$};
\node (12) at (1,3){$(1,2)$};
\node (111) at (3,3){$(1,1,1)$};
\node (2) at (-1,4.5){$(2)$};
\node (11) at (1,4.5){$(1,1)$};
\node (1) at (0,6){$(1)$};
\foreach \a in {31,13,121}
\draw (131.north)--(\a.south);
\foreach \a in {3,21}
\draw (31.north)--(\a.south);
\foreach \a in {3,12}
\draw (13.north)--(\a.south);
\foreach \a in {21,12,111}
\draw (121.north)--(\a.south);
\foreach \a in {3,21,12}
\draw (\a.north)--(2.south);
\foreach \a in {21,12,111}
\draw (\a.north)--(11.south);
\foreach \a in {2,11}
\draw (1.south)--(\a.north);
\end{tikzpicture}
\caption{The Hasse diagram of the lattice $\mathcal{L}(1,3,1)$.}
\label{fig:typeS}
\end{figure}
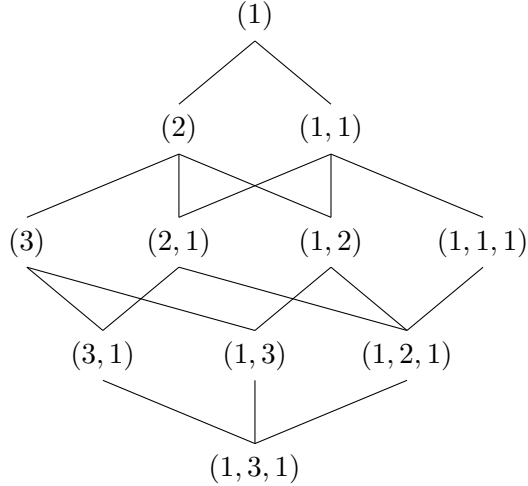
\end{example}

\subsection{Labeled binary trees}
\label{sec:lbt}
A {\it binary tree} is a rooted tree in which each node has 
at most two children.
Each child is referred as either the left child or the right child.
A binary tree is said to be complete if each node has exactly 
two children.
An {\it internal node} is a node which has child nodes.
An {\it external node} is a node that does not 
have child nodes.
An external node is sometimes called a leaf of a tree.
Similarly, an {\it internal edge} is an edge that has 
child nodes, and an {\it external edge} is an edge that 
does not have child nodes.

We identify a complete binary tree with a binary tree 
by deleting external edges from the complete tree.
It is obvious that this identification is a bijection 
between complete binary trees and (incomplete) binary trees.

Suppose that a binary tree $T$ has $n$ nodes.
We introduce a notion of labeled trees.
\begin{defn}
\label{def:labeledtree}
A labeled tree $L(T)$ is a tree $T$ which 
has integers in $[n]$ on nodes.
It satisfies the following conditions:
\begin{enumerate}
\item Each integer $i\in[n]$ appears exactly once in $L(T)$.
\item The labels of nodes are increasing from the root to a leaf.
\end{enumerate}
\end{defn}

Given a node $n_{1}$ in a tree $T$, we have a unique sequence of edges from 
$n_{1}$ to the root. 
We denote by $\mathtt{seq}(n_{1})$ the sequence of edges obtained 
as above.
We say that $\mathtt{seq}(n_2)$ is right to $\mathtt{seq}(n_1)$ 
if and only if $\mathtt{seq}(n_1)\cap \mathtt{seq}(n_2)=\emptyset$
and $n_2$ is right to $n_1$.
Similarly, we say that $\mathtt{seq}(n_2)$ is weakly right 
to $\mathtt{seq}(n_1)$ if and only if 
$\mathtt{seq}(n_1)\cap \mathtt{seq}(n_2)$ may not be empty and 
$T$ contains a subtree in which $n_2$ is right to $n_1$.

\begin{defn}
Suppose that the label of a node $n_2$ is larger than that of a node $n_1$ 
in a labeled tree. 
The node $n_{2}$ is said to be weakly right to the node $n_{1}$ 
if and only if the node $n_{2}$ is a descendant node of $n_1$ or 
$\mathtt{seq}(n_{2})$ is weakly right to $\mathtt{seq}(n_1)$.
\end{defn}

Given a binary tree $T$, we define a {\it canonical} labeling 
on $T$ as follows.
\begin{defn}
\label{defn:canotree}
Let $n(i)$ and $n(i+1)$ be two nodes with labels $i$ and $i+1$ 
respectively.
Then, $T$ is said to be canonical if and only if $n(i+1)$ is 
weakly right to the node $n(i)$ for all $1\le i\le n-1$.
\end{defn}

We will construct a bijection $\phi$ between permutations $w$ in $[n]$ 
and labeled trees $L(T)$ with $n$ nodes.
We define $\phi: L(T)\mapsto w$ as follows.
Given a rooted labeled tree (not necessarily binary) $L$, we denote 
by $\mathrm{word}(T)$ the word obtained from $L$ by the in-order.
Here, the in-order means that we visit first the left subtree, secondly 
the root, then finally the right subtree.

The inverse map $\phi^{-1}:w\mapsto L(T)$ is constructed recursively as follows:
\begin{enumerate}
\item If $w=1$, $\phi^{-1}(w)$ is a binary tree with a single node with a label $1$.
\item 
Suppose $w\in\mathcal{S}_{n}$ and $w_{i}=n$ for some $i\in[n]$. 
We have a unique permutation $w'$ obtained from $w$ by deleting $w_{i}$. 
We denote by $L(T')$ the labeled binary tree corresponding 
to $w'$. 
We have four cases: (a) $i=1$ or $i=n$, (b) $w_{i-1}>w_{i+1}$ for $2\le i\le n-1$ and 
(c) $w_{i-1}<w_{i+1}$ for $2\le i\le n-1$.
\begin{enumerate}
\item A labeled binary tree $L(T)$ is obtained from $L(T')$ 
by adding a left (resp. right) node labeled $n$
to the node labeled $w'_{1}$ (resp. $w'_{n-1}$) if $i=1$ (resp. $i=n$).
\item Add a right node labeled $n$ to the node labeled  $w'_{i-1}$ in $L(T')$.
\item Add a left node labeled $n$ to the node labeled $w'_{i+1}$ in $L(T')$.
\end{enumerate}
\end{enumerate}

For example, a labeled binary tree for $w=83741526$ is depicted as in Figure \ref{fig:bt}.
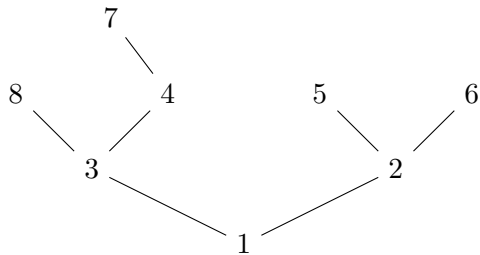
\begin{figure}[ht]
\begin{tikzpicture}
[grow'=up,level distance=1cm, 
level 1/.style={sibling distance=4cm},
level 2/.style={sibling distance=2cm},
level 3/.style={sibling distance=1.5cm},
]
\node {$1$} child{ node {$3$}
		 child{node{$8$}}
		 child{node{$4$}
		 	child{node{$7$}}
		 	child[missing]
		 }
		}
	     child{node{$2$}
	     	child{node{$5$}}
	     	child{node{$6$}}};
\end{tikzpicture}
\caption{A labeled binary tree for $w=83741526$}
\label{fig:bt}
\end{figure}

\begin{prop}
The number of labeled binary trees with $n$ nodes is 
$n!$.
\end{prop}
\begin{proof}
It is obvious from the bijection $\phi$ between permutations in $[n]$ and 
labeled binary trees with $n$ nodes.
\end{proof}

\begin{prop}
\label{prop:clbt}
The number of canonical labeled binary trees with $n$ nodes
is the $n$-th Catalan number.
\end{prop}
\begin{proof}
Given a binary tree $T$, we have a unique canonical labeling of $T$.
Conversely, given a canonical labeling of a binary tree $T$, 
we have a binary tree by forgetting labels on nodes.
Thus, the number of canonical labeled binary tree with $n$ nodes 
is equal to the number of (unlabeled) binary trees, which is known 
to be the $n$-th Catalan number.
\end{proof}

\subsection{A bijection between \texorpdfstring{$\mathtt{NCP}(n)$}{NCP(n)} 
and \texorpdfstring{$\mathtt{NCCP}(n;312)$}{NCCP(n;312)}}
In the previous section, we introduce a bijection $\phi$ between 
$\mathtt{NCCP}(n)$ and a labeled binary tree with $n$ nodes.	
In this section, we introduce another map $\psi$ from a canonical labeled 
binary tree to an element in $\mathtt{NCCP}(n)$.

\begin{lemma}
\label{lemma:312clbt}
We have
\begin{align*}
\pi\in\mathtt{NCCP}(n;312)
\Leftrightarrow
\text{A labeled binary tree } \phi^{-1}(\pi) \text{ is canonical}.
\end{align*}
\end{lemma}
\begin{proof}
(Proof for $\Rightarrow$): 
Since $\pi\in\mathtt{NCCP}(n;312)$, there is no labels 
$n_i$, $1\le i\le 3$, such that $n_{1}<n_{2}<n_{3}$,   
$\mathtt{seq}(n_2)$ is right to $\mathtt{seq}(n_3)$, 
and the largest label in $\mathtt{seq}(n_2)\cap\mathtt{seq}(n_3)$ 
is $n_1$.
This means that a labeled tree $\phi^{-1}(\pi)$ satisfies the 
condition that $\mathtt{seq}(n_2)$ is weakly right to $\mathtt{seq}(n_1)$
if $n_1<n_2$. Thus, a labeled tree $\phi^{-1}(\pi)$ is canonical. 

(Proof for $\Leftarrow$): 
Since the label of the tree $\phi^{-1}(\pi)$ is canonical, the sequence of edges 
$\mathtt{seq}(n_2)$ is weakly right to $\mathtt{seq}(n_1)$ if $n_1<n_2$.
Recall that the partition $\pi$ is obtained from $\phi^{-1}(\pi)$ by reading the labels of 
nodes in the in-order. 
Then, $n_2$ is left and adjacent to $n_1$ or right to $n_1$ in $\pi$ if $n_2=n_1+1$.
This means that $\pi$ has no pattern $312$.
\end{proof}

Given a canonical labeled tree $L(T)$, we will construct a map 
$\psi: L(T)\mapsto\pi'\in\mathtt{NCCP}(n)$.  
Note that $\phi$ is also a map from a canonical labeled tree 
to a noncommutative crossing partition. 
The map $\psi$ is different from $\phi$ for a general canonical labeled binary tree.

We decompose a canonical labeled tree $L(T)$ by erasing the left edges 
from $L(T)$. Then, we obtain chains of nodes which are connected by 
only right edges.
For each chain, we obtain an increasing integer sequence as a block.
By sorting these increasing sequences in the canonical order,
we obtain a unique canonical partition in $\mathtt{NCCP}(n)$.
We denote by $\psi(L(T))$ the canonical partition obtained 
from $L(T)$.

\begin{remark}
Consider the restriction of the composition $\psi\circ\phi^{-1}$ to the set $\mathtt{NCCP}(n;312)$. 
Then, we have a map $\psi\circ\phi^{-1}|_{\mathtt{NCCP}(n;312)}:\mathtt{NCCP}(n;312)\rightarrow\mathtt{NCP}(n)$, 
and this map is a bijection as we will see below.
\end{remark}

\begin{example}
Take $\pi=24/3/1$. Then, we have
\begin{align*}
24/3/1\quad \xrightarrow{\phi^{-1}}\quad
\tikzpic{-0.5}{
\node (1) at (0,0){$1$};
\node (2) at (-0.8,0.8){$2$};
\node (3) at (0,1.6){$3$};
\node (4) at (-0.8,2.4) {$4$};
\draw (1)--(2)--(3)--(4);
}
\quad\xrightarrow{\psi}\quad 4/23/1.
\end{align*}
Note that $24/3/1\notin \mathtt{NCP}(4)$, $24/3/1\in\mathtt{NCCP}(4;312)$, and $4/23/1\in\mathtt{NCP}(4)$.
\end{example}

\begin{lemma}
\label{lemma:psiL}
Suppose that $L(T)$ be a canonical labeled tree.
Them, $\psi(L(T))$ gives the canonical non-crossing partition.
\end{lemma}
\begin{proof}
Recall that $\psi$ is the map from $L(T)$ to a partition $\pi\in\mathtt{NCCP}(n)$.
By definition of $\psi$, $\pi$ is canonical. 
Therefore, it is enough to show that $\pi$ is non-crossing.

Suppose that $\pi$ is crossing, i.e., there exist four labels $a<b<c<d$ such 
that $a$ and $c$ are in the same block $\pi_{i}$ and $b$ and $d$ are in $\pi_{j}$
with $i\neq j$.
In terms of binary trees, labels in both $\pi_i$ and $\pi_j$ are connected by 
right edges in the tree. 
We consider two cases: $\pi_{i}$ is left to $\pi_{j}$, and vice versa.
First, assume that $\pi_{i}$ is left to 
$\pi_{j}$ in $L(T)$. Then, when we read the labels in $L(T)$ according to the 
order of the blocks, we first read the labels of $\pi_{i}$, then those of $\pi_{j}$.
By the crossing property, we read the label $c$ before $b$, which implies 
that the label $c$ is left to $b$.
Thus $L(T)$ is not canonical.
Secondly, assume that $\pi_{j}$ is left to $\pi_{i}$ in $L(T)$.
Then, if we read the labels in $L(T)$ according to the order of blocks,
we first read $b$ and $d$, then $a$ and $c$.
Since the label $d$ is left to $c$, $L(T)$ is not canonical.
In both cases, we have a contradiction since $L(T)$ is canonical.
Thus, $\psi(L(T))$ is a canonical non-crossing partition if $L(T)$ is canonical.	
\end{proof}

\begin{lemma}
\label{lemma:psibij}
The map $\psi$ is a bijection between $\mathtt{NCP}(n)$ and the set of 
canonical labeled trees.
\end{lemma}
\begin{proof}
From Proposition \ref{prop:cardNCP} and Proposition \ref{prop:clbt}, 
$|\mathtt{NCP}(n)|$ is equal to the number of canonical labeled trees with 
$n$ internal nodes.

Below, we will construct a bijection between them.
Let $\pi:=(\pi_{1},\ldots,\pi_{l})\in\mathtt{NCP}(n)$.
We assign a labeled binary tree $LT(\pi_{i})$ 
to each $\pi_{i}$, $1\le i\le l$ such that it consists of only right edges and their labels are in $\pi_{i}$.
We glue $LT(\pi_{i})$, $1\le i\le l$ together and construct a labeled binary tree 
in the following way.
Let $l_{i}$ be the label in $LT(\pi_{i})$ such that it is the largest label 
which is smaller than the minimum label in $LT(\pi_{i-1})$.
We glue $LT(\pi_{i-1})$ with $LT(\pi_{i})$ such that 
the root of $LT(\pi_{i-1})$ is connected to the node labeled $l_{i}$ in $LT(\pi_{i})$
by a left edge.
By repeating this process for all $i\in[l]$, we obtain a labeled tree $L(T)$.

Since $\pi\in\mathtt{NCP}(n)$ is canonical and non-crossing, 
the labels on the newly obtained tree $L(T)$ are increasing 
from the root to leaves. We will show that $L(T)$ is canonical.
Suppose that $L(T)$ is not canonical.
Then, there exists a pair of labels $n_1<n_2$ such that 
the node labeled $n_2$ is left to the node labeled $n_1$, and 
these two nodes are contained in different $LT(\pi_{i})$'s.
Let $LT(\pi_{p})$ (resp. $LT(\pi_{q})$) be the labeled tree which 
contains the node labeled $n_2$ (resp. $n_1$).
By the definition of gluing process, we have $p<q$,  
$\min(LT(\pi_{p}))>\min(LT(\pi_{q}))$, and $\min(LT(\pi_p))<n_1$.
Thus, we have 
\begin{align*}
\min(LT(\pi_q))<\min(LT(\pi_{p}))<n_1<n_2,
\end{align*} 
which implies that $\pi_p$ and $\pi_q$ are crossing.
This contradicts the fact that $\pi$ is non-crossing.
Thus, the labeled tree $L(T)$ is canonical.
\end{proof}

The composition map $\psi\circ\phi^{-1}$ is characterized as follows.
\begin{prop}
\label{prop:bij312NCP}
The map $\psi\circ\phi^{-1}$ is a bijection between $\mathtt{NCCP}(n;312)$ and 
$\mathtt{NCP}(n)$.
\end{prop}
\begin{proof}
The map $\phi$ is a bijection between $\mathtt{NCCP}(n;312)$ and the set of 
canonical labeled trees by Lemma \ref{lemma:312clbt}.
The map $\psi$ is a bijection between the set of canonical labeled trees 
and $\mathtt{NCP}(n)$ by Lemma \ref{lemma:psibij}.
Thus, the composition $\psi\circ\phi^{-1}$ is a bijection between 
$\mathtt{NCCP}(n;312)$ and $\mathtt{NCP}(n)$, 
which completes the proof.	
\end{proof}

\section{Lattice of noncommutative crossing partitions}
\label{sec:LatNCCP}
In this section, we introduce a new lattice whose nodes in the Hasse diagram are labeled 
noncommutative crossing partitions.
This lattice contains the Kreweras lattice as a sublattice.
We first define a cover relation $\lessdot$ on $\mathtt{NCCP}(n)$ by 
use of the notion of a rotation of a labeled binary tree.
By introducing another cover relation $\subset$ on $\mathtt{NCP}(n)$, 
we show that the poset $(\mathtt{NCP}(n),\subset)$ is equivalent to 
the Kreweras lattice of non-crossing partitions.
The subposet $(\mathtt{NCCP}(n;312),\le)$ is shown to be isomorphic 
to the Kreweras lattice. As a consequence, the lattice $(\mathtt{NCCP}(n),\le)$
contains the Kreweras lattice as a sublattice.

\subsection{A lattice on \texorpdfstring{$\mathtt{NCCP}(n)$}{NCCP(n)}}
\subsubsection{Cover relation on \texorpdfstring{$\mathtt{NCCP}(n)$}{NCCP(n)}}
\label{sec:covNCCP}
We define a cover relation $\lessdot$ on $\mathtt{NCCP}(n)$ 
by making use of labeled binary trees. 

Let $L(T)$ be a labeled binary tree of shape $T$.
We denote by $L(\mathfrak{n})$ the label of a node $\mathfrak{n}$. 
Given a pair of nodes $\mathfrak{n}_{1}$ and $\mathfrak{n}_2$,  
we denote by $\mathfrak{n}_1\rightarrow\mathfrak{n}_2$ if 
$\mathfrak{n}_1$ is a child node of $\mathfrak{n}_2$ in the labeled 
tree $L(T)$.

Since $T$ is a rooted binary tree, we have a unique sequence of 
nodes from $\mathfrak{n}$ to the root $\mathfrak{r}$, 
and denote it by 
\begin{align*}
p(\mathfrak{n}): \mathfrak{n}=\mathfrak{n}_0\rightarrow \mathfrak{n}_1\rightarrow
\ldots \rightarrow \mathfrak{n}_{k}=\mathfrak{r}.
\end{align*} 
We call $p(\mathfrak{n})$ a path from $\mathfrak{n}$ to $\mathfrak{r}$.
We write $\overline{p(\mathfrak{n})}:=p(\mathfrak{n})\setminus\{\mathfrak{n}\}$ 
as the path from $\mathfrak{n}_1$ to the root in $p(\mathfrak{n})$.

Below, we assume that the node $\mathfrak{n}$ is connected to the parent node 
by a left edge.
We define the set of nodes 
\begin{align*}
LS(\mathfrak{n}):=\{\mathfrak{n}_{i}\in p(\mathfrak{n})| 
\mathfrak{n}_{i-1} \text{ and } \mathfrak{n}_{i} \text{ is connected by a left edge} \},
\end{align*}
and call it a {\it left set} of the node $\mathfrak{n}$.

Let $\mathfrak{n}'$ be the node with label smaller than $L(\mathfrak{n})$.
Then, we have the following cases: 
\begin{enumerate}
\item The path $p(\mathfrak{n}')$ is contained in $p(\mathfrak{n})$, i.e., 
$\mathfrak{n}_{j}\in p(\mathfrak{n})$ for some $j\in[k]$ is $\mathfrak{n'}$.
\item The path $p(\mathfrak{n}')$ is said to be left (resp. right) to $p(\mathfrak{n})$, i.e., 
there exists a node $\mathfrak{n}_{c}$, $c\le k-1$, such that 
$p(\mathfrak{n}_{c})=p(\mathfrak{n})\cap p(\mathfrak{n'})$, 
the path $p(\mathfrak{n'})\setminus p(\mathfrak{n}_{c})$ is in the left (resp. right) subtree of $\mathfrak{n}_{c}$,
and path $p(\mathfrak{n})\setminus p(\mathfrak{n}_{c})$ is in the right (resp. left) subtree of $\mathfrak{n}_{c}$.
\end{enumerate}
We say that $p(\mathfrak{n'})$ is weakly left (resp. right) to $p(\mathfrak{n})$ if it 
is contained in or left (resp. right) to $p(\mathfrak{n})$. 

\begin{example}
We consider the labeled tree in Figure \ref{fig:bt}. 
The path from the node with label $7$ is given by $p_1:=7\rightarrow 4\rightarrow 3 \rightarrow 1$. 
The path $p_2:=8\rightarrow3\rightarrow1$ is left to $p_1$ and the path $p_3:=5\rightarrow2\rightarrow1$
is right to $p_1$.
The left set $LS(7)$ of the node with label $7$ is $LS(7)=\{4,1\}$.
\end{example}

\begin{defn}
\label{defn:lereseq}
We define two paths, which are sequences of nodes: 
\begin{enumerate}
\item 
Given a node $\mathfrak{n}$, we define the path $\overrightarrow{p}(\mathfrak{n})$ of nodes by 
\begin{align}
\label{eq:leseq}
\mathfrak{m}_{p}\rightarrow\mathfrak{m}_{p-1}\rightarrow\ldots \rightarrow
\mathfrak{m}_{0}=\mathfrak{n},
\end{align}
where two nodes $\mathfrak{m}_{i}$ and $\mathfrak{m}_{i-1}$ are connected by 
a left edge and $p$ is maximal.
We call $\overrightarrow{p}(\mathfrak{n})$ a left-extended sequence of $\mathfrak{n}$.
The label of $\mathfrak{m}_{i}$ is larger than that of $\mathfrak{m}_{i-1}$.
\item
Given a node $\mathfrak{m}$, we define $\overrightarrow{q}(\mathfrak{m})$ by the sequence of nodes 
\begin{align*}
\mathfrak{m}=\mathfrak{m'}_{0}\leftarrow \mathfrak{m'}_{1}\leftarrow \ldots \leftarrow \mathfrak{m'}_{q},
\end{align*}
where $\mathfrak{m'}_{i}$ and $\mathfrak{m'}_{i+1}$ are connected by a right edge in $L(T)$ and $q$ is maximal. 
We call $\overrightarrow{q}(\mathfrak{m})$ a right-extended sequence of $\mathfrak{m}$.
The label of $\mathfrak{m'}_{i}$ is larger than that of $\mathfrak{m'}_{i-1}$.
\end{enumerate}
\end{defn}

Note that, by definition, the paths $\overrightarrow{p}(\mathfrak{n})$ and $\overrightarrow{q}(\mathfrak{n})$ are 
unique if the node $\mathfrak{n}$ is given.

Choose a node $\mathfrak{n}_{a}$ in $LS(\mathfrak{n})$  
if $p(\mathfrak{n'})$ is weakly left to $p(\mathfrak{n})$ for all $L(\mathfrak{n'})<L(\mathfrak{n})$.
Similarly, suppose $p(\mathfrak{n'})$ is right to $p(\mathfrak{n})$ for some $\mathfrak{n'}$ in $L(T)$, where 
$L(\mathfrak{n'})<L(\mathfrak{n})$.
Among such $\mathfrak{n'}$, denote the left-most $\mathfrak{n'}$ by $\mathfrak{n'}_{l}$.
Choose a node $\mathfrak{n}_{a}$ in $LS(\mathfrak{n})\setminus \overline{p(\mathfrak{n}_c)}$, where 
$\mathfrak{n}_{c}$ is the node such that $\mathfrak{n}_{c}\in p(\mathfrak{n})\cap p(\mathfrak{n'}_{l})$
and its label is maximum.
Note that by definition of $\overline{p(\mathfrak{n}_c)}$, $\mathfrak{n_a}$ can be $\mathfrak{n}_{c}$.

Suppose that the left-extended sequence of $\mathfrak{n}$ consists of $p+1$ nodes as in 
Eq. (\ref{eq:leseq}). 
Let $S:=S(\mathfrak{n})\subseteq[1,p]$ be the subset and $\mathfrak{m}(S)$ be the set of nodes such that 
if and only if $i\in S$, then $\mathfrak{m}_{i}\in\mathfrak{m}(S)$.

We define a {\it rotation} of a labeled tree $L(T)$ as follows.
Fix a triplet $\nu:=(\mathfrak{n},\mathfrak{n}_{a},\mathfrak{m}(S))$ in $L(T)$.
A rotation is an operation on a labeled tree characterized by the triplet $\nu$.
We first construct two labeled trees $L(T_{1})$ and $L(T_{2})$.
Then, we construct a new labeled tree from $L(T_{1})$ and $L(T_{2})$, 
which is defined to be a rotation of $L(T)$.

We first construct two labeled binary trees $L(T_1)$ and $L(T_2)$ from $L(T)$ 
by use of the triplet $\nu$:
\begin{enumerate}
\item 
First, we construct a tree, whose root is $\mathfrak{n}$, from the set of nodes 
$\mathfrak{m}(S)\cup\{\mathfrak{n}\}$ and their right subtrees.  
We connect the nodes in $\mathfrak{m}(S)\cup\{\mathfrak{n}\}$ 
by left edges such that the labels are increasing from the root to leaves.  
We keep the right subtrees of each node in $\mathfrak{m}(S)\cup\{\mathfrak{n}\}$ as it is.
We denote by $L(T_{1})$ the newly obtained labeled binary tree.
\item 
We construct a labeled binary tree $L(T_2)$ by the following two steps:
\begin{enumerate}
\item
Recall that $\mathfrak{n}$ is connected to the parent edge $\mathfrak{n}_{0}$ by a left edge 
by assumption.
Let $\mathfrak{n'}\in \overrightarrow{p}(\mathfrak{n}_{0})\setminus(\mathfrak{m}(S)\cup\{\mathfrak{n}\})$ 
and $RT(\mathfrak{n'})$ be a subtree consists of the root $\mathfrak{n'}$ and its right subtree.
The root $\mathfrak{n'}$ in the tree $RT(\mathfrak{n'})$ does not have its left subtree.
We connect the roots of the trees $RT(\mathfrak{n'})$ by left edges such that we have an increasing 
labels from $\mathfrak{n}_{0}$ to leaves.
We denote the new labeled binary tree by $\widetilde{L}(T_2)$. The root of $\widetilde{L}(T_2)$ 
is the node $\mathfrak{n}_{0}$.

\item 
Let $ST(\mathfrak{n}_{0})$ be the subtree in $L(T)$ whose root is the node $\mathfrak{n}_{0}$.
Suppose that the node $\mathfrak{n}_{0}$ is connected to the parent node by a left or right edge.
We connect the tree $\widetilde{L}(T_2)$  and $L(T)\setminus ST(\mathfrak{n}_{0})$ 
at the node $\mathfrak{n}_{0}$ by a left or right edge.
Then, we denote the new binary labeled tree by $L(T_2)$.
If the node $\mathfrak{n}_{0}$ is the root of $L(T)$, we define 
$L(T_2):=\widetilde{L}(T_{2})$.
\end{enumerate}
\end{enumerate}

\begin{example}
We consider the the left-most tree $L(T)$ in Figure \ref{fig:LTles}. 
In the tree, the node $\mathfrak{n}_{0}$ is the parent node of the node $\mathfrak{n}$, and 
$U$ and $U_{i}$, $0\le i\le 3$, are right subtrees whose roots are $\mathfrak{n}$ and 
$\mathfrak{n}_{i}$ for $0\le i\le 3$. 
We construct two trees $L(T_1)$ and $L(T_2)$ from $L(T)$ by following the 
procedure given above.
Since we consider the left-extended sequence which contains the node $\mathfrak{n}$, the node 
$\mathfrak{n}_3$ has no left subtree.

\begin{figure}[ht]
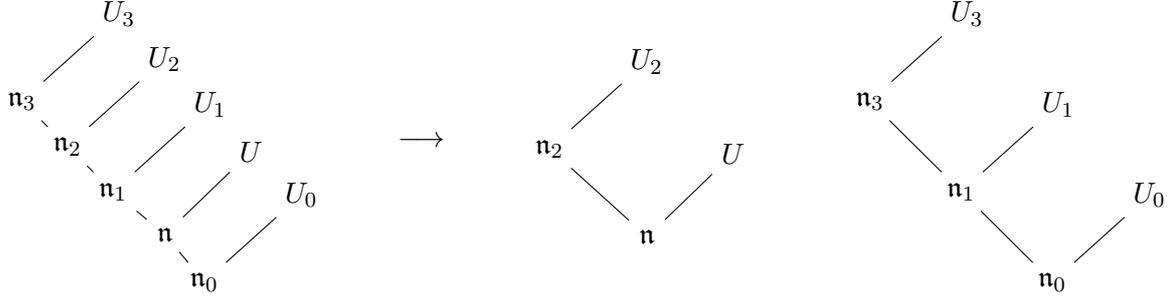

\tikzpic{-0.5}{[scale=0.6]
\node (root) at(0,0)[anchor=north east]{$\mathfrak{n}_{0}$};
\node (Tr)at(1,1)[anchor=south west]{$U_{0}$};
\node (n)at(-1,1)[anchor=north east]{$\mathfrak{n}$};
\node (Tn)at(0,2)[anchor=south west]{$U$};
\node (n1)at(-2,2)[anchor=north east]{$\mathfrak{n}_1$};
\node (T1)at(-1,3)[anchor=south west]{$U_1$};
\node (n2)at(-3,3)[anchor=north east]{$\mathfrak{n}_{2}$};
\node (T2)at(-2,4)[anchor=south west]{$U_2$};
\node (n3)at(-4,4)[anchor=north east]{$\mathfrak{n}_3$};
\node (T3)at(-3,5)[anchor=south west]{$U_3$};
\draw (Tr)--(root)--(n)--(n1)--(n2)--(n3)(n)--(Tn)(n1)--(T1)(n2)--(T2)(n3)--(T3);
}\qquad$\longrightarrow$\qquad
\tikzpic{-0.5}{[scale=0.6]
\node (n)at(-1,1)[anchor=north east]{$\mathfrak{n}$};
\node (Tn)at(0,2)[anchor=south west]{$U$};
\node (n2)at(-3,3)[anchor=north east]{$\mathfrak{n}_{2}$};
\node (T2)at(-2,4)[anchor=south west]{$U_2$};
\draw (n)--(n2)(n)--(Tn)(n2)--(T2);
}\qquad
\tikzpic{-0.5}{[scale=0.6]
\node (root) at(0,0)[anchor=north east]{$\mathfrak{n}_{0}$};
\node (Tr)at(1,1)[anchor=south west]{$U_{0}$};
\node (n1)at(-2,2)[anchor=north east]{$\mathfrak{n}_1$};
\node (T1)at(-1,3)[anchor=south west]{$U_1$};
\node (n3)at(-4,4)[anchor=north east]{$\mathfrak{n}_3$};
\node (T3)at(-3,5)[anchor=south west]{$U_3$};
\draw (Tr)--(root)--(n1)--(n3)(n1)--(T1)(n3)--(T3);
}

\caption{An example of constructing two labeled binary trees}
\label{fig:LTles}
\end{figure}

We consider the set of nodes $\mathfrak{m}(S)=\{\mathfrak{n}_2\}$.
The middle labeled binary tree is $L(T_1)$ and the right tree
is $L(T_2)$.
The tree $L(T_1)$ consists of two nodes $\mathfrak{n}$ and $\mathfrak{n}_2$, 
and two subtrees $U$ and $U_2$.
Similarly, the tree $L(T_2)$ consists of three nodes $\mathfrak{n}_0$, $\mathfrak{n}_1$
and $\mathfrak{n}_3$ and three subtrees $U_0$, $U_1$ and $U_3$.

Note that the connectivity of nodes is changed, but 
the trees $U$ and $U_{i}$, $0\le i\le 3$, remain as it was.
\end{example}

From $L(T_1)$ and $L(T_2)$ together with $\nu$, we construct a 
labeled binary tree $L(T')$ of $n$ internal nodes as follows.

Recall that the root of $L(T_1)$ has a label $L(\mathfrak{n})$.
Let $\overrightarrow{q}(\mathfrak{n})$ be a right-extended sequence of $\mathfrak{n}$ 
in $L(T_1)$.
Similarly, let $\overrightarrow{q}(\mathfrak{n}_a)$ be a right-extended sequence of 
$\mathfrak{n}_{a}$ in $L(T_2)$.

We insert the labeled tree $L(T_1)$ into $L(T_2)$ by the following procedures.
Since $L(\mathfrak{n}_a)<L(\mathfrak{n})$, one can construct a unique 
right-extended sequence of $\mathfrak{n}_a$  by sorting the nodes in 
$\overrightarrow{q}(\mathfrak{n})$ and $\overrightarrow{q}(\mathfrak{n}_a)$
in the increasing order.
We lengthen the right-extended sequence $\overrightarrow{q}(\mathfrak{n}_a)$ 
by use of $\overrightarrow{q}(\mathfrak{n})$.
We denote the new right-extended sequence by 
$\overrightarrow{q}(\mathfrak{n}\cup\mathfrak{n}_a)$.

If a node $\mathfrak{m}\in\overrightarrow{q}(\mathfrak{n}\cup\mathfrak{n}_a)$
has a left subtree $T_{L}(\mathfrak{m})$ in $L(T_1)$ or $L(T_2)$, 
we append $T_{L}(\mathfrak{m})$ to the node $\mathfrak{m}$ in $\overrightarrow{q}(\mathfrak{n}\cup\mathfrak{n}_a)$
by a left edge.
We denote by $L(T')$ the newly obtained tree.

\begin{defn}
The labeled binary tree $L(T')$ covers $L(T)$ if and only if $L(T')$ can be obtained from $L(T)$ by the above 
operation with the triplet $\nu$.
We say $\pi'$ is a rotation of $\pi$ if a labeled tree $L(T')$ for $\pi'$ is obtained from $L(T)$ for $\pi$ by 
the triplet $\nu$.
The partition $\pi'$ covers $\pi$, denoted by $\pi\lessdot\pi'$, if and only if $\pi'$ is a rotation of $\pi$. 
We write $\pi\le\pi'$ if and only if there exists an unrefinable sequence of partition 
$\pi=\pi_0\lessdot \pi_1\lessdot\ldots\lessdot\pi_{k}=\pi'$.
\end{defn}

\begin{example}
We consider the same labeled binary tree $L(T)$ as in Figure \ref{fig:bt}.
We take $\mathfrak{n}=7$ (which means the node with label $7$).
Since $LS(7)=\{1,4\}$ and $p(7)\cap p(5)=\{1\}$, $\mathfrak{n}_{a}$ is either $4$ or $1$.
We have two labeled binary trees $L(T')$ which covers $L(T)$ in the case of $\mathfrak{n}=7$:
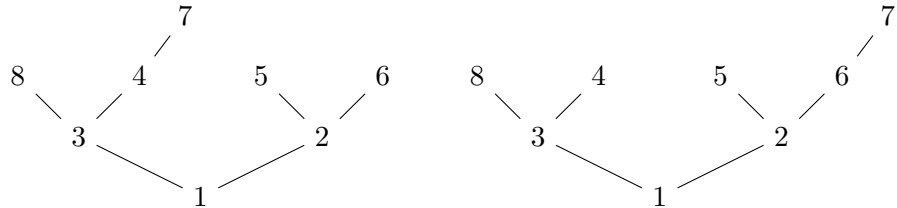
\begin{align*}
\begin{tikzpicture}
[scale=0.8,grow'=up,level distance=1cm, 
level 1/.style={sibling distance=4cm},
level 2/.style={sibling distance=2cm},
level 3/.style={sibling distance=1.5cm},
]
\node {$1$} child{ node {$3$}
		 child{node{$8$}}
		 child{node{$4$}
		 	child[missing]
		 	child{node{$7$}}
		 }
		}
	     child{node{$2$}
	     	child{node{$5$}}
	     	child{node{$6$}}};
\end{tikzpicture}
\qquad
\begin{tikzpicture}
[scale=0.8,grow'=up,level distance=1cm, 
level 1/.style={sibling distance=4cm},
level 2/.style={sibling distance=2cm},
level 3/.style={sibling distance=1.5cm},
]
\node {$1$} child{ node {$3$}
		 child{node{$8$}}
		 child{node{$4$}
		 }
		}
	     child{node{$2$}
	     	child{node{$5$}}
	     	child{node{$6$}
	     	child[missing]
	     	child{node{$7$}}}};
\end{tikzpicture}
\end{align*}
Note that the node labeled $7$ is connected to its parent node
by a right edge in the labeled binary tree $L(T')$.

Similarly, take $\mathfrak{n}=3$. Then, we have $\mathfrak{n}_{a}=1$.
If we choose $S=\emptyset$, i.e., $\mathfrak{m}(S)=\emptyset$, 
the labeled binary tree which covers $L(T)$ is 
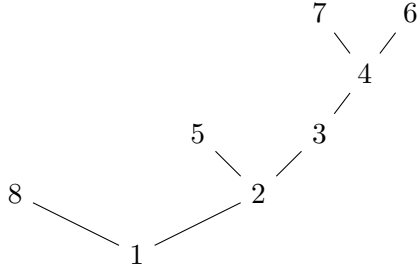
\begin{align*}
\begin{tikzpicture}
[scale=0.8,grow'=up,level distance=1cm, 
level 1/.style={sibling distance=4cm},
level 2/.style={sibling distance=2cm},
level 3/.style={sibling distance=1.5cm},
]
\node {$1$} child{ node {$8$}
		}
	     child{node{$2$}
	     	child{node{$5$}}
	     	child{node{$3$}
	     	child[missing]
	     	child{node{$4$}
	     		child{node{$7$}}
	     		child{node{$6$}}}}};
\end{tikzpicture}
\end{align*}

Finally, if we take $\mathfrak{n}=8$, the admissible $\mathfrak{n}_a$ is $3$.
This is because $LS(8)=\{3,1\}$, $p(7)$ is right to $p(8)$ and $p(8)\cap p(7)=3\rightarrow 1$. 
The partition which covers $L(T)$ is given by $37/48/15/26$.
\end{example}

\subsubsection{Lattice on \texorpdfstring{$\mathtt{NCCP}(n)$}{NCCP(n)}}

Let $(L,\le)$ be a partially ordered set (poset).
We define two binary operations $\vee$ and $\wedge$: 
$\vee,\wedge: L\times L\rightarrow L$.
Given two elements $x,y\in L$, an element $x\vee y$ is called a {\it join} (least upper bound)
and $x\wedge y$ a {\it meet} (greatest lower bound).
A poset $(L,\le)$ is called a {\it join-semilattice} (resp. {\it meet-semilattice})
if two elements in $L$ have a unique join (resp. meet).
The poset $(L,\le)$ is called a {\it lattice} if it is a join- and meet-semilattice.

To view the poset $(\mathtt{NCCP}(n),\le)$ as a lattice, 
we will define binary operations $\wedge$ and $\vee$ on the 
poset $(\mathtt{NCCP}(n,\le))$, that is, 
$\wedge,\vee: \mathtt{NCCP}(n)\times\mathtt{NCCP}(n)\rightarrow\mathtt{NCCP}(n)$.
Given two elements in $\mathtt{NCCP}(n)$, we will define a join $a\vee b$ and 
a meet $a\wedge b$ by use of the rotation of a labeled binary tree 
introduced in Section \ref{sec:covNCCP}.

Given an element $\pi\in\mathtt{NCCP}(n)$ in one-line notation, we define two sets for each $i\in[n-1]$ by 
\begin{align*}
\mathfrak{R}(i;\pi)&:=\{j>i | j \text{ is right to } i \text{ in } \pi \}, \\
\mathfrak{L}(i;\pi)&:=\{j>i | j \text{ is left to } i \text{ in } \pi \}.
\end{align*}

By construction, we have 
\begin{align*}
\mathfrak{R}(i;\pi)\cup\mathfrak{L}(i;\pi)=[n]\setminus[i],
\end{align*}
for all $i\in[n-1]$.

The following lemma is a key to construct a join and a meet.
\begin{lemma}
\label{lemma:lbtLR}
There is a bijection between a labeled binary tree of $n$ internal nodes 
and a collection of the sets $\{(\mathfrak{L}(i;\pi),\mathfrak{R}(i;\pi))\}_{1\le i\le n-1}$.
\end{lemma}
\begin{proof}
When $n=1$, we have a labeled binary tree with one internal node labeled one.
We assume that $n\ge2$.
We construct a labeled binary tree from the sets $\{(\mathfrak{L}(i;\pi),\mathfrak{R}(i;\pi))\}_{1\le i\le n-1}$.
Let $r(i):=\min(\mathfrak{R}(i;\pi))$ and $l(i):=\min(\mathfrak{L}(i;\pi))$.
We connect a node labeled one and a node labeled $r(1)$ (resp. $l_1$) by a right (resp. $l(1)$) edge.
Then, we connect a node labeled $r(1)$ and a node labeled $r(r(1))$ 
(resp. $\min(\mathfrak{L}(r(1);\pi)\setminus\mathfrak{L}(1;\pi))$) 
by a right (resp. left) edge.
Similarly, we connect a node labeled $l(1)$ and a node labeled $l(l(1))$ 
(resp. $\min(\mathfrak{R}(l(1);\pi)\setminus\mathfrak{R}(1;\pi))$) by a left (resp. right) edge.
We continue this process until we add $n-1$ nodes one-by-one to the labeled tree, and 
obtain a labeled binary tree.

Conversely, given a labeled binary tree $L(T)$, by reading the labels of $L(T)$ by the in-order,
we recover the collection of the sets  $\{(\mathfrak{L}(i;\pi),\mathfrak{R}(i;\pi))\}_{1\le i\le n-1}$ 
by following the definition of the sets.

From these observations, we have a bijection between a labeled tree and the sets. 
\end{proof}

Given two element $\pi,\pi'\in\mathtt{NCCP}(n)$, we define 
a join $\pi\vee\pi'$ and a meet $\pi\wedge\pi'$ by use of the rotation 
as follows.
 
\begin{defn}
\label{defn:joinmeet}
Given two elements $\pi_{1},\pi_{2}\in\mathtt{NCCP}(n)$, a join $\pi^{\vee}:=\pi_1\vee\pi_2$ and 
a meet $\pi^{\wedge}:=\pi_1\wedge\pi_2$ 
are defined as follows.
We define $\pi^{\vee}=\pi$ and $\pi^{\wedge}=\pi$ if $\pi=\pi_1=\pi_2$.
For $\pi_1\neq\pi_2$, the join $\pi^{\vee}$ is recursively given: 
\begin{enumerate}
\item We define 
$\mathfrak{R}_{\cup}(i):=\mathfrak{R}(i;\pi_1)\cup\mathfrak{R}(i;\pi_2)$.

\item We construct $\pi'_{i}$, $i=1,2$, recursively:
\begin{enumerate}
\item Set $p=n-1$.
\item Suppose that $\mathfrak{R}(p;\pi_{i})\neq \mathfrak{R}_{\cup}(p)$ 
and $q\in\mathfrak{R}_{\cup}(p)\setminus\mathfrak{R}(p;\pi_{i})$.

We perform rotations to obtain a new partition $\pi'_{i}$.
In the partition $\pi'_{i}$, the integer $q$ is right to $p$, 
i.e., $q\in\mathfrak{R}(p;\pi'_{i})$, and $|\mathfrak{R}(p;\pi'_{i})|$ 
is minimum. 
These two conditions give a unique partition from $\pi_{i}$.

\item Decrease $p$ by one. Continue the same process as (b) until the new partition 
$\pi'_1$ satisfies $\mathfrak{R}_{\cup}(i)\subseteq\mathfrak{R}(i;\pi'_{i})$ for 
all $i\in[n-1]$.
\end{enumerate}
\item We define $\pi^{\vee}:=\pi'_{1}\vee\pi'_{2}$.
\end{enumerate}
Similarly, the meet $\pi^{\wedge}$ is also recursively given as in the case of a join.
We replace $\mathfrak{R}(p;\pi_{i})$ and $\mathfrak{R}_{\cup}(i)$ in (1) and (2) 
by $\mathfrak{L}(p;\pi_{i})$ and $\mathfrak{L}_{\cup}(i):=\mathfrak{L}(i;\pi_1)\cup\mathfrak{L}(i;\pi_{2})$ 
respectively.  
Further, we replace a rotation of $\pi_{i}$ in (2-b) by an inverse rotation of $\pi_{i}$, and 
define $\pi^{\wedge}:=\pi'_{1}\wedge\pi'_{2}$ in (3).
\end{defn}

\begin{example}
Take $\pi_1=23/14$ and $\pi_{2}=4/23/1$.
We have 
\begin{align*}
&\mathfrak{R}(1;\pi_1)=\{4\},\quad \mathfrak{R}(2;\pi_1)=\{3,4\},\quad \mathfrak{R}(3;\pi_1)=\{4\},  \\
&\mathfrak{R}(1;\pi_2)=\emptyset,\quad \mathfrak{R}(2;\pi_2)=\{3\},\quad \mathfrak{R}(3;\pi_2)=\emptyset,  
\end{align*}
and $\mathfrak{R}_{\cup}(i)=\mathfrak{R}(i;\pi_1)$ for $1\le i\le 3$.
From this, we have $\pi'_1=\pi_{1}$.
We have $\mathfrak{R}(3;\pi_2)=\emptyset\subset\mathfrak{R}_{\cup}(3)$. 
Since the node labeled $3$ is right to the node labeled by $4$ in $\pi_{2}$, we have to rotate 
$\pi_{2}$ such that $4$ is right to $3$. 
Then, we obtain $\pi''_{2}:=234/1$ by a rotation on $\pi_2$. 
We have $\mathfrak{R}_{\cup}(2)=\mathfrak{R}(2;\pi''_{2})$ and 
$\mathfrak{R}(1;\pi''_{2})=\emptyset\subset\mathfrak{R}_{\cup}(1)$.
We rotate $\pi''_{2}$ such that $2$ is right to $1$.
By a rotation, we obtain a new partition $\pi'_{2}=1234$.
Thus, the join is given by 
\begin{align*}
23/14\vee 4/23/1=23/14\vee 1234 =1234.
\end{align*}
For the meet, we have 
\begin{align*}
23/14\wedge 4/23/1=24/3/1\wedge 4/23/1=4/3/2/1.
\end{align*}

Similarly, the join and the meet of $23/14$ and $24/3/1$ is given by 
\begin{align*}
&23/14\vee 24/3/1=23/14\vee 23/14=23/14, \\
&23/14\wedge 24/3/1=24/3/1\wedge 24/3/1=24/3/1.
\end{align*}
\end{example}

\begin{prop}
The poset $(\mathtt{NCCP}(n),\le)$ is a lattice.
\end{prop}
\begin{proof}
To show the poset $(\mathtt{NCCP}(n),\le)$ is a lattice, 
we have to show that there exists a unique join and meet 
for given two elements $\pi_1$ and $\pi_2$ in $\mathtt{NCCP}(n)$.

From Definition \ref{defn:joinmeet}, we have a unique 
set $\mathfrak{R}_{\cup}(p)$ constructed from $\mathfrak{R}(p;\pi_{i})$, $i=1,2$.
Further, we construct $\mathfrak{R}(p;\pi'_{i})$ as a minimum set which satisfies 
$\mathfrak{R}_{\cup}(p)\subseteq\mathfrak{R}(p;\pi'_{i})$.
From Lemma \ref{lemma:lbtLR}, we have a unique $\pi'_{1}$ and $\pi'_{2}$ given 
$\pi_1$ and $\pi_{2}$.
Note that we have $\mathfrak{R}(p,\pi)=[n]\setminus[p]$ if $\pi=123\ldots n$ and 
this set is maximal with respect to $\subseteq$.
Thus, $\pi_1\vee\pi_2=\pi'_{1}\vee\pi'_{2}\le 123\ldots n$ and the binary operation 
$\vee$ gives a unique element.

By a similar argument, one can show that there exists a unique meet for given two 
elements $\pi_1$ and $\pi_2$.
\end{proof}

\begin{lemma}
\label{lemma:01}
The maximum (resp. minimum) element in the lattice $\mathcal{L}_{NCCP}$
is $\hat{1}=12\ldots n$ (resp. $\hat{0}=n/n-1/\ldots/1$).
\end{lemma}
\begin{proof}
It is obvious from the definition of the cover relation on $\mathtt{NCCP}(n)$. 
\end{proof}

Lemma \ref{lemma:01} implies the lattice is bounded.
To view the lattice $(\mathtt{NCCP}(n),\le)$ as a graded lattice,
we need to introduce a rank function $\rho$.

\begin{defn}
We define the function $\rho:\mathtt{NCCP}(n)\rightarrow\mathbb{N}$ 
as $\rho(\pi)=n-l$ where $l$ is the length of $\pi$, or equivalently, 
$l$ is the number of blocks in $\pi$. 
\end{defn}

\begin{prop}
\label{prop:rank}
The function $\rho$ is a rank function on the lattice $(\mathtt{NCCP}(n),\le)$.
\end{prop}
\begin{proof}
Recall that we have a labeled binary tree $L(\pi)$ associated to an element $\pi$ 
in $\mathtt{NCCP}(n)$.
Suppose $\pi_2$ covers $\pi_1$. 
Then, the number of left edges in $L(\pi_2)$ is one less than that of 
left edges in $L(\pi_1)$.
In terms of the length, the length of $\pi_2$ is one less than that of $\pi_1$,
that is, $l(\pi_2)=l(\pi_1)-1$.
By definition of $\rho$, we have $\rho(\pi_2)=\rho(\pi_1)+1$, which 
implies that the function $\rho$ is consistent with the cover relation.
It is obvious that the function $\rho$ is compatible with the ordering in $\mathtt{NCCP}(n)$,
i.e., $\rho(\pi_1)<\rho(\pi_2)$ if $\pi_1<\pi_2$.
From these, the function $\rho$ is  a rank function.
\end{proof}

\begin{remark}
Given a partition $\pi\in\mathtt{NCCP}(n)$, the rank $\rho(\pi)$ is 
equal to the number of right edges in the labeled binary tree corresponding 
to $\pi$.
This is because the number of right edges in a labeled binary tree
is increased by one by a rotation of $\pi$ by definition.
\end{remark}

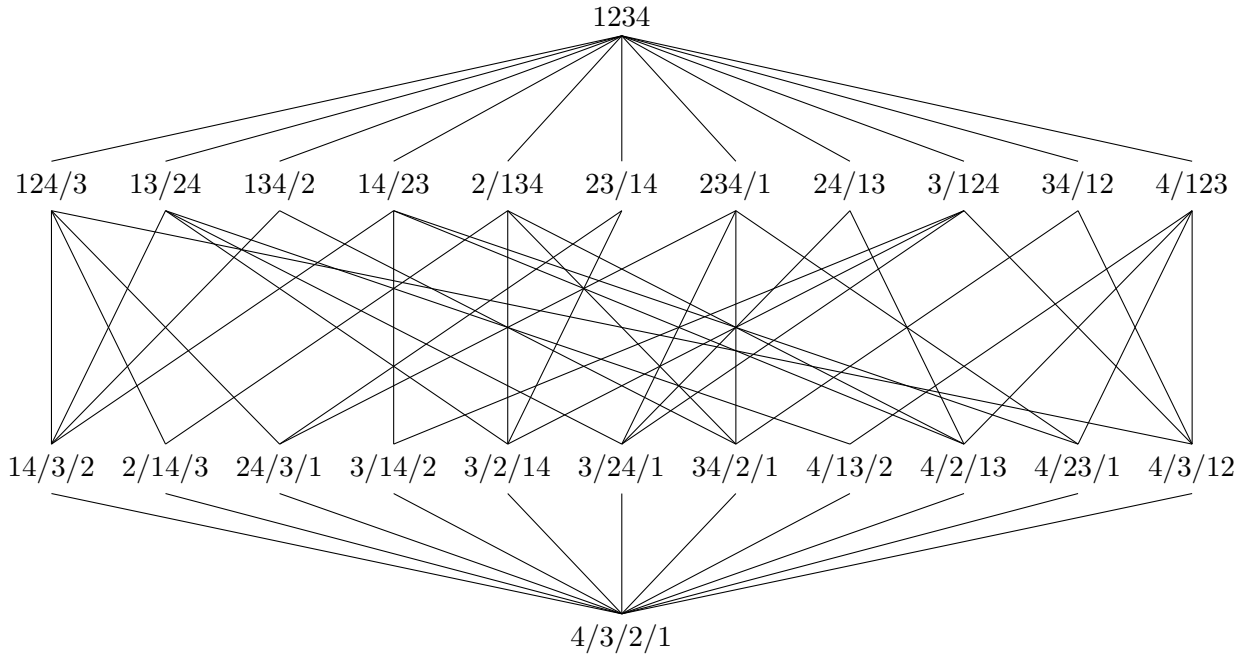
\begin{figure}[ht]
\begin{tikzpicture}[scale=1.5]
\node (max) at (0,4) {$1234$};
\node (1243) at (-5,2.5) {$124/3$};
\node (1324) at (-4,2.5){$13/24$};
\node (1342) at (-3,2.5){$134/2$};
\node (1423) at (-2,2.5){$14/23$}; 
\node (2134) at (-1,2.5){$2/134$};
\node (2314) at (0,2.5){$23/14$};
\node (2341) at (1,2.5){$234/1$};
\node (2413) at (2,2.5){$24/13$};
\node (3124) at (3,2.5){$3/124$};
\node (3412) at (4,2.5){$34/12$};
\node (4123) at (5,2.5){$4/123$};
\node (1432) at (-5,0){$14/3/2$};
\node (2143) at (-4,0){$2/14/3$};
\node (2431) at (-3,0){$24/3/1$};
\node (3142) at (-2,0){$3/14/2$};
\node (3214) at (-1,0){$3/2/14$};
\node (3241) at (0,0){$3/24/1$};
\node (3421) at (1,-0){$34/2/1$};
\node (4132) at (2,0){$4/13/2$};
\node (4213) at (3,0){$4/2/13$}; 
\node (4231) at (4,0){$4/23/1$};
\node (4312) at (5,0){$4/3/12$};
\node (min) at (0,-1.5) {$4/3/2/1$};
\foreach \a in {1432,2143,2431,3142,3214,3241,3421,4132,4213,4231,4312}
\draw (min.north)--(\a.south);
\foreach \a in {1243,1324,1342,1423,2134,2314,2341,2413,3124,3412,4123}
\draw (max.south)--(\a.north);
\foreach \a in {1342,1324,1243,1423}
\draw (1432.north)--(\a.south);
\foreach \a in {1243,2134}
\draw (2143.north)--(\a.south);
\foreach \a in {1243,2341,2314}
\draw (2431.north)--(\a.south);
\foreach \a in {1423,3124}
\draw (3142.north)--(\a.south);
\foreach \a in {1324,2314,2134,3124}
\draw (3214.north)--(\a.south);
\foreach \a in {3124,2341,2413,1324}
\draw (3241.north)--(\a.south);
\foreach \a in {3412,2341,2134,1342}
\draw (3421.north)--(\a.south);
\foreach \a in {4123,1324}
\draw (4132.north)--(\a.south);
\foreach \a in {4123,2413,2134,1423}
\draw (4213.north)--(\a.south);
\foreach \a in {4123,2341,1423}
\draw (4231.north)--(\a.south);
\foreach \a in {4123,3412,3124,1243}
\draw (4312.north)--(\a.south);
\end{tikzpicture}
\caption{The lattice $\mathcal{L}_{NCCP}(4)$.}
\label{fig:latNCCP}
\end{figure}

\begin{cor}
\label{cor:NCCPgbl}
The lattice $(\mathtt{NCCP}(n),\le)$ is a graded 
bounded lattice.
\end{cor}
\begin{proof}
The lattice is graded and bounded from Lemma \ref{lemma:01} and Proposition \ref{prop:rank}.
\end{proof}

\begin{defn}
We denote by $\mathcal{L}_{NCCP}(n)$ the graded bounded lattice $(\mathtt{NCCP}(n),\le)$. 
\end{defn}

Figure \ref{fig:latNCCP} shows the lattice $\mathcal{L}_{NCCP}(4)$. 
An edge in the Hasse diagram corresponds to the cover relation.

\subsection{Kreweras lattice}
We first characterize the set $\mathtt{NCP}(n)$ by introducing a cover relation $\subset$.
We will see that the poset $(\mathtt{NCP}(n),\subset)$ is equivalent to the Kreweras lattice 
of non-crossing partitions. To characterize the Kreweras lattice as a sublattice of the lattice $(\mathtt{NCCP}(n),\le)$, 
we study the subposet $(\mathtt{NCCP}(n;312),\le)$, and show that this poset is 
isomorphic to the Kreweras lattice.

Let $\pi:=(\pi_{1},\ldots,\pi_{l})$ and $\pi':=(\pi'_{1},\ldots,\pi'_{m})$ with 
$l\le m$ be noncommutative crossing partitions in $\mathtt{NCCP}(n)$.
Given two crossing blocks $\pi_{i}$ and $\pi_{i'}$, 
we write $\pi_{i}\cup \pi_{i'}$ as a unique crossing partition formed by 
integers in $\pi_{i}$ and $\pi_{i'}$.
This process corresponds to coarsement of blocks since we have obviously  
$\pi_{i},\pi_{i'}\subset \pi_{i}\cup\pi_{i'}$.
For example, we write $12467=146\cup 27$ in one-line notation, and 
$146/27$ and $27/146$ are refinements of $12467$. 

\begin{defn}
\label{defn:refpi}
Suppose that 
\begin{align*}
\pi_{i}=\bigcup_{j\in J(i)} \pi'_{j}, \quad \forall i\in[l],
\end{align*}
where $J(i)\neq\emptyset$ is a subset in $[m]$ such that 
$\bigsqcup_{i}J(i)=[m]$ and at least one $J(i)$ satisfies $|J(i)|\ge2$. 
We say that $\pi'$ is a refinement of $\pi$, and equivalently 
$\pi$ is a coarsement of $\pi'$.
\end{defn}

Note that in Definition \ref{defn:refpi}, the order of blocks 
in $\pi\in\mathtt{NCCP}(n)$ is irrelevant to the definition.
For example,  $36/5/24/1$ is a refinement of $\pi=15/24/36$.
Similarly, $24/36/5/1$ is also a refinement of $\pi$.
However, $6/5/4/13/2$ is not a refinement of $\pi$ since $\pi$ does not 
have a block containing $13$.
Similarly, $24/15/36$ is not a refinement of $\pi$ since it can be obtained 
from $\pi$ by rearranging the order of blocks of $\pi$.

\begin{defn}
Let $\pi,\pi'\in\mathtt{NCP}(n)$ and $l(\pi)$ be the number of blocks in $\pi$.
We say $\pi\subset\pi'$ if and only if 
$\pi$ is a refinement of $\pi'$ and $\rho(\pi)=\rho(\pi')+1$ where $\rho$ 
is the rank function.
\end{defn}

\begin{remark}
A cover relation $\pi\subset\pi'$ can be defined on $\mathtt{NCP}(n)$.
Take $\pi=4/2/13$. Definition \ref{defn:refpi} implies that 
$24/13$ is a coarsement of $\pi$. However, $24/13$ is a crossing 
partition and not in $\mathtt{NCP}(n)$.
Thus, the coarsements of $\pi$ in $\mathtt{NCP}(n)$ are 
$4/123$, $2/134$ and $1234$.
\end{remark}

It is a routine to show the following lemma.
\begin{lemma}
\label{lemma:spo}
The binary relation $\subset$ on $\mathtt{NCP}(n)$ is 
irreflexive, antisymmetric, and transitive.
\end{lemma}
From Lemma \ref{lemma:spo}, the relation $\subset$ on $\mathtt{NCP}(n)$ is 
a strict partial order.
By a natural correspondence between a strict partial order and a non-strict partial 
order, we obtain a non-strict partial order whose relation is denoted by $\subseteq$.
In fact, if we take the reflexive closure of the strict partial order $\subset$,
that is, 
\begin{align*}
\pi\subseteq \pi' \text{ if } \pi\subset \pi' \text{ or } \pi=\pi',
\end{align*}
we obtain the non-strict partial order on $\mathtt{NCCP}(n)$.

Further, this poset $(\mathtt{NCP}(n),\subseteq)$ has a lattice structure. 
The following proposition is obvious from the definition of 
the poset $(\mathtt{NCP}(n),\subseteq)$.
\begin{prop}
The poset $(\mathtt{NCP}(n),\subseteq)$ is a bounded lattice.
The greatest element $\hat{1}$ corresponds to $12\ldots n$ and the least element $\hat{0}$ is $n/n-1/\ldots/2/1$.
\end{prop}

\begin{remark}
Two remarks are in order.
\begin{enumerate}
\item The lattice $(\mathtt{NCP}(n),\subseteq)$ is the Kreweras lattice studied in \cite{Kre72}.
\item The Kreweras lattice is a graded poset. 
One can define a rank function $\rho:\mathtt{NCP}(n)\rightarrow \mathbb{N}$ of the poset 
as $n-l$ where $l$ is the number of blocks in a non-crossing canonical partition.
This rank function is the same as the one for $(\mathtt{NCCP}(n),\le)$.
\end{enumerate} 
\end{remark}

Recall that the set $\mathtt{NCCP}(n;312)$ of $312$-avoiding crossing partitions is 
a subset in $\mathtt{NCCP}(n)$.

\begin{lemma}
\label{lemma:lat312}
The poset $(\mathtt{NCCP}(n;312),\le)$ is a graded bounded lattice with the greatest element 
$\hat{1}=123\ldots n$ and the minimum element $\hat{0}=n/n-1/\ldots /1$.
\end{lemma}
\begin{proof}
Since $\hat{0},\hat{1}\in\mathtt{NCCP}(n;312)$, and $\hat{0}$ and $\hat{1}$ are 
the greatest and the minimum element in $\mathtt{NCCP}(n)$,
they are the greatest and the minimum element in $\mathtt{NCCP}(n;312)$.

We first show that given $\hat{0}\neq\pi'\in\mathtt{NCCP}(n;312)$, 
there exits at least one $\pi\in\mathtt{NCCP}(n;312)$ which is covered by $\pi'$.
From Lemma \ref{lemma:312clbt}, the labeled binary tree $L(T')$ corresponding to 
$\pi'$ is canonical.
Let $RL(T')$ be the set of labels such that if $i\in RL(T')$ then the node labeled $i$
is connected to the parent node by a right edge in $L(T')$.
Since $\pi'\neq\hat{0}$, we have at least one element in $RL(T')$.
Fix $l\in RL(T')$.
Since $L(T')$ is canonical, the node labeled $l-1$ is weakly left to the node labeled 
$l$.
We denote by $L(T';l)$ the subtree in $L(T')$ whose root is the node labeled $l$.
We separate $L(T';l)$ from $L(T')$, where the remaining tree is denoted by $L(T')\setminus L(T';l)$.
We reconnect $L(T';l)$ to $L(T')\setminus L(T';l)$ at the node labeled $l-1$ by a left edge.
Then, the newly obtained tree $L(T)$ is again canonical.
By Lemma \ref{lemma:312clbt}, we have an element $\pi\in\mathtt{NCCP}(n;312)$ corresponding to 
$L(T)$.
By construction, $\pi'$ covers $\pi$. 

Further, if we restrict the rank function $\rho$ on $\mathtt{NCCP}(n)$ to $\mathtt{NCCP}(n;312)$,
the function $\rho$ is a rank function on $\mathtt{NCCP}(n;312)$.
Similarly, by restricting the construction of a join and a meet on $\mathtt{NCCP}(n)$ to
$\mathtt{NCCP}(n;312)$, we have a unique join and a meet on the poset $\mathtt{NCCP}(n;312)$.

From these observations, the poset $(\mathtt{NCCP}(n;312),\le)$ is the graded bounded lattice 
and has the maximum and minimum element.
\end{proof}

The following proposition shows the relation between the lattice $(\mathtt{NCCP}(n;312),\le)$
and the Kreweras lattice. The former lattice is characterized by the pattern avoidance and 
the latter is by being non-crossing and canonical. 
Thus, to avoid the pattern $312$ can be regarded to be non-crossing and canonical. 
This identification will also appear when we consider $k-1$-chains in the lattice 
and its corresponding labeled $k$-ary trees (see Section \ref{sec:int}).

\begin{prop}
\label{prop:312Kre}
The lattice $(\mathtt{NCCP}(n;312),\le)$ is isomorphic to 
Kreweras lattice $(\mathtt{NCP}(n),\subseteq)$.
\end{prop}

Before proceeding the proof of Proposition \ref{prop:312Kre}, we study 
the relation between the two cover relations $\lessdot$ and $\subset$.

\begin{lemma}
\label{lemma:lesub}
The cover relation $\le$ on $\mathtt{NCCP}(n;312)$ is 
compatible with the cover relation $\subseteq$ on $\mathtt{NCP}(n)$.
\end{lemma}
\begin{proof}
Let $\pi,\pi'\in\mathtt{NCCP}(n;312)$ such that $\pi'$ covers $\pi$. 
As in the proof of Lemma \ref{lemma:lat312},
when we obtain $\pi'$ from $\pi$ by the rotation on the labeled tree $L(T)$,
we separate the subtree $L(T;l)$ from the whole labeled tree.
Then, we reconnect $L(T;l)$ to the labeled tree $L(T)\setminus L(T;l)$ by the rotation, 
and obtain the new tree $L(T')$.
It is easy to see that $\psi(L(T))$ is the refinement of $\psi(L(T'))$.
Note that from Lemma \ref{lemma:psiL}, $\psi(L(T'))$ is in $\mathtt{NCP}(n)$.

Therefore, the cover relation $\le$ is compatible with $\subseteq$, which 
completes the proof.
\end{proof}

\begin{proof}[Proof of Proposition \ref{prop:312Kre}]
We have a bijection $\psi\circ\phi^{-1}$ between $\mathtt{NCCP}(n;312)$ and $\mathtt{NCP}(n)$
from Proposition \ref{prop:bij312NCP}.
From Lemma \ref{lemma:lesub}, the cover relations $\le$ and $\subseteq$ are compatible 
with each other.
Thus, we have an isomorphism between $(\mathtt{NCCP}(n;312),\le)$ and $(\mathtt{NCP}(n),\subseteq)$, 
which completes the proof.
\end{proof}

By restricting ourselves to $\mathtt{NCCP}(n;312)$, we have the sublattice 
$(\mathtt{NCCP}(n;312),\le)$ in the poset $(\mathtt{NCCP}(n),\le)$.
The following proposition is a direct consequence of the observation above.
\begin{prop}
The poset $(\mathtt{NCP}(n),\subseteq)$ is a sublattice in the lattice $(\mathtt{NCCP}(n),\le)$.
\end{prop}

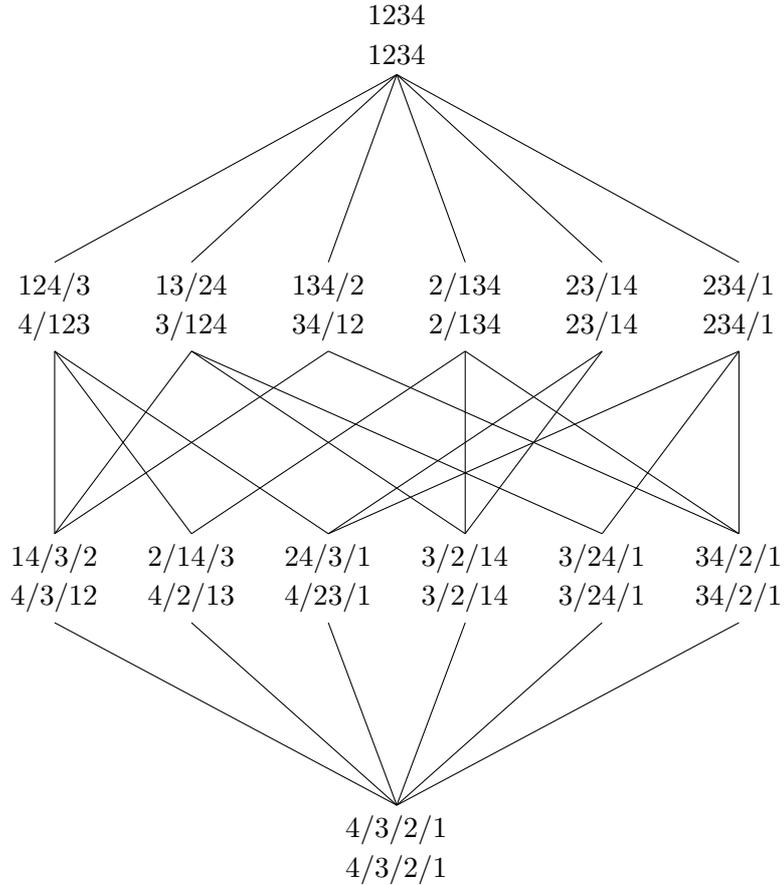
\begin{figure}[ht]	
\begin{tikzpicture}[scale=0.9]
\node (max) at (0,6) {$\displaystyle\genfrac{}{}{0pt}{}{1234}{1234}$};
\node (11) at (-3,2) {$\displaystyle\genfrac{}{}{0pt}{}{13/24}{3/124}$};
\node (12) at (-1,2) {$\displaystyle\genfrac{}{}{0pt}{}{134/2}{34/12}$};
\node (13) at (1,2) {$\displaystyle\genfrac{}{}{0pt}{}{2/134}{2/134}$};
\node (14) at (3,2) {$\displaystyle\genfrac{}{}{0pt}{}{23/14}{23/14}$};
\node (15) at (5,2) {$\displaystyle\genfrac{}{}{0pt}{}{234/1}{234/1}$};
\node (16) at (-5,2) {$\displaystyle\genfrac{}{}{0pt}{}{124/3}{4/123}$};
\node (21) at (-5,-2) {$\displaystyle\genfrac{}{}{0pt}{}{14/3/2}{4/3/12}$};
\node (22) at (-3,-2) {$\displaystyle\genfrac{}{}{0pt}{}{2/14/3}{4/2/13}$};
\node (23) at (-1,-2) {$\displaystyle\genfrac{}{}{0pt}{}{24/3/1}{4/23/1}$};
\node (24) at (1,-2) {$\displaystyle\genfrac{}{}{0pt}{}{3/2/14}{3/2/14}$};
\node (25) at (3,-2) {$\displaystyle\genfrac{}{}{0pt}{}{3/24/1}{3/24/1}$};
\node (26) at (5,-2) {$\displaystyle\genfrac{}{}{0pt}{}{34/2/1}{34/2/1}$};
\node (min) at (0,-6) {$\displaystyle\genfrac{}{}{0pt}{}{4/3/2/1}{4/3/2/1}$}; 
\foreach \a in {11,12,13,14,15,16}
\draw (max.south)--(\a.north);
\foreach \a in {21,22,23,24,25,26}
\draw (min.north)--(\a.south);
\foreach \a in {21,24,25}
\draw (11.south)--(\a.north);
\foreach \a in {21,26}
\draw (12.south)--(\a.north);
\foreach \a in {22,24,26}
\draw (13.south)--(\a.north);
\foreach \a in {23,24}
\draw (14.south)--(\a.north);
\foreach \a in {23,25,26}
\draw(15.south)--(\a.north);
\foreach \a in {21,22,23}
\draw(16.south)--(\a.north);
\end{tikzpicture}
\caption{Kreweras lattice. The top row in each node is an element in $\mathtt{NCCP}(4;312)$ and the bottom row
is the corresponding element in $\mathtt{NCP}(4)$.}
\label{fig:Hasse1}
\end{figure}

\begin{defn}
We denote by $\mathcal{L}_{\mathrm{Kre}}(n)$ the sublattice $(\mathtt{NCP}(n),\subseteq)$, or equivalently
$(\mathtt{NCCP}(n;312),\le)$ in $\mathcal{L}_{NCCP}$.
\end{defn}

Figure \ref{fig:Hasse1} shows the two lattices $(\mathtt{NCP}(4),\subseteq)$ and $(\mathtt{NCCP}(4;312),\le)$
in one Hasse diagram.
Note that some partitions in the two lattices coincide with each other, but not in general.

The set $\mathtt{NCP}(n)$ can be characterized by $\mathtt{NCCP}(n;312)$.
We have another characterization of $\mathtt{NCP}(n)$ by $132$-avoiding permutations.
Recall that a noncommutative crossing partition is bijective to 
a permutation. Let $\sigma$ be a $312$-avoiding permutation.
Then, $\sigma^{-1}$ is a $132$-avoiding permutation.
Since we have a correspondence between $\pi\in\mathtt{NCP}(n)$ and 
an element $\sigma\in\mathtt{NCCP}(n;312)$, it is natural 
to study a correspondence between $\pi$ and $\sigma^{-1}$.

\begin{prop}
\label{prop:NCP132}
The set $\mathtt{NCP}(n)$ is equivalent to the set $\mathtt{NCCP}(n;132)$ 
of $132$-avoiding crossing partitions.
\end{prop}
\begin{proof}
We first show that $\mathtt{NCCP}(n;132)\supseteq\mathtt{NCP}(n)$.
Suppose $\pi:=(\pi_1,\ldots,\pi_{l})\in\mathtt{NCP}(n)$ and 
$\pi$ has the pattern $132$.
Given a triplet $(w_1,w_2,w_3)$ with $w_1<w_2<w_3$, 
we assume that $w_1,w_2$ and $w_3$ form a pattern $132$ and 
$w_1\in\pi_{i}$ and $w_2\in\pi_{j}$ with $i<j$. 
Since $\pi$ is a canonical partition, $\pi_{j}$ contains 
an integer $w_0$ where $w_0<w_1$.
Since each $\pi_{p}$ is increasing from left to right,
$\pi_{i}$ contains $w_3$ or an integer $w_4$ which is greater than $w_3$.
From these observations, it is easy to see that $\pi$ is a crossing partition.
However, $\mathtt{NCP}(n)$ is the set of non-crossing partitions by definition,
and we have a contradiction. Thus, $\pi$ avoids the pattern $132$, which 
leads to $\mathtt{NCCP}(n;132)\supseteq\mathtt{NCP}(n)$.

Conversely, take a partition $\lambda:=(\lambda_1,\ldots,\lambda_l)\in\mathtt{NCCP}(n;132)$.
We first assume that $\lambda$ is not canonical.
There exists a pair of $(i,j)$ such that $\min(\lambda_{i})<\min(\lambda_{j})$
with $i<j$.
Since $\lambda_{i}$ and $\lambda_{j}$ are two different blocks, 
one can assume that $\lambda_{i}$ contains an integer $p$ which is greater than 
$\min(\lambda_{j})$ without loss of generality.
It is obvious that $\lambda$ contains the pattern $132$.
Since this is a contradiction, $\lambda$ is canonical.
Secondly, we assume that $\lambda$ has a crossing.
By a similar argument as above, $\lambda$ contains the pattern $132$.
Thus, $\lambda$ is non-crossing.
From these observations, $\lambda$ is a non-crossing canonical partition.
We have $\mathtt{NCCP}(n;132)\subseteq\mathtt{NCP}(n)$. 

As a consequence, we have $\mathtt{NCCP}(n;132)=\mathtt{NCP}(n)$.
\end{proof}

\begin{remark}
The Kreweras lattice is self-dual for each $n\ge1$ (see Theorem 1.1 in \cite{SimUll91}).
However, the lattice $\mathcal{L}_{NCCP}(n)$ is not self-dual in general.
The lattice $\mathcal{L}_{NCCP}(n)$ is self-dual up to $n=3$. 
For $n=4$, it is no longer self-dual (see Figure \ref{fig:latNCCP}).
The isomorphism between $\mathtt{NCP}(n)$ and $\mathtt{NCCP}(n;312)$ 
implies that the self-dual lattice can be a sublattice of $\mathcal{L}_{NCCP}(n)$.
An element in $\mathtt{NCCP}(n)\setminus\mathtt{NCCP}(n;312)$ is not canonical
or crossing. The two properties, being canonical and non-crossing, make the Kreweras 
lattice self-dual.
\end{remark}

\subsection{Maps on the Kreweras lattice}
We study two maps on the Kreweras lattice: an involution 
introduced in \cite{SimUll91} and Kreweras complement map \cite{Kre72}.
The latter is not an involution on the lattice, but this map gives a cyclic 
shift of an element in the lattice.
The former is constructed by modifying the complement map.
We interpret these two maps on the lattice $\mathcal{L}_{NCCP}(n)$.

\subsubsection{Involution on the Kreweras lattice}
\label{sec:Kreinv}

Given $\pi\in\mathtt{NCCP}(n)$, 
we define a dual $\overline{\pi}$ of $\pi$ where 
the bar operation is given by $\overline{i}=n+1-i$ for $i\in[n]$.
For example, $\overline{15/46/23}=6/23/15/4$.

The next lemma is a direct consequence of the bar operation.
\begin{lemma}
\label{lemma:312132}
Suppose $\pi\in\mathtt{NCCP}(n;312)$.
Then, $\overline{\pi}\in\mathtt{NCCP}(n;132)$.
\end{lemma}

From Proposition \ref{prop:NCP132}, we have $\mathtt{NCP}(n)=\mathtt{NCCP}(n;132)$.
Further, Lemma \ref{lemma:312132} implies that the bar operation gives 
a bijection between $\mathtt{NCCP}(n;312)$ and $\mathtt{NCCP}(n;132)$.
Recall that the composition $\phi\circ\psi^{-1}$ of the bijections 
$\psi^{-1}$ and $\phi$ gives a bijection between 
$\mathtt{NCP}(n)$ and $\mathtt{NCCP}(n;312)$.
Schematically, we have 
\begin{align}
\label{eq:312to312}
\mathtt{NCCP}(n;312)\xrightarrow{\text{bar}}
\mathtt{NCCP}(n;132)=\mathtt{NCP}(n)
\xrightarrow{\phi\circ\psi^{-1}}
\mathtt{NCCP}(n;312).
\end{align}

\begin{defn}
We denote by $\alpha:\mathtt{NCCP}(n;312)\rightarrow\mathtt{NCCP}(n;312)$ 
the map obtained by the composition of $\phi\circ\psi^{-1}$ and the bar operation 
(see Eq. (\ref{eq:312to312})).
\end{defn}

To characterize the map $\alpha$, we introduce another 
map $\alpha':\mathtt{NCP}(n)\rightarrow\mathtt{NCP}(n)$.
This map $\alpha'$ is studied in \cite{SimUll91} as an involution 
on $\mathtt{NCP}(n)$.

We consider a circle $C$ with $n$ points. 
The points are enumerated by $1,2\ldots,n$ clock-wise on $C$.
A circular presentation of $\pi:=(\pi_1,\ldots,\pi_{l})\in\mathtt{NCP}(n)$ on $C$ 
is given by connecting integers in each $\pi_{i}$.
If the cardinality of $\pi_{i}$ is one, we put a loop on the point.
We call a line connecting two labels a diagonal of $C$.

We append new $n$ points on $C$ by dividing the interval between two points labeled 
$i$ and $i+1$ for $1\le i\le n-1$, or $n$ and $1$.
We put a label $1'$ on the point between $n$ and $1$, and 
put label $i'$ on the $n-1$ remaining new points counterclockwise.
	
\begin{figure}[ht]
\begin{tikzpicture}[scale=0.8]
\draw circle(3cm);
\foreach \a in {0,36,72,108,...,360}
\filldraw [black] (\a:3cm)circle(1.5pt);
\draw (108:3cm) node[anchor=south] {$1'$};
\draw (180:3cm) node[anchor=east]{$2'$};
\draw (252:3cm) node[anchor=north east]{$3'$};
\draw (324:3cm) node[anchor=north west]{$4'$};
\draw (36:3cm) node[anchor=south west]{$5'$};
\draw (72:3cm) node[anchor=south west]{$1$};
\draw (0:3cm) node[anchor=west]{$2$};
\draw (-72:3cm)node[anchor=north west]{$3$};
\draw (-144:3cm)node[anchor=north east]{$4$};
\draw (-216:3cm)node[anchor=south east]{$5$};
\draw(72:3cm)to[bend right=30](0:3cm)to[bend right=15](-144:3cm)to[bend right=15](72:3cm);
\draw (-216:3cm) to [bend left=90] (-216:2.3cm) to [bend left=90] (-216:3cm);
\draw (-72:3cm) to [bend left=90] (-72:2.3cm) to [bend left=90] (-72:3cm);
\draw[dashed] (108:3cm) to[bend left=40] (180:3cm);
\draw[dashed] (252:3cm)to[bend left=40](324:3cm);
\draw[dashed] (36:3cm) to [bend left=90] (36:2.3cm)to[bend left=90](36:3cm);
\end{tikzpicture}
\caption{The map $\alpha'$. The partition $5/3/124$ is mapped to $5/34/12$.}
\label{fig:alphaprime}
\end{figure}

Since $\pi$ is non-crossing, the diagonals in $C$ are also non-crossing.
This means that the diagonals divide the inside of $C$ into
smaller regions.  
Given a circular presentation $\pi$, we construct $\pi'$ by 
connecting the points in the same region.

For example, we have $\alpha':5/3/124\mapsto 5/34/12$ as in Figure \ref{fig:alphaprime}.
The partition $5/3/124$ corresponds to the blocks grouped by solid lines, 
and $5/34/12$ corresponds to the blocks grouped by dashed lines.
Note that, by construction, solid lines and dashed lines never cross.
The pictorial presentation possesses the non-crossing property. 

\begin{lemma}
The map $\alpha':\mathtt{NCP}(n)\rightarrow\mathtt{NCP}(n)$ satisfies the following 
properties:
\begin{enumerate}
\item $\alpha'$ is an involution.
\item The rank function $\rho$ satisfies $\rho(\pi)+\rho(\alpha'(\pi))=n+1$.
\item $\alpha'$ is order-reversing.
\end{enumerate}
\end{lemma}
\begin{proof}
(1) Consider the circular presentation of $\pi$ and $\pi':=\alpha'(\pi)$ on the same 
circle. If we take a mirror image of the circular presentation and replace non-primed 
integers by primed integers, and vice versa, we have $\pi=\alpha'(\pi')=\alpha'^{2}(\pi)$.
Therefore, $\alpha'$ is an involution.

(2) We prove the statement by induction on $n$. It is trivial when $n=1,2$.
Suppose that $\pi$ contains the increasing sequence $(k_{1},k_2,\ldots,k_{p})$
where $k_1=1$ and $p$ is an integer. 
In the circular presentation, we have diagonals which connect $k_{i}$ and $k_{i+1}$ for 
$1\le i\le p-1$ and $k_{p}$ and $k_1$.
Since $\pi$ is non-crossing, we denote by $b_{i}$ the number of blocks  
between $k_{i}$ and $k_{i+1}$.
We have $k_{i+1}-k_{i}$ primed integers in the region $R$ separated by the diagonal connecting 
$k_{i}$ and $k_{i+1}$.
In the region $R$, we have $k_{i+1}-k_{i}-b_{i}$ blocks in $\pi'$.
We have $1+\sum_{i=1}^{p}b_{i}$ blocks in $\pi$ and 
$\sum_{i=1}^{p}(k_{i+1}-k_{i}-b_{i})=n-\sum_{i=1}^{p}b_{i}$ in $\pi'$.
By taking the sum of the number of blocks in $\pi$ and $\pi'$, 
we have $\rho(\pi)+\rho(\pi')=n+1$, which completes the proof.

(3) Suppose that $\pi'$ covers $\pi$, which implies that $\rho(\pi')=\rho(\pi)+1$.
From (2), we have $\rho(\alpha'(\pi))=\rho(\alpha'(\pi'))+1$.
It is enough to show that $\alpha'(\pi)$ covers $\alpha'(\pi')$.
Since $\pi'$ covers $\pi$, we have two blocks $\pi_{1},\pi_{2}\in\pi$ such that 
the merge $\pi_{1}\cup\pi_{2}$ of two blocks is a block in $\pi'$.
We assume that $\min(\pi_1)<\min(\pi_2)$ without loss of generality.
We have two cases:  a) $\max(\pi_1)<\min(\pi_2)$ and b) $\max(\pi_1)>\max(\pi_2)$.
Further, since $\pi$ is a non-crossing partition, 
there is no diagonal in $C$ connecting $i$ with $j$ in $\pi'$ such that
$\max(\pi_1)<i<\min(\pi_2)<\max(\pi_2)<j$ in the case a), and
$\min(\pi_1)<i<\min(\pi_2)<\max(\pi_2)<j<\max(\pi_1)$ in the case b).

\paragraph{Case a)}
Let $U(\pi)$ be the set of primed points between $\max(\pi_2)$ and $\min(\pi_{1})$ 
in the circular presentation, and $D(\pi)$ the set of primed points between 
$\max(\pi_1)$ and $\min(\pi_{2})$.
If we apply $\alpha'$ on $\pi$, we have diagonals connecting the primed points 
in $U(\pi)$ and $D(\pi)$.
Since $\pi_{1}\cup\pi_{2}$ is a single block in $\pi'$, there is no diagonal 
connecting the primed points in $U(\pi)$ and $D(\pi)$.
These observations imply that $\alpha'(\pi')$ is covered by $\alpha'(\pi)$.
Therefore, $\alpha'$ is order-reversing.

\paragraph{Case b)}
Let $i_0\in\pi_{1}$ (resp. $j_0\in\pi_{1}$) be the largest (resp. smallest) integer which 
is smaller (resp. larger) than $\min(\pi_2)$ (resp. $\max(\pi_2)$).
We denote by $U(\pi)$ the set of integers which are in $[\max(\pi_2),j_{0}]$, 
and by $D(\pi)$ the set of integers which are in $[i_{0}+1,\min(\pi_2)-1]$.
By applying the same argument as the case a) to $U(\pi)$ and $D(\pi)$, 
one can show that $\alpha'$ is order-reversing, which completes the proof.
\end{proof}

The next proposition reveals the relations among $\psi\circ\phi^{-1}$, $\alpha$, 
$\alpha'$ and the bar operation.
\begin{prop}
\label{prop:combaralpha}
The diagram 
\begin{align*}
\begin{tikzpicture}
\node (0) at (0,0){$\mathtt{NCCP}(n;312)$};
\node (1) at (3,0){$\mathtt{NCCP}(n;312)$};
\node (2) at (0,-2){$\mathtt{NCP}(n)$};
\node (3) at (3,-2){$\mathtt{NCP}(n)$};
\draw[->,anchor=south] (0) to node {$\alpha$} (1);
\draw[->,anchor=south] (2) to node {$\alpha'$} (3);
\draw[->,anchor=east] (0) to node {$\psi\circ\phi^{-1}$}(2);
\draw[->,anchor=west] (1) to node {$\psi\circ\phi^{-1}$}(3);
\draw[->,anchor=south west] (0) to node {\text{bar}} (3);
\end{tikzpicture}
\end{align*}
is commutative.
\end{prop}
\begin{proof}
By a simple calculation, one can show that the diagram is commutative $n=1$ and $2$.

Let $\pi\in\mathtt{NCCP}(n;312)$, $\lambda:=\varphi(\pi)\in\mathtt{NCP}(n)$ with $\varphi:=\psi\circ\phi^{-1}$ 
and $\alpha'(\lambda)\in\mathtt{NCP}(n)$.
We show that $\overline{\pi}=\alpha'(\lambda)$ by induction on $n$.
We assume that the diagram is commutative up to $n-1$.
Given a partition $\zeta\in\mathtt{NCCP}(n)$, we denote by $\zeta\downarrow i$ the 
partition obtained from $\zeta$ by deleting $i$.
Here, we view $\zeta\downarrow i$ as a partition in $\mathtt{NCCP}(n-1)$ by 
replacing an integer $j$ by $j-1$ if $j>i$ in $\zeta\downarrow i$. 
Note that we have $\lambda\downarrow n=\varphi(\pi)\downarrow n=\varphi(\pi\downarrow n)$.
By induction assumption, 
we have $\overline{\pi\downarrow n}=\alpha'(\lambda\downarrow n)$.
The bar operation and the deletion of $i$ on $\pi$ are commutative, which implies 
$\overline{\pi\downarrow n}=\overline{\pi}\downarrow 1$. 
We first show that $\alpha'(\lambda\downarrow n)=\alpha'(\lambda)\downarrow 1$.
We have two cases: 1) $1$ and $n$ are in the same block in $\lambda$, and 
2) $1$ and $n$ are in different blocks in $\lambda$.

\paragraph{Case 1)}
Since $1$ and $n$ are in the same block, $1'$ forms a block by itself. 
The deletion of $n$ in $\lambda$ corresponds to the deletion of $1$ in $\alpha'(\lambda)$.
This means that we delete the block containing only $1$ from $\alpha'(\lambda)$.
Therefore, we have $\alpha'(\lambda\downarrow n)=\alpha'(\lambda)\downarrow 1$.

\paragraph{Case 2)} 
We denote by $S(i')$ the set of primed integers which belong to the same block as $i'$ in 
$\alpha'(\lambda)$.
Since $1$ and $n$ are in different blocks, $1'$ is connected to some primed integers in 
$\alpha'(\lambda)$. We have $|S(1')|\ge2$.
The deletion of $n$ in $\lambda$ corresponds to merging $S(1')$ and $S(2')$.
Since $\alpha'(\lambda)$ is canonical, the deletion of $1$ simply corresponds to the merge
of the two blocks $S(1')$ and $S(2')$.
From these, $\alpha'(\lambda\downarrow n)=\alpha'(\lambda)\downarrow 1$.

From above considerations, we have $\overline{\pi}\downarrow1=\alpha'(\lambda)\downarrow1$.
To show $\overline{\pi}=\alpha'(\lambda)$, it is enough to show that 
the positions of $1$ in $\overline{\pi}$ and $\alpha'(\lambda)$ are the same. 
This means that if there are $p$ integers right to $n$ in $\pi$, 
the size of block which $1'$ belongs to is $p$.
In fact, we construct $\pi$ from a labeled binary tree by the in-order,
the number $p$ is equal to the number of primed integers which belong to 
the same block as $1'$.
From these, we have $\overline{\pi}=\alpha'(\lambda)$, which 
completes the proof.
\end{proof}

\begin{remark}
The involution $\alpha'$ is an order-reversing involution on 
the Kreweras lattice. This leads to the property that the Kreweras 
lattice is self-dual. Further, the Kreweras lattice admits the symmetric 
chain decomposition \cite{SimUll91}.
However, the lattice $\mathcal{L}_{NCCP}(n)$, $n\ge 4$, is not self-dual.
This is easily seen even for $n=4$ (see Figure \ref{fig:latNCCP}).
\end{remark}

\subsubsection{The Kreweras complement map}
In the previous section, we study an involution on the Kreweras lattice.
In \cite{Kre72}, Kreweras introduced a map $c:\mathtt{NCP}(n)\rightarrow\mathtt{NCP}(n), \pi\mapsto\pi^{c}$, 
which is similar to the involution in Section \ref{sec:Kreinv}.
The square of the complement map $c^{2}$ can be regarded as a rotation of a non-crossing partition.

Let $\pi\in\mathtt{NCP}(n)$ and consider its circular presentation $C(\pi)$.
When we construct a map $\alpha'$ in Section \ref{sec:Kreinv}, we add primed 
integers $i'$, $i\in[n]$, counter-clockwise. 
To construct the map $c$, we put primed integers clockwise such that $1'$ is between $1$ and $2$.
Recall that the diagonals connecting unprimed integers divide the inside of $C(\pi)$ into smaller regions.
As in the case of $\alpha'$, we construct $\pi^{c}$ by connecting the points with a prime in the 
same region.

\begin{example}
We have $\pi^{c}=7/68/45/2/13$ if $\pi=78/5/23/146$.
Then, the action of $c$ on $\pi^{c}$ is given by 
$(\pi^{c})^c=67/4/358/12$. 
Note that $(\pi^{c})^{c}$ can be obtained from $\pi$ by shifting
the integers by one and arranging the blocks in the canonical order.
\end{example}

Since we have a bijection between $\mathtt{NCCP}(n;312)$ and $\mathtt{NCP}(n)$, 
it is natural to construct a map $\beta:\mathtt{NCCP}(n;312)\rightarrow\mathtt{NCP}(n)$
which is compatible with complement map $c$.

Let $\omega:=(\omega_1,\ldots,\omega_{n})$ be a sequence of integers of length $n$ such that each integer appears at most once.
As in the case of a noncommutative crossing partition, we insert ``$/$" between $\omega_i$ and $\omega_{i+1}$
if $\omega_i>\omega_{i+1}$. By definition, ``$/$" divides $\omega$ into several blocks where each block 
contains an increasing sequence.
We define an increasing sequence $\mu(\omega;i):=(\mu_1,\ldots,\mu_{p})$ starting from $i$ 
from right to left in $\omega$ such that $\mu_1=i$, $\mu_1<\mu_2<\ldots<\mu_{p}$, 
and each $\mu_{i}$, $i\ge2$, is the maximal element in the block which it belongs to.  

More precisely, we construct a partition $\beta(\omega)$ from $\omega$ recursively as follows.
\begin{enumerate}
\item Set $i=1$. 
\item Define $r_{i}:=\min(\omega)$. Then, we construct an increasing sequence 
$\mu(\omega;r_{i}):=(r_i=\mu_1<\mu_2<\ldots<\mu_{p})$ from $\omega$ 
such that $\mu_{i}$ is the maximal element in the block of $\omega$ which it belongs to.
We denote this increasing sequence by $\lambda_{i}:=\mu(\omega;r_{i})$.
\item  Replace $\omega$ by $\omega\setminus\lambda_{i}$ where $\omega\setminus\lambda_{i}$ is an integer sequence 
obtained from $\omega$ by deleting the elements in $\lambda_{i}$. 
Then, if $\omega\neq\emptyset$, increase $i$ by one and go to (2). 
If $\omega=\emptyset$, go to (4).
\item Since we have the set of increasing sequences $\{\lambda_{i} | i\in[1,q]\}$ for some $q$, 
we construct a partition $\beta(\omega)$ by a concatenation of $\lambda_{i}$, {\it i.e.}, 
$\beta(\omega):=\lambda_{q}/\lambda_{q-1}/\ldots/\lambda_{1}$. 
\end{enumerate}

\begin{example}
Let $\pi=78/5/23/146\in\mathtt{NCP}(8)$. 
The corresponding element in $\mathtt{NCCP}(8;312)$ is $\pi'=23/15/478/6$ by 
the bijection $\phi\circ\psi^{-1}$.
The increasing sequence starting from $1$ is $\{1,3\}$.
Then $\pi'_{1}:=\pi'\setminus\{1,3\}=25/478/6$. 
Similarly, the increasing sequences of $\pi'_{1}$ starting from $2$, $4$, $6$ and
$7$ are $\{2\}$, $\{4,5\}$, $\{6,8\}$ and $\{7\}$ respectively.
The increasing sequences are the blocks in $\beta(\pi')$, and we arrange the blocks
in the canonical order. 
Therefore, $\beta(\pi')=7/68/45/2/13$.
\end{example}

The map $\beta$ can be characterized by restricting $\omega$ to the set $\mathtt{NCCP}(n;312)$.
\begin{lemma}
If $\omega\in\mathtt{NCCP}(n;312)$, then $\beta(\omega)\in\mathtt{NCP}(n)$.
\end{lemma}
\begin{proof}
By construction of $\beta(\omega)$, it is canonical. 
To show $\beta(\omega)\in\mathtt{NCP}(n)$, it is enough to show that 
$\beta(\omega)$ is non-crossing.
Suppose that $\beta(\omega):=\lambda_{q}/\lambda_{q-1}/\ldots/\lambda_{1}$ is crossing.
Then, there exist four integers $a<b<c<d$ such that $a$ and $c$ are in the block $\lambda_{p}$
and $b$ and $d$ are in the block $\lambda_{p'}$.
In $\omega$, $c$ is left to $a$ and $d$ is left to $b$.
If $b$ is left to $c$, then we have a $312$-pattern $dbc$ in $\omega$.
Similarly, if $b$ is right to $a$, we have a $312$-pattern $cab$ in $\omega$.
These contradict to the fact that $\omega\in\mathtt{NCCP}(n;312)$.
So, $b$ is between $c$ and $a$ in $\omega$. 
Since $b$ is not in the same block as $a$ and $c$ in $\beta(\omega)$, 
there exists $b<c'<c$ such that $c'$ is in the same block as $b$ in $\omega$, and 
in the same block as $a$ and $c$ in $\beta(\omega)$.
Then, the order of $\{b,c',c\}$ in $\omega$ is $c,b,c'$ from left to right.
This implies that $\omega$ has a $312$-pattern and we have a contradiction.
Thus, $b$ should not be between $c$ and $a$ in $\omega$.
This is a contradiction. 
Therefore, $\beta(\omega)$ is non-crossing, which completes the proof.
\end{proof}

The relations among the maps $\psi\circ\phi^{-1}$, $\beta$ and $c$ are summarized 
as the next proposition.
\begin{prop}
The diagram
\begin{align*}
\begin{tikzpicture}
\node (0) at (0,0){$\mathtt{NCCP}(n;312)$};
\node (1) at (0,-1.5){$\mathtt{NCP}(n)$};
\node (2) at (3,-1.5){$\mathtt{NCP}(n)$};
\draw[->,anchor=south west](0) to node {$\beta$} (2);
\draw[->,anchor=east](0) to node{$\psi\circ\phi^{-1}$} (1);
\draw[->,anchor=south](1) to node {$c$}(2);
\end{tikzpicture}
\end{align*}
is commutative.
\end{prop}
\begin{proof}
We use the same notation $\zeta\downarrow i$ as in the proof of Proposition \ref{prop:combaralpha} and 
define a bijection $\varphi:=\psi\circ\phi^{-1}:\mathtt{NCCP}(n;312)\rightarrow\mathtt{NCP}(n)$.
Let $\pi\in\mathtt{NCCP}(n;312)$ and $\lambda:=\varphi(\pi)$.
We will show that $c(\lambda)=\beta(\pi)$ by induction on $n$.
When $n=1,2$, we have $c(\lambda)=\beta(\pi)$ by a simple calculation.
We assume that the diagram is commutative up to $n-1$.
From the assumption, we have $c(\varphi(\pi\downarrow 1))=\beta(\pi\downarrow 1)$.
Since $\pi$ is in $\mathtt{NCCP}(n;312)$, the labeled binary tree $\phi^{-1}(\pi)$ 
is canonical.  
The deletion of $1$ in $\pi$ corresponds to the merge of the two blocks which contain 
$1$ and $2$ respectively in $\varphi(\pi)$ since $\varphi(\pi)$ is non-crossing. 
If a block contains the two integers $1$ and $2$, we just simply delete $1$ from this block. 
Thus, we have $\varphi(\pi\downarrow1)=\varphi(\pi)\downarrow1$, which implies 
$c(\varphi(\pi)\downarrow1)=\beta(\pi\downarrow1)$.

Since $\varphi(\pi)\in\mathtt{NCP}(n)$, it is canonical. 
If we delete $1$ from $\varphi(\pi)$, we may have to merge two blocks 
containing $1$ and $2$ if $1$ and $2$ belong to different blocks.
This corresponds to obtaining $c(\varphi(\pi)\downarrow1)$ from $c(\varphi(\pi))$ 
by deleting $1$ and sorting blocks in the canonical order.

Similarly, since $\pi\in\mathtt{NCCP}(n;312)$, we have no pattern $312$ in $\pi$.
This is equivalent to the fact that if a node labeled $i$ has a left edge, then 
the child node has a label $i+1$ in the labeled binary tree.
Since $\pi$ is obtained from a labeled binary tree by the in-order, the labels in the right 
subtree, whose root is the node labeled $i$, are right to $i$ in $\pi$.
Therefore, if we apply $\beta$ on $\pi$, $\beta(\pi\downarrow1)$ is obtained from 
$\beta(\pi)$ by deleting $1$ and sorting the blocks in the canonical order. 

Below, we will show that the block which $1$ belongs to is the same in 
$c(\varphi(\pi))$ and $\beta(\pi)$.
Let $L(T)$ be the labeled binary tree corresponding to $\pi$ and $\varphi(\pi)$.
Let $S$ be the set of labels in $L(T)$ satisfying the following conditions:
\begin{enumerate}
\item An element $s\in S$ is in the left subtree whose root is the node 
labeled $1$.
\item 
An element $s\in S$ is connected to the parent node by a right edge 
and it does not have a right child.
\item 
Let $s_{1}<s_{2}$ be the labels of the nodes which satisfy the above 
two conditions. 
If $s_2$ is left to $s_{1}$ in $L(T)$, two elements satisfy 
$s_{1},s_{2}\in S$.
If $s_{2}$ is right to $s_{1}$ in $L(T)$, 
we have $s_2\in S$ and $s_{1}\notin S$. 
\end{enumerate}
Then, by the definition of $c$ and $\beta$, 
the set $S\cup\{1\}$ forms the same block in both $c(\varphi(\pi))$ and $\beta(\pi)$.

From these observations, we have $c(\varphi(\pi))=\beta(\pi)$, which completes the proof.
\end{proof}

Let $\pi\in\mathtt{NCCP}(n;312)$ and $\pi':=\psi\circ\phi^{-1}(\pi)\in\mathtt{NCP}(n)$. 
Let $C(\pi')$ be a circular presentation of $\pi'$.

We say that $\pi'$ is {\it symmetric} if $C(\pi')$ is symmetric along the line 
which passes through two points labeled $1$ and $n/2+1$ 
(resp. the middle of two points labeled $(n+1)/2$ and $(n+3)/2$) 
if $n$ is even (resp. odd).

\begin{prop}
Suppose $\pi'\in\mathtt{NCP}(n)$ is symmetric.
Then, we have $\overline{\pi}=\beta(\pi)$ for $\pi\in\mathtt{NCP}(n;312)$.
\end{prop} 
\begin{proof}
Recall that the primed integers in the circular presentation are placed 
counter-clockwise for $\overline{\pi}$ and clockwise for $\beta(\pi)$.
Since $\pi'$ is symmetric, the connectivity of the primed integers in 
$\overline{\pi}$ is the same as that in $\beta(\pi)$.
\end{proof}

\begin{remark}
The involution $\alpha'$ studied in \cite{SimUll91} and the complement map 
$c$ defined in \cite{Kre72} are both skew-automorphisms.
Biane showed that the group of these skew-automorphisms is isomorphic 
to the dihedral group in \cite{Bia97}. 
\end{remark}

\subsection{Enumeration in Kreweras lattice}
Given a permutation $\omega:=\omega_1\ldots\omega_{n}$, we define the value $\mathrm{inv}(\omega)$
by 
\begin{align*}
\mathrm{inv}(\omega):=\{(i,j) | w_{i}>w_{j}, i<j \}.
\end{align*}
By regarding $\pi\in\mathtt{NCCP}(n)$ as a permutation $\omega(\pi)$ (by forgetting ``/" in one-line notation),
we define $\mathrm{inv}(\pi):=\mathrm{inv}(\omega(\pi))$.

We consider a multivariate generating function of $\mathtt{NCCP}(n;312)$:
\begin{align*}
C(r,t,q):=1+\sum_{n\ge1}\sum_{\pi\in\mathtt{NCCP}(n;312)}r^{n}t^{\rho'(\pi)}q^{\mathrm{inv}(\pi)},
\end{align*}
where $\rho'(\pi)=n-\rho(\pi)$, i.e., the number of blocks in $\pi$.
The first few terms of $C(r,t,q)$ are as follows.
\begin{align*}
C(r,t,q)=1+ r t +r^{2}(t+t^2q) +r^{3}(t+t^2(2q+q^2)+t^3q^3)+\cdots.
\end{align*}

\begin{prop}
\label{prop:C1}
The generating function $C(r,t,q)$ satisfies the following 
recurrence equation.
\begin{align}
\label{eq:recC1}
C(r,t,q)=1+tr C(rq,t,q)+C(rq,t,q)(C(r,t,q)-1)r.
\end{align}
\end{prop}
\begin{proof}
Let $\pi\in\mathtt{NCCP}(n;312)$ and $\omega:=\omega_1\ldots\omega_n$ 
be the corresponding permutation to $\pi$.
Suppose $\omega_i=1$.
Since $\omega$ is a $312$-pattern avoiding permutation, 
$\omega_{l}:=\omega_{1}\ldots\omega_{i-1}$ is a permutation in $[1,i-1]$ and 
$\omega_{r}:=\omega_{i+1}\ldots\omega_{n}$ is a permutation in $[i+1,n]$.
Thus, $\omega_{i}$ divides $\omega$ into two permutations.
Note that we have 
\begin{align*}
&\mathrm{inv}(\omega)=\mathrm{inv}(\omega_{l})+\mathrm{inv}(\omega_{r})+i-1, \\
&\rho'(\omega)=
\begin{cases}
\rho'(\omega_{l})+1, & \omega_{r}=\emptyset, \\
\rho'(\omega_{l})+\rho'(\omega_{r}), & \omega_{r}\neq\emptyset.
\end{cases}
\end{align*}
The shift $r\rightarrow rq$ gives $\mathrm{inv}(\omega)\rightarrow\mathrm{inv}(\omega)+n$, where 
$n$ is the size of the permutation $\omega$.

From these observations, we have the recurrence equation (\ref{eq:recC1}).
\end{proof}

From Proposition \ref{prop:C1}, $C(r,t,q)$ has an expression
in terms of a continued fraction:
\begin{align*}
C(r,t,q)&=1-\cfrac{rt}{r-\cfrac{1}{C(rq,t,q)}}, \\
&=1-\cfrac{rt}{r-\cfrac{1}{1-\cfrac{rtq}{rq-\cfrac{1}{1-\cdots}}}}.
\end{align*}

Note that the recurrence relation (\ref{eq:recC1}) is similar to 
that of Catalan numbers.
Actually, the number of non-crossing partitions is given by 
the Catalan numbers.
If we set $t=q=1$ in Eq. (\ref{eq:recC1}), we obtain the generating 
function of Catalan numbers.
We will rewrite the formal power series $C(r,t,q)$ in terms of 
another combinatorial object which is counted by Catalan numbers. 

To rewrite $C(r,t,q)$, we briefly introduce the notion of Dyck paths. 

\begin{defn}
A {\it Dyck} path of size $n$ is a lattice path from $(0,0)$ to (2n,0) 
with up steps $(1,1)$ and down steps $(1,-1)$, which never goes below
the horizontal line $y=0$.  
\end{defn}

It is well-known that the number of Dyck paths of size $n$ is given by 
the $n$-th Catalan number.	

We write an up step $U$ and a down step $D$. 
Then, we have five Dyck paths of size $3$:
\begin{align*}
UDUDUD, \quad UDUUDD, \quad UUDDUD, \quad UUDUDD, \quad UUUDDD.
\end{align*} 
The lowest path and the highest path are given by $(UD)^{n}$ and $U^{n}D^{n}$ respectively.
Let $d$ be a Dyck path of size $n$.
A pattern $UD$ in a Dyck path $d$ is called a {\it peak}. 
We denote by $\mathrm{Peak}(d)$ the number of peaks in $d$. 
We denote by $\mathrm{Area}(d)$ the number of squares which are above the lowest path
and $d$.
For example, the Dyck path $d=UUDDUD$ has two peaks and $\mathrm{Area}(d)=1$.

The generating function $C(r,t,q)$ can be expressed as a generating function 
of Dyck paths:
\begin{prop}
\label{prop:CinDyck}
We have 
\begin{align*}
C(r,t,q)=\sum_{n\ge0} r^{n}\sum_{d\in\mathtt{Dyck}(n)}t^{n+1-\mathrm{Peak}(d)}q^{\mathrm{Area}(d)}.
\end{align*}
\end{prop}
\begin{proof}
From the definition of a Dyck path $d$, $d$ touches the horizontal line $y=0$ at least once 
except the origin $(0,0)$.
If the path $d$ touches once, $d$ cannot be decompose into a concatenation of 
two Dyck paths. On the other hand, if $d$ touches more than twice, 
$d$ can be decomposed into a concatenation of two Dyck paths $d_1\circ d_2$,
where $d_1$ touches the horizontal line once.
If we translate this decomposition of a Dyck path into a recurrence 
relation for $C(r,t,q)$, one sees that $C(r,t,q)$ satisfies Eq. (\ref{eq:recC1}).
This completes the proof.
\end{proof}

From Proposition \ref{prop:CinDyck}, it is clear that 
the rank function $\rho$ reflects the number of peaks in a Dyck path, 
and the inversion number of $\pi$ reflects the area of the Dyck path.

\subsection{Shellability}
We introduce the notion of shellability of a poset $P$ following \cite{Bjo80,BjoGarSta82,BjoWac83}.
We denote by $C(P)$ the covering relations, that is, 
$C(P):=\{(x,y)\in P\times P| x<y\}$.
An {\it edge-labeling} of $P$ is a map $\lambda: C(P)\rightarrow \Lambda$, 
where $\Lambda$ is some poset. 
In this paper, since we consider an integer labeling of 
the poset $P$, we let $\Lambda=\mathbb{N}$.
We assign an integer to an each edges of the Hasse diagram of $P$.
Let $c:=x_{0}<x_{1}<\ldots<x_{k}$ be an unrefinable chain in the poset.
Then, an edge-labeling $\lambda$ is called {\it rising}
if $\lambda(x_{0},x_{1})\le\lambda(x_{1},x_{2})\le\ldots\le\lambda(x_{k-1},x_{k})$.
 
\begin{defn}[\cite{Bjo80,BjoGarSta82,BjoWac83}]
\label{defn:Llabeling}
We define an $R$-labeling and $EL$-labeling as follows.
\begin{enumerate}
\item An edge-labeling $\lambda$ is an $R$-labeling if any interval $[x,y]$ in $P$,
there exists a unique unrefinable chain $c: x=x_{0}<x_{1}<\ldots<x_{k}=y$ whose 
edge-labeling is rising.
\item $\lambda$ is called an $EL$-labeling if 
\begin{enumerate}
\item $\lambda$ is an $R$-labeling, 
\item for every interval $[x,y]$, these is a unique unrefinable chain $c$ and 
if $x<z\le y$, $z\not=x_{1}$, then $\lambda(x,x_{1})<\lambda(x,z)$.
\end{enumerate}
\end{enumerate}
\end{defn}

The condition (2b) implies that the labeling of the unique rising chain $c:x<x_1<\ldots<x_{k}=y$
for an interval $[x,y]$ is lexicographically first compared to other unrefinable chains.

\begin{defn}
A poset is {\it lexicographically shellable} ($EL$-shellable for short) 
if is is graded and admits an $EL$-labeling.
\end{defn}

\begin{prop}[Theorem 2.3 in \cite {Bjo80}]
\label{prop:lsp}
If the poset $P$ is a lexicographically shellable poset, then $P$ is shellable.
\end{prop}

\begin{prop}
\label{prop:NCCPLlabeling}
The poset $\mathcal{L}_{NCCP}$ admits an $EL$-labeling.
\end{prop}

To prove Proposition \ref{prop:NCCPLlabeling}, we will construct an $EL$-labeling 
on $\mathcal{L}_{NCCP}$.

Let $L(T)$ be a labeled tree and $\nu:=(\mathfrak{n},\mathfrak{n}_a,\mathfrak{m}(S))$ 
be the triplet constructed from $L(T)$ as in Section \ref{sec:covNCCP}.
We denote by $l(\mathfrak{n})$ the label of the node $\mathfrak{n}$ in $L(T)$.

\begin{defn}
If $L(T')$ covers $L(T)$ in the lattice $\mathcal{L}_{NCCP}$, that is, 
$\phi(L(T'))$ covers $\phi(L(T))$, 
we put a label $l(\mathfrak{n}_a)$ on the edge connecting $L(T)$ and $L(T')$.
We call this labeling $\widetilde{L}$-labeling.
\end{defn}

\begin{example}
An explicit $\widetilde{L}$-labeling of the lattice $\mathcal{L}_{NCCP}(3)$ is 
shown in Figure \ref{fig:LlabelNCCP3}. 
\begin{figure}[ht]
\begin{tikzpicture}
\node (0) at (0,0) {$123$}; 
\node (1) at (-3,-2){$13/2$};
\node (2) at (-1,-2){$2/13$};
\node (3) at (1,-2){$23/1$};
\node (4) at (3,-2){$3/12$};
\node (5) at (0,-4){$3/2/1$};
\draw[anchor=south east] (0.south) to node{$2$} (1.north);
\draw[anchor=east] (0.south) to node{$1$} (2.north);
\draw[anchor=west] (0.south) to node{$1$} (3.north);
\draw[anchor=south west] (0.south) to node{$1$} (4.north);
\draw[anchor=north east] (5.north) to node {$1$}(1.south);
\draw[anchor=east] (5.north) to node {$1$}(2.south);
\draw[anchor=west] (5.north) to node {$2$}(3.south);
\draw[anchor=north west] (5.north) to node {$1$}(4.south);
\end{tikzpicture}
\caption{$\widetilde{L}$-labeling on the lattice $(\mathtt{NCCP}(3),\le)$}
\label{fig:LlabelNCCP3}
\end{figure}
For example, we have a label $2$ on the edge from $3/2/1$ to $23/1$.
This is because the triplet $\nu$ is given by $\nu=(3,2,\emptyset)$.
For other edges from $3/2/1$, we have a label $1$ since $l(\mathfrak{n}_a)=1$.
\end{example}

We construct an $EL$-labeling from the $\widetilde{L}$-labeling.
\begin{defn}
Suppose that $\pi'$ covers $\pi$ and $\widetilde{L}(\pi,\pi')$ be the 
$\widetilde{L}$-labeling on the edge connecting $\pi$ and $\pi'$ in
the lattice $\mathcal{L}_{NCCP}$. 
The, we define 
\begin{align*}
L(\pi,\pi'):=n+\rho'(\pi)-\widetilde{L}(\pi,\pi'),
\end{align*}
where $\rho'(\pi)$ is the number of blocks in $\pi$.
\end{defn}

\begin{lemma}
\label{lemma:LNCCP}
The edge-labeling $L(\pi,\pi')$ is an $EL$-labeling.
\end{lemma}
\begin{proof}
By definition of an $\widetilde{L}$-labeling, $L(\pi,\pi')$ is 
an $R$-labeling if $\widetilde{L}(\pi,\pi')$ satisfies the following 
condition: 
For any interval $[x,y]$ in $\mathcal{L}_{NCCP}$, there exists a 
unique unrefinable chain $c:x=x_0<x_1<\ldots<x_{k}=y$
such that
\begin{align}
\label{eq:Lwdec}
\widetilde{L}(x_0,x_1)>\widetilde{L}(x_1,x_2)>\ldots>\widetilde{L}(x_{k-1},x_{k}).
\end{align}

Suppose that $x<y$ in $\mathcal{L}_{NCCP}$. 
Recall that the triplet $\nu=(\mathfrak{n},\mathfrak{n}_{a},\mathfrak{m}(S))$ 
gives a rotation on $x_{i}$.
Let $\mathfrak{n}_{\mathrm{max}}(x_{i})$ be the node with a maximum label in $\mathfrak{m}(S)$ 
such that it is connected to the child node by a left edge.
Since $x_{i+1}$ covers $x_{i}$ in the chain $c$, we have Eq. (\ref{eq:Lwdec}) 
by considering $L(x_{i},x_{i+1})=l(\mathfrak{n}_{\mathrm{max}})$.
The uniqueness is guaranteed by the uniqueness of $\mathfrak{n}_{\mathrm{max}}$ and 
by the fact that $l(\mathfrak{n}_{\mathrm{max}}(x_{i}))>l(\mathfrak{n}_{\mathrm{max}}(x_{i+1}))$	
By the condition (1) in Definition \ref{defn:Llabeling}, $L(x,y)$ is an $R$-labeling. 
 
Further, by construction of $\mathfrak{n}_{\mathrm{max}}$, there exists a unique chain 
$c$ and $\widetilde{L}(x,x_1)> \widetilde{L}(x,z)$ if $x<z\le y$ and $z\neq x_1$.
Therefore, by the condition (2b) in Definition \ref{defn:Llabeling}, $L(x,y)$
is an $EL$-labeling.
This completes the proof.
\end{proof}

\begin{proof}[Proof of Proposition \ref{prop:NCCPLlabeling}]
We have an explicit $EL$-labeling from Lemma \ref{lemma:LNCCP}.
\end{proof}

\begin{theorem}
The lattice $\mathcal{L}_{NCCP}$ is shellable.
\end{theorem}
\begin{proof}
From Corollary \ref{cor:NCCPgbl}, the poset $\mathcal{L}_{NCCP}$ is graded.
From Proposition \ref{prop:NCCPLlabeling}, $\mathcal{L}_{NCCP}$ admits 
an $EL$-labeling.
From Poposition \ref{prop:lsp}, $\mathcal{L}_{NCCP}$ is shellable, 
which completes the proof.
\end{proof}

\begin{remark}
The $EL$-shellability is a stronger condition compared to chain lexicographically shellability ($CL$-shellablity) \cite{BjoWac83}
 and standard notion of shellability.
One of important consequences of shelability of a poset $P$ is that 
the poset $P$ is Cohen--Macaulay (see for example \cite{Bjo80,BjoGarSta82,BjoWac83}).
\end{remark}

\subsection{M\"obius function and \texorpdfstring{$EL$}{EL}-labeling}
We will compute the M\"obius function of $\mathcal{L}_{NCCP}(n)$ and 
the Kreweras lattice by use of the explicit $EL$-labeling introduced 
in the previous subsection.

\subsubsection{M\"obius function}
\label{subsec:moebius}
We recall the definition of the M\"obius function following \cite{Gre82}.
The {\it M\"obius function} of a lattice $\mathcal{L}$, 
$\mu:\mathcal{L}\times\mathcal{L}\rightarrow\mathbb{Z}$, is defined recursively 
by 
\begin{align*}
\mu(x,y):=
\begin{cases}
1 & \text{if } x=y, \\
-\sum_{x\le z<y}\mu(x,z) & \text{if } x<y.
\end{cases}
\end{align*}
We define 
\begin{align*}
\mu(\mathcal{L}):=\mu(\hat{0},\hat{1}).
\end{align*}

One of the main results in this paper is the following theorem regarding 
to the M\"obius function of $\mathcal{L}_{NCCP}(n)$.

\begin{theorem}
\label{thrm:moebius}
Let $\mathcal{L}$ be the lattice $(\mathtt{NCCP}(n),\le)$ of noncommutative crossing 
partitions. 
Then, we have 
\begin{align*}
\mu(\mathcal{L})=(-1)^{n-1}(2n-3)!!.
\end{align*}
\end{theorem}
\begin{remark}
\label{rmrk:moebius}
Three remarks are in order.
\begin{enumerate}
\item Note that the value $\mu(\mathcal{L})$ of the M\"obius function is equal to 
the number of following combinatorial objects: 
the labeled ternary trees with $n-1$ internal nodes,
the number of labeled rooted planar trees with $n$ nodes, and the total number of Dyck tilings of size $n-1$.
\item Let $\mathcal{L'}$ be the sublattice $(\mathtt{NCCP}(n;312),\le)$ of $312$-avoiding noncommutative 
crossing partitions in $\mathcal{L}$.
Then, the M\"obius function is equal to the $(n-1)$-th Catalan number up to a sign  as calculated 
by Krewaras in \cite{Kre72}.
As in the case of $\mathcal{L}$, 
the $(n-1)$-th Catalan number is equal to  the number of (unlabeled) rooted planar trees 
with $n$ node and the one of (unlabeled) binary trees with $n-1$ internal nodes.  
\item 
Form observations (1) and (2), $\mu(\mathcal{L})$ corresponds to labeled combinatorial objects,
on the other hand, $\mu(\mathcal{L'})$ corresponds to unlabeled combinatorial objects.
Thus, the restriction of $\mathcal{L}$ to $\mathcal{L'}$ corresponds to forgetting 
the labeled structure in the lattice $\mathcal{L}$. 
\end{enumerate}
\end{remark}

\begin{remark}
Given a poset $(P,\le)$, one can associate a simplicial complex $\Delta:=\Delta(P,\le)$ such 
that the vertices of $\Delta$ are the elements of $P$, and the simplices of $\Delta$ are the chains 
of $P$. This simplicial complex is called the {\it order complex} of $P$.
Then, the M\"obius function of $P$ can be interpreted as the Euler characteristic 
of the order complex \cite{Sta97b1}. 
In our context, we have 
\begin{align*}
\mu(\mathcal{L})=\chi(\Delta(\mathcal{L})).
\end{align*}
\end{remark}

Before proceeding to the proof of Theorem \ref{thrm:moebius}, 
we introduce some propositions and lemmas.
We start from a characterization of the M\"obius function 
in terms of unrefinable chains in the Hasse diagram.

\begin{prop}[\cite{Bjo80,Sta74}]
\label{prop:MoebiusRlabel}
Suppose $\mathcal{L}$ admits an $R$-labeling.
When $x<y$ in $\mathcal{L}$, 
$(-1)^{r(x,y)}\mu(x,y)$ equals the number of chains 
$x=x_0<x_1<\ldots <x_{k}=y$ such that 
\begin{align}
\label{eq:chaindec}
\lambda(x_{0},x_1)>\lambda(x_{1},x_2)>\ldots>\lambda(x_{k-1},x_{k}),
\end{align}
where $\lambda$ is an edge-labeling.	
\end{prop}

From Proposition \ref{prop:MoebiusRlabel}, to compute the M\"obius 
function of a poset $P$, it is enough to 
count the number of chains which satisfy Eq. (\ref{eq:chaindec}).

For the computation of M\"obius function, we introduce a pair of integer 
sequences which come from the edge labeling.

Let $\mathbf{a}:=(a_1,\ldots,a_{n-1})$ be an integer sequence such that 
\begin{align}
\label{eq:conda}
\begin{split}
&1\le a_{i}\le i, \\
&a_{i}\le a_{i+1}, \quad i\in[n-2].
\end{split}
\end{align}
We denote by $A(n-1)$ the set of such sequences $\mathbf{a}$.
Give a sequence $\mathbf{a}$, we consider an integer sequence 
$\mathbf{b}:=(b_{1},\ldots,b_{n-1})$ 
such that 
\begin{align}
\label{eq:condb}
a_{i}< b_{i}\le n, \qquad i\in[n-1],
\end{align}
and all $b_{i}$'s are distinct in $[2,n]$, i.e., each integer $j\in[2,n]$ appears exactly 
once in $b_{i}$'s. 
We denote by $B(\mathbf{a})$ the set of sequences $\mathbf{b}$
satisfying Eq. (\ref{eq:condb}).

\begin{example}
Take $\mathbf{a}=(1,1,2)$.
Then, we have 
\begin{align*}
B(1,1,2)=\{ (2,3,4), (3,2,4), (2,4,3), (4,2,3) \}.
\end{align*}
\end{example}

\begin{lemma}
\label{lemma:cardA}
The cardinality $|A(n-1)|$ of $A(n-1)$ is given by $n-1$-st Catalan number.
\end{lemma}
\begin{proof}
We have a bijection between $\mathbf{a}$ and canonical labeled binary trees.
To see this, we construct a canonical labeled binary tree from $\mathbf{a}$ by 
the following procedures.
First, since $a_{1}=1$, we put a binary tree consisting of a single node which is 
labeled one.
Suppose that we have a canonical labeled binary tree $T_{p}$ corresponding to $(a_{1},\ldots,a_{p})$ with 
$p\ge1$.
Then, we add a binary tree consisting of a single node labeled $p+1$ to the $a_{p+1}$-th 
external edge of $T_{p}$ from left.
Since $a_{p}\le a_{p+1}$, if $T_{p}$ is canonical, then $T_{p+1}$ is also canonical.
It is obvious that this map from $\mathbf{a}$ to a canonical labeled binary tree 
is invertible. Thus, the number of $\mathbf{a}$ is equal to that of canonical labeled 
binary trees with $n-1$ internal nodes, which is equal to the $n-1$-st Catalan number.
\end{proof}

\begin{prop}
\label{prop:Cardab}
We have 
\begin{align*}
\sum_{\mathbf{a}\in A(n-1)}|B(\mathbf{a})|
=\prod_{j=0}^{n-1}(2j+1).
\end{align*}
\end{prop}
\begin{proof}
From the definition of $\mathbf{a}$, the cardinality of $B(\mathbf{a})$ 
is given by
\begin{align*}
|B(\mathbf{a})|=\prod_{i=1}^{n-1}(i+1-a_{i}).
\end{align*}
Thus, we have
\begin{align}
\label{eq:sumB}
\begin{split}
\sum_{\mathbf{a}\in A(n-1)}|B(\mathbf{a})|
&=\sum_{\mathbf{a}\in A(n-1)}\prod_{i=1}^{n-1}(i+1-a_{i}), \\
&=\sum_{\mathbf{a}\in A'(n-1)}\prod_{i=1}^{n-1}(1+a_{i}),
\end{split}
\end{align}
where 
\begin{align*}
A'(n-1):=\Big\{\mathbf{a}:=(a_1,\ldots,a_{n-1})\Big| 
\begin{array}{c}
a_1=0 \\
0\le a_{i}\le a_{i-1}+1, \quad \text{ for } i\in[2,n-1]
\end{array}
\Big\}.
\end{align*}
We have the following identity, which is easy to prove:
\begin{align}
\label{eq:sumfact}
\sum_{0\le x\le y+1}\prod_{j=1}^{n}(x+j)
=\genfrac{}{}{}{}{1}{n+1}\prod_{j=2}^{n+2}(y+j).
\end{align}
By taking the summation of Eq. (\ref{eq:sumB}) with respect to $a_{n-1}, a_{n-2},\ldots,a_{2}$,
we have 
\begin{align*}
\sum_{\mathbf{a}\in A'(n-1)}\prod_{i=1}^{n-1}(1+a_{i})
&=\prod_{j=1}^{n-2}\genfrac{}{}{}{}{1}{2j}\sum_{a_{1}=0}\prod_{j=1}^{2n-3}(a_{1}+j), \\
&=\left(\prod_{j=1}^{n-2}\genfrac{}{}{}{}{1}{2j}\right)(2n-3)!, \\
&=\prod_{j=0}^{n-2}(2j+1),
\end{align*}
where we have used Eq. (\ref{eq:sumfact}). This completes the proof. 
\end{proof}

\begin{proof}[Proof of Theorem \ref{thrm:moebius}]
Recall that $\mathcal{L}_{NCCP}$ admits an $EL$-labeling, which is an $R$-labeling.
Proposition \ref{prop:MoebiusRlabel}, to compute the M\"obius function $\mu(\mathcal{L}_{NCCP})$, 
we will count the number of chains which satisfy 
Eq. (\ref{eq:chaindec}).
In terms of the $\widetilde{L}$-labeling, it is enough to count the number of chains which satisfy
\begin{align}
\label{eq:LiLi1}
\widetilde{L}_{i}\le \widetilde{L}_{i+1}, \qquad i\in[1,n-1]
\end{align}
where $\widetilde{L}_{i}:=\widetilde{L}(x_{i-1},x_{i})$.

We refine the $\widetilde{L}$-Labeling. 
Recall that if $x_{i+1}$ covers $x_{i}$, the triplet $\nu=(\mathfrak{n},\mathfrak{n}_a,\mathfrak{m}(S))$
defines the rotation on $x_{i}$.
By definition of the $\widetilde{L}$-labeling, we have $\widetilde{L}_{i}=l(\mathfrak{n}_{a})$.
We define two integer sequences $\mathbf{a}:=(a_1,\ldots,a_{n-1})$ and $\mathbf{b}:=(b_1,\ldots,b_{n-1})$ by 
\begin{align}
\label{eq:ab}
\begin{split}
&a_i=l(\mathfrak{n}_{a}), \\
&b_{i}=l(\mathfrak{n}).
\end{split}
\end{align}
By definition of the triplet $\nu$ for $x_{i}$, we have $l(\mathfrak{n})>l(\mathfrak{n}_a)$, 
i.e., $b_{i}>a_{i}$.
Since we consider maximal chains from $\hat{0}$ to $\hat{1}$, 
each integer $p\in[2,n]$ appears exactly once in $\mathbf{b}$.

The condition (\ref{eq:LiLi1}) is equivalent to the condition $a_{i}\le a_{i+1}$.
Further, if the chains are maximal, the integer sequence $\mathbf{a}$ satisfies $a_{i}\le i$.
Suppose that $a_{1}\ge2$. Then, from Eq. (\ref{eq:LiLi1}), there is no integer $i$ such that 
$a_{i}=1$. This means that the node labeled $1$ (the root of the binary tree) has a left 
child. This contradicts the fact that we consider the maximal chains.
We have $a_{1}\le 1$. One can similarly show $a_{i}\le i$ by induction on $i$. 

As a summary, the sequences $\mathbf{a}$ and $\mathbf{b}$ satisfy the 
conditions (\ref{eq:conda}) and (\ref{eq:condb}).
Counting the maximal chains with the condition (\ref{eq:LiLi1}) is equivalent 
to counting the pairs $(\mathbf{a},\mathbf{b})$ with conditions (\ref{eq:conda}) 
and (\ref{eq:condb}).
From Proposition \ref{prop:Cardab}, the number of such maximal chains 
is equal to $\prod_{j=0}^{n-2}(2j+1)$ and the sign is $(-1)^{n-1}$, 
which completes the proof.	
\end{proof}

In Figure \ref{fig:LlabelNCCP32}, we give an explicit refined labeling 
for the lattice $\mathcal{L}_{NCCP}(3)$.
The pair of integers $(a,b)$ refers the pair $(l(\mathfrak{n}_a),l(\mathfrak{n}))$
in the triplet (see Eq. (\ref{eq:ab})). 
Note that an integer in $\{2,3\}$ appears exactly once as $b$ in an unrefinable maximal 
chain. 
\begin{figure}[ht]
\begin{tikzpicture}
\node (0) at (0,0) {$123$}; 
\node (1) at (-3,-2){$13/2$};
\node (2) at (-1,-2){$2/13$};
\node (3) at (1,-2){$23/1$};
\node (4) at (3,-2){$3/12$};
\node (5) at (0,-4){$3/2/1$};
\draw[anchor=south east] (0.south) to node{$(2,3)$} (1.north);
\draw[anchor=east] (0.south) to node{$(1,2)$} (2.north);
\draw[anchor=west] (0.south) to node{$(1,2)$} (3.north);
\draw[anchor=south west] (0.south) to node{$(1,3)$} (4.north);
\draw[anchor=north east] (5.north) to node {$(1,2)$}(1.south);
\draw[anchor=east] (5.north) to node {$(1,3)$}(2.south);
\draw[anchor=west] (5.north) to node {$(2,3)$}(3.south);
\draw[anchor=north west] (5.north) to node {$(1,2)$}(4.south);
\end{tikzpicture}
\caption{The refined $\widetilde{L}$-labeling on the lattice $(\mathtt{NCCP}(3),\le)$.}
\label{fig:LlabelNCCP32}
\end{figure}

\subsubsection{M\"obius function for the Kreweras lattice}
Let $\mathbf{a}$ and $\mathbf{b}$ be integer sequences defined in Section \ref{subsec:moebius}.

\begin{prop}
\label{prop:Krec}
Given an integer sequence $\mathbf{a}$,
we have a unique maximal chain $c$ satisfying Eq. (\ref{eq:chaindec}) 
in the lattice $\mathcal{L}_{NCCP}$ such that $c$ contains only elements 
in $\mathtt{NCCP}(n;312)$.
\end{prop}
\begin{proof}
Suppose that $\mathbf{a}:=(a_1,\ldots,a_{n-1})$ is given.
We construct $\mathbf{b}'$ from $\mathbf{a}$ recursively as follows:
\begin{enumerate}
\item Set $i=n-1$ and $S=[2,n]$.
\item Define $b_{i}$ as the $a_{n-1}$-th smallest integer $s_{\mathrm{min}}$ in $S$.
\item Then, replace $S$ by $S\setminus\{s_{\mathrm{min}}\}$, decrease $i$ by one, and 
go to (2).
The algorithm stops when $i=0$.
\end{enumerate}
Once we fix a pair of integer sequences $\mathbf{a}$ and $\mathbf{b'}$, we have 
a maximal chain from $\hat{0}$ to $\hat{1}$.
Then, by the definition of a rotation, this maximal chain contains only elements 
in $\mathtt{NCCP}(n;312)$ since the labeled binary trees associated to the chain
are all canonical.
Suppose $\mathbf{b''}\neq\mathbf{b'}$ for a given $\mathbf{a}$.
Then, it is clear that the pair $(\mathbf{a},\mathbf{b''})$ gives an 
element which is not in $\mathtt{NCCP}(n;312)$ since there exists at least 
one binary trees which are not canonical.
Therefore, $\mathbf{b'}$ gives a unique maximal chain satisfying Eq. (\ref{eq:chaindec}).
This completes the proof.
\end{proof}

\begin{cor}[Theorem 6 in \cite{Kre72}]
\label{cor:KreMoebius}
The M\"obius function $\mu(\mathcal{L}_{\mathrm{Kre}}(n))$ is equal to 
$(-1)^{n-1}C_{n-1}$, where $C_{n-1}$ is the $(n-1)$-st Catalan number.
\end{cor}
\begin{proof}
The $EL$-labeling on $\mathcal{L}_{NCCP}$ is also a $EL$-labeling on 
the Kreweras lattice labeled an element in $\mathtt{NCCP}(n;312)$.
From Proposition \ref{prop:MoebiusRlabel},
it is enough to count the number of maximal chains satisfying 
the condition Eq. (\ref{eq:chaindec}) in the Kreweras lattice 
to compute the M\"obius function.

From Proposition \ref{prop:Krec}, the number of chains is 
equal to the cardinality of $\mathbf{a}$, which 
is equal to the $(n-1)$-st Catalan number 
by Lemma \ref{lemma:cardA}.
This completes the proof.
\end{proof}

\begin{remark}
For the Kreweras lattice, two types of $EL$-labelings are known.
The first one is proposed by Gessel in \cite[Exapmpel 2.9]{Bjo80} and the second 
one is by Stanley in \cite{Sta97}.
We briefly recall their definitions.

Suppose we have a cover relation $\pi\subset\pi'$. Then this corresponds to
a merging of two blocks $B_1$ and $B_2$ of $\pi$ into a single block $B_1\cup B_2$ 
of $\pi'$. 
We have the following two edge labelings $\lambda_1$ and $\lambda_2$.
\begin{enumerate}
\item 
The edge labeling $\lambda_1$ is given by $\lambda_1(\pi,\pi'):=\max\{\min(B_1), \min(B_{2})\}$ 
\cite{Bjo80}.
\item 
Suppose that $\min(B_1)<\min(B_{2})$. 
Then, the edge labeling $\lambda_2$ is given by 
$\lambda_2(\pi,\pi')=\max\{i\in B_1: i < B_{2} \}$,
where $i< B_2$ denotes that $i$ is less than any element of $B_2$ \cite{Sta97}.
\end{enumerate}

The first edge-labeling $\lambda_1$ behaves nicely in the case of the lattice 
$(\mathtt{NCP}(n),\subseteq)$. 
The second edge-labeling $\lambda_2$ is equivalent to the restriction of 
the $\widetilde{L}$-labeling from $(\mathtt{NCCP}(n),\le)$ to $(\mathtt{NCCP}(n;312),\le)$.
\end{remark}

\begin{remark}
The M\"obius function for the Kreweras lattice was also recomputed by use of 
what is called {\it NBB bases} in the lattice in \cite{BlaSag97}.
\end{remark}

\subsubsection{Multiplicities}
\label{sec:mult}
Let $(\mathbf{a},\mathbf{b})$ a pair of integer sequences of length $n-1$ 
defined in Section \ref{subsec:moebius}.
From Proposition \ref{prop:Cardab}, we have $(2n-3)!!$ such pairs.
By definition of $\mathbf{b}$, we have a permutation 
$w:=\mathbf{b}-1=(b_1-1,\ldots,b_{n-1}-1)$ of length $n-1$.
We say that $\mathbf{b}$ is congruent to a permutation $w$ if $w=\mathbf{b}-1$.

In the previous section, we focus on $\mathbf{a}$ rather than $\mathbf{b}$.
In this section, we focus on $\mathbf{b}$ and count the number of 
$\mathbf{a}$ satisfying the conditions (\ref{eq:conda}) and (\ref{eq:condb}).

\begin{defn}
We denote by $M(w)$ the number of $\mathbf{a}$'s such that $\mathbf{b}$ is 
congruent to $w$.
\end{defn}

To obtain a recurrence equation for $M(w)$, we introduce some definitions.

\begin{defn}
Given a permutation $w:=(w_1,\ldots,w_{n-1})$, we define an increasing sequence 
(from left to right) $\gamma(w):=(\gamma_1,\ldots,\gamma_{n-1})$ recursively such that 
\begin{enumerate}
\item Each $\gamma_{i}$ is either ``$\ast$" or an integer in $[n-1]$.
\item $\gamma_{n-1}=w_{n-1}$. 
\item Suppose $\gamma_{j}\neq\ast$. 
If $w_{i}<w_{j}<w_{k}$ with $i+1\le k\le j-1$, we define $\gamma_{i}:=w_{i}$ and $\gamma_{k}:=\ast$
for $i+1\le k\le j-1$.
\end{enumerate}
We say that $w$ is type $\gamma$ if $\gamma=\gamma(w)$.
\end{defn}

Note that if we ignore $\ast$'s in $\gamma(w)$, we have an increasing integer sequence 
from left to right.
For example, we have $\gamma(w)=*1**25$ if $w=614325$.

\begin{defn}
We denote by $\Gamma(n-1)$ the set of possible types $\gamma(w)$ where $w$ is a permutation
of length $n-1$.
\end{defn}

Let $\gamma\in\Gamma(n)$.
We construct $\gamma'\in\Gamma(n-1)$ from $\gamma$ as follows.
Fix  an integer $\nu$ in $\{1,2,\ldots,\gamma_{n}\}$. 
Note that $\gamma_{n}\neq\ast$ by construction of $\gamma$ and $\gamma_{n}$ is 
the maximum in $\gamma$ if we ignore $\ast$'s.
We have two cases 1) $\gamma_{n-1}=\ast$, and 2) $\gamma_{n-1}\neq\ast$.

\paragraph{Case 1)}
We define $\gamma':=(\gamma'_{1},\ldots,\gamma'_{n-1})$:  
\begin{enumerate}
\item If $\gamma_{i}=\ast$, then $\gamma'_{i}:=\ast$ for $1\le i\le n-2$.
\item $\gamma'_{n-1}:=\min\{\nu,n-1\}$.
\item If $\gamma_{i}\neq\ast$ and $\gamma_{i}\ge\gamma'_{n-1}$, then $\gamma'_{i}=\gamma'_{n-1}$ for $1\le i\le n-2$.
Similarly, if $\gamma_{i}\neq\ast$ and $\gamma'_{i}<\gamma'_{n-1}$, then $\gamma'_{i}=\gamma_{i}$ for $1\le i\le n-2$.
\end{enumerate}

\paragraph{Case 2)}
We define $\gamma':=(\gamma'_{1},\ldots,\gamma'_{n-1})$ as Case 1).
We replace the condition (2) by the following condition:  
\begin{enumerate}
\item[(2')] $\gamma'_{n-1}:=\min\{\gamma_{n-1},\nu\}$.
\end{enumerate}

\begin{defn}
We write $\gamma\xrightarrow{\nu}\gamma'$ if $\gamma'\in\Gamma(n-1)$ is obtained from 
$\gamma\in\Gamma(n)$ by the procedure above.
\end{defn}

\begin{remark}
Since there are $\gamma_{n}$ choices for $\nu$, we have $\gamma_{n}$ $\gamma'$'s.
Note that we may have the same $\gamma'$ from different $\nu\in\{1,\ldots,\gamma_{n}\}$.
This occurs when $\gamma_{n-1}\neq\ast$. 
\end{remark}

\begin{example}
Suppose $\gamma=1234$. 
We have $\gamma'\in\{123, 122, 111 \}$ and 
the multiplicity of $123$ is two ($\nu=3,4$).
\end{example}

\begin{defn}
We define a number $m(\gamma)$, $\gamma\in\Gamma(n)$ recursively by
\begin{align*}
m(\gamma)=\sum_{\gamma':\gamma\xrightarrow{\nu}\gamma'}m(\gamma'),
\end{align*}
with the initial condition $m(\gamma_{0})=1$ where $\gamma_{0}=(1)\in\Gamma(1)$.
We call $m(\gamma)$ the {\it multiplicity} of $\gamma$.
\end{defn}

\begin{example}
Take $\gamma=12*3$.
Then, by definition, we have
\begin{align*}
m(12*3)=m(123)+m(122)+m(111). 
\end{align*}
Similarly, we have $m(123)=2m(12)+m(11)=5$, $m(122)=m(12)+m(11)=3$, and $m(111)=1$.
Therefore, $m(12*3)=5+3+1=9$.
In fact, by definition, $m(12*3)$ is  the number of integer sequences $\mathbf{a}$ 
such that the integer sequence $\mathbf{b}$ is type $12*3$, i.e.,  
\begin{align*}
m(12*3)&=\#\{1111,1112,1113,1122,1123,1133,1222,1223,1233\}, \\
&=9.
\end{align*}
\end{example}

\begin{prop}
Suppose $\gamma=(1,2,\ldots,n)$.
Then, the multiplicity $m(\gamma)$ is equal to the $n$-th Catalan number.
\end{prop}
\begin{proof}
Suppose that we have a sequence $\nu_{n},\nu_{n-1},\ldots,\nu_{1}$ 
such that 
\begin{align*}
\gamma_{i}\xrightarrow{\nu_{i}}\gamma_{i-1}, \qquad 1\le i\le n,
\end{align*}
where $\gamma_{i}\in\Gamma(i)$.
Then, the sequence $\widetilde{\nu}=(\nu_1,\ldots,\nu_{n})$ satisfies 
$1\le \nu_{i}\le i$ and $\nu_{i-1}\le \nu_{i}$ by definition.
The total number of $\widetilde{\nu}$ is equal to the number of 
integer sequences $\mathbf{a}$, which is equal to the $n$-th
Catalan number. This completes the proof.
\end{proof}

The following proposition gives another formula for the 
M\"obius function of $\mathcal{L}_{NCCP}(n)$ in terms of 
the multiplicities $m(\gamma)$.

Let $\Gamma'(n)$ be the set of sequences of integers $(a_1,\ldots,a_{n})$
such that $1\le a_{i}\le i$ for $i\in[1,n]$.
It is easy to see that the cardinality of $\Gamma'(n)$ is $n!$.
\begin{prop}
The sum of $m(\gamma)$ is given by 
\begin{align}
\label{eq:summgamma}
\sum_{\gamma\in\Gamma'(n-1)}m(\gamma)=(2n-3)!!.
\end{align}
\end{prop}
\begin{proof}
The sum in the left hand side of Eq. (\ref{eq:summgamma}) is equal 
to the number of the pairs $(\mathbf{a},\mathbf{b})$ satisfying the 
conditions (\ref{eq:conda}) and (\ref{eq:condb}).
Then, by Proposition \ref{prop:Cardab}, the sum is equal to 
$(2n-3)!!$, which completes the proof.
\end{proof}
For example, we have six sequences in $\Gamma'(3)$:
\begin{align*}
111\quad 112 \quad 113 \quad 121 \quad 122 \quad 123
\end{align*}
Note that two sequences $(1,1,1)$ and $(1,2,1)$ give the 
same type $(*,*,1)$.
If we replace the sum on $\Gamma'(n-1)$ by $\Gamma(n-1)$ in Eq. (\ref{eq:summgamma}),
we need to introduce the multiplicity of the type $\gamma$.
It is easy to verify that $m(**1)=m(111)=m(121)=1$.

The next proposition gives the number $M(w)$ in terms of $m(\gamma)$. 
\begin{prop}
\label{prop:wgamma}
Suppose that $w$ is type $\gamma\in\Gamma(n-1)$.
Then, we have 
\begin{align*}
M(w)=m(\gamma).
\end{align*}
\end{prop}
\begin{proof}
The value $M(w)$ is equal to the number of $\mathbf{a}:=(a_1,\ldots,a_{n-1})$ such that 
$a_{i}$ is weakly increasing and $a_{i}\le w_{i}\le n-1$.
Given a permutation $w$ of type $\gamma=(\gamma_1,\ldots,\gamma_{n-1})$, 
we define $\mathbf{a}_{\mathrm{max}}:=(a'_1,\ldots,a'_{n-1})$ by 
$a'_{i}=\gamma_{i}$ if $\gamma_i\neq\ast$, and 
$a'_{j}=\gamma_{k}$ if $\gamma_{j}=\ast$ for $i<j<k$ and $\gamma_{i},\gamma_{k}\neq\ast$.
Then, $M(w)$ is the number of sequences $\mathrm{a}$ such that $\mathbf{a}$ is smaller 
than or equal to $\mathrm{a}_{\mathrm{max}}$ in the lexicographic order.

On the other hand, we have a chain
\begin{align}
\label{eq:chainnu}
\gamma\xrightarrow{\nu_{n-1}}\gamma'\xrightarrow{\nu_{n-2}}\ldots\xrightarrow{\nu_2}(1)\xrightarrow{\nu_1}\emptyset.
\end{align} 
Then, $m(\gamma)$ counts the number of $\mu=(\nu_1,\ldots,\nu_{n-1})$ satisfying 
Eq. (\ref{eq:chainnu}).
By construction of $\nu$, it is easy to see that $\nu$ is smaller or equal to 
$\mathrm{a}_{\mathrm{max}}$ in the lexicographic order.
Therefore, We have $M(w)=m(\gamma)$.
\end{proof}

The following corollary is a direct consequence of Proposition \ref{prop:wgamma}.
\begin{cor}
If $w$ and $w'$ has the same type $\gamma\in\Gamma(n-1)$, then $M(w)=M(w')$.	
\end{cor}
 
\subsection{Maximal chains and labeled parking functions} 
\label{sec:lpk}
We briefly review the definition of a parking function following \cite{Sta97} and introduce 
its labeling.

A {\it parking function} is a sequence $\mathbf{a}:=(a_1,\ldots,a_{n})$ 
of positive integers in $[n]$ such that 
if $\widetilde{a}_{1}\le \widetilde{a}_{2}\le\cdots\le \widetilde{a}_{n}$
is the increasing rearrangement of $\mathbf{a}$, then $\widetilde{a}_{i}\le i$ 
for $1\le i\le n$.

A {\it labeled parking function} is a pair of two integer sequences 
$(\mathbf{a},\mathbf{b})$ in $[n]$ such that 
$\mathbf{a}$ is a parking function, $a_{i}\le b_{i}$ for $1\le i\le n$, 
and each integer $p\in[n]$ appears in $\mathbf{b}$ exactly once.
We denote by $\mathcal{LP}_{n}$ the set of labeled parking functions $(\mathbf{a},\mathbf{b})$.

\begin{example}
For $n=3$, we have sixteen parking functions:
\begin{align*}
111 \qquad 112 \qquad 121 \qquad 211\qquad 113 \qquad 131 \qquad 311 \qquad 122 \\
212 \qquad 221 \qquad 123 \qquad 132 \qquad 213\qquad 231 \qquad 312 \qquad 321
\end{align*}
Consider a parking function $\mathbf{a}=121$. Then, possible sequences for $\mathbf{b}$ are 
\begin{align*}
123 \qquad 321 \qquad 132 \qquad 231.
\end{align*}
The number of labeled parking functions for $n=3$ is $36$.  
\end{example}

The next two propositions give the number of parking functions 
and labeled parking functions.

\begin{prop}[\cite{KonWei66,Pyk59}]
\label{prop:pk}
The number of parking functions in $[n]$ is $(n+1)^{n-1}$.
\end{prop}

Define the generating function for labeled parking functions:
\begin{align*}
I_{n}(p,q):=\sum_{(\mathbf{a},\mathbf{b})\in\mathcal{LP}_{n}}
p^{a_1+a_2+\cdots+a_{n}-n}q^{\mathrm{inv}(\mathbf{b})},
\end{align*} 
where $\mathrm{inv}(\mathbf{b})$ is the number of inversions 
in the permutation $\mathbf{b}$.

\begin{prop}
\label{prop:lpk}
We have
\begin{align*}
I_{n}(p,q)=\prod_{i=1}^{n}\genfrac{}{}{}{}{(1-p^{i})(1-q^{i})}{(1-p)(1-q)}.
\end{align*}
\end{prop}
\begin{proof}
Let $\mathbf{b}_{0}=(1,2,\ldots,n)$.
The number of labeled parking functions $(\mathbf{a},\mathbf{b}_{0})$ is 
$n!$ since $\mathbf{b}_{0}$ implies $1\le a_{i}\le i$.
Recall that each $a_{i}$ contributes to $I_{n}(p,q)$ by $p^{a_{i}}$.
Therefore, we have a generating function $\prod_{i=1}^{n}(1-p^{i})/(1-p)$
for a fixed $\mathbf{b}_{0}$.	

Since a sequence $\mathbf{b}$ is a permutation of $[n]$, the number of 
$\mathbf{b}$ is $n!$.
It is obvious that the generating function of such $\mathbf{b}$ is given 
by $\prod_{i=1}^{n}(1-q^{i})/(1-q)$ since $\mathrm{inv}(\mathbf{b})$ is the 
same as the length function.
We denote by $w$ the permutation for $\mathbf{b}$.
Then, a labeled parking function $(\mathbf{a}',\mathbf{b})$ is obtained 
from $(\mathbf{a},\mathbf{b}_{0})$ by the action of a permutation 
$w$ on $\mathbf{a}$.
This means that if we fix a $\mathbf{b}$, the generating function with respect 
to $\mathbf{a}$ is given by $\prod_{i=1}^{n}(1-p^{i})/(1-p)$.

From these observations, we have $n!$ integer sequences $\mathbf{b}$, and 
each integer sequence $\mathbf{b}$ has $n!$ integer sequences $\mathbf{a}$.
Therefore, the generating function of labeled parking functions is 
given by $\prod_{i=1}^{n}(1-p^{i})/(1-p)\prod_{j=1}^{n}(1-q^{j})/(1-q)$.
\end{proof}

The number of labeled parking functions can be deduced from Proposition \ref{prop:lpk}.
\begin{cor}
\label{cor:lpk}
The number of labeled parking functions in $[n]$ is $(n!)^{2}$.
\end{cor}
\begin{proof}
Set $p=q=1$ in Proposition \ref{prop:lpk}.
\end{proof}

The following theorem connects the number of maximal chains in $\mathcal{L}_{NCCP}(n)$ 
with the number of labeled parking functions.
\begin{theorem}
\label{thrm:mchainNCCP}
The number of maximal chains in the lattice $\mathcal{L}_{NCCP}(n)$ 
is $((n-1)!)^{2}$.
\end{theorem}
\begin{proof}
Recall that we have the refined $\widetilde{L}$-labeling on the edges in $\mathcal{L}_{NCCP}(n)$
as in the proof of Theorem \ref{thrm:moebius}.
We define a pair of integer sequences $(\mathbf{a},\mathbf{b})$ by Eq. (\ref{eq:ab}).
The number of maximal chains in $\mathcal{L}_{NCCP}(n)$ is equal to the 
number of pairs $(\mathbf{a},\mathbf{b})$.
Note that the conditions on $(\mathbf{a},\mathbf{b})$ are the same as labeled parking functions.
Thus, the number of pairs $(\mathbf{a},\mathbf{b})$ is equal to 
the number of labeled parking functions, i.e., $((n-1)!)^2$ by Corollary \ref{cor:lpk}.
\end{proof}

\begin{theorem}[\cite{Sta97}]
\label{thrm:mchainKre}
The number of maximal chains in the lattice $\mathcal{L}_{\mathrm{Kre}}(n)$ 
is $n^{n-2}$.
\end{theorem}
\begin{proof}
Recall that the lattice $\mathcal{L}_{\mathrm{Kre}}(n)$ is a sublattice of 
$\mathcal{L}_{NCCP}(n)$.
We consider the $\widetilde{L}$-labeling on $\mathcal{L}_{NCCP}(n)$ 
as in the proof of Theorem \ref{thrm:mchainNCCP}.
Note that the element in $\mathcal{L}_{\mathrm{Kre}(n)}$ is in $\mathtt{NCCP}(n;312)$.
This means that we have a unique maximal chain if we choose a parking function $\mathbf{a}$.
In fact, the integer sequence $\mathbf{b}$ in the refined $\widetilde{L}$-labeling for 
$\mathbf{a}:=(a_1,\ldots,a_{n-1})$ is given as in the proof of Proposition \ref{prop:Krec}.
Thus, the number of maximal chains is equal to the number of 
parking functions, i.e., $(n)^{n-2}$ by Proposition \ref{prop:pk}.
\end{proof}

\begin{remark}
From Theorems \ref{thrm:mchainNCCP} and \ref{thrm:mchainKre}, 
the number of maximal chains in $\mathcal{L}_{NCCP}(n)$ is equal to 
the number of labeled parking functions, and that in $\mathcal{L}_{\mathrm{Kre}}(n)$ 
is equal to the number of (unlabeled) parking functions.
As observed in Remark \ref{rmrk:moebius}, combinatorics on $\mathcal{L}_{\mathrm{Kre}}(n)$ 
can be obtained from $\mathcal{L}_{NCCP}(n)$ by ignoring the labels on combinatorial objects.
\end{remark}

\section{Generalized Dyck tilings and \texorpdfstring{$k$}{k}-ary trees}
\label{sec:kDyck}

In this section, we introduce two combinatorial objects: $k$-ary trees 
and $k$-Dyck tilings.
As we will see below, a labeled $k$-ary trees is bijective to 
a $k-1$-Dyck tiling.
If we forget the labels of a labeled $k$-ary tree, we have 
$k$-ary trees (without labels) which are enumerated by 
the Fuss--Catalan number. This implies that a $k$-ary tree without 
labels is bijective to a $k-1$-Dyck paths.
Similarly, if we restrict ourselves to $k-1$-Dyck tilings without 
non-trivial tiles, they are bijective to $k-1$-Dyck paths.
When we study the noncommutative crossing partitions, 
labels in a tree and non-trivial tiles in a $k$-Dyck tiling 
play a central role.

\subsection{\texorpdfstring{$k$}{k}-Ary trees}
\label{sec:kDktree}
In Section \ref{sec:lbt}, we introduce a labeled binary tree.
In this section, we consider a $k$-ary tree and its labeling.

A {\it $k$-ary tree} is a rooted tree in which each node has 
at most $k$ children. 
A labeled $k$-ary tree $L(T)$ is a tree $T$ with labeling 
satisfying the conditions in Definition \ref{def:labeledtree}.

Let $L(T)$ be a complete labeled $k$-ary tree with $n$ nodes.
We will construct an integer sequence $h(L(T)):=(h_1,\ldots,h_{n})$ 
such that 
\begin{align}
\label{eq:condh}
0\le h_{i}\le (k-1)(i-1).
\end{align}
When $n=1$, we define $h=(0)$, which corresponds to a tree consisting 
of a single node, the root.
Suppose that $n\ge2$ and $L(T')$ be a labeled tree obtained 
from $L(T)$ by deleting the node with the integer label $n$.
We denote by $h_{n}+1$ the position (counted from right to left) 
of the edge connected to the node with the label $n$ in $L(T)$.
We also denote by $h(L(T')):=(h_{1},\ldots,h_{n-1})$ the integer 
sequence associated with $L(T')$.
Then, we define $h(L(T)):=(h_{1},\ldots,h_{n})$.
By construction of $h(L(T))$, the integer $h_{n}$ is in $[0,(k-1)(i-1)]$.

\begin{prop}
\label{prop:numlktree}
The number of labeled $k$-ary trees with $n$ nodes  is given 
by 
\begin{align}
\label{eq:numktree}
\prod_{j=0}^{n-1}((k-1)j+1).
\end{align}
\end{prop} 
\begin{proof}
By definition of an integer sequence $h(L(T))$ constructed from a labeled $k$-ary tree $L(T)$,
each integer $h_{i}$, $1\le i\le n$, satisfies the condition (\ref{eq:condh}).
Conversely, given an integer sequence $h$ satisfying the condition (\ref{eq:condh}), 
we have a unique labeled $k$-ary tree.
Thus, we have a canonical bijection between $h$ and $L(T)$.
The number of $k$-ary trees with $n$ nodes is equal to the number of $h$ satisfying 
(\ref{eq:condh}). 
We have Eq. (\ref{eq:numktree}) as a direct consequence of observations above.
\end{proof}

We define analogously a canonical $k$-ary tree by use of Definition \ref{defn:canotree} 
as in the case of binary trees.
Consider the two labels $1\le l_1<l_2\le n$ in a $k$-ary tree $T$.
The tree $T$ is canonical if and only if the node labeled $l_2$ is 
weakly right to the node labeled $l_1$.

It is well-known that the number of binary trees with $n$ nodes is
given by the Catalan number. The following proposition gives the 
number of $k$-ary trees with $n$ nodes (which do not have labels).

\begin{prop}
\label{prop:numcktree}
The number of canonical $k$-ary tree with $n$ nodes is given by 
the $n$-th Fuss--Catalan number
\begin{align*}
\genfrac{}{}{}{}{1}{kn+1}\genfrac{(}{)}{0pt}{}{kn+1}{n}.
\end{align*}
\end{prop}
\begin{proof}
Note that given a $k$-ary tree, we have a unique labeling on the tree
which is canonical.
Thus, the number of canonical $k$-ary trees is equal to the number 
of $k$-ary trees without labels.
Then, this number is equal to the $n$-th Fuss--Catalan number 
by a classical theory of enumerative combinatorics. 
\end{proof}

\subsection{Generalized Dyck tilings}
\label{sec:GDyckT}
We briefly summarize basic facts about $k$-Dyck tilings
(see e.g. \cite{JosVerKim16,KeyWil12,KimMesPanWil14,Shi21,ShiZinJus12} and 
references therein).

A $k$-Dyck path of size $n$ is a lattice path from $(0,0)$ to 
$(kn,n)$ which never goes below the line $y=x/k$.
We represent a $k$-Dyck path by up steps $U=(0,1)$ and down steps $D=(1,0)$.
By definition, a $k$-Dyck path of size $n$ has $n$ up steps and $kn$ down steps.
We have a unique lowest path, denoted $\mu_{0}=(UD^{k})^{n}$.

A {\it $k$-Dyck tile} is a ribbon (connected skew shape which does not contain a $2\times2$ box)
such that the centers of the unit boxes form a $k$-Dyck path. 
The size of $k$-Dyck tile is the size of the $k$-Dyck path which characterize the tile.
Especially, a unit box is a Dyck tile of size $0$.
A {\it $k$-Dyck tiling} is a tiling by $k$-Dyck tiles above a Dyck path $\mu$.
A $k$-Dyck tiling is called {\it cover-inclusive} if the sizes of $k$-Dyck tiles 
are weakly decreasing from bottom to top.
In terms of the tiles, cover-inclusiveness implies that if we translate a Dyck tile by $(1,-1)$,
then it is contained in another Dyck tile or strictly below the path $\mu$. 
In this paper, we consider only cover-inclusive $k$-Dyck tilings above the lowest path $\mu_{0}$.
The restriction of the lowest path is important when we study the relation between
a chain in the lattice of noncommutative crossing partitions and a $k$-Dyck tilings.

We say that a unit box in a Dyck tiling is a trivial tile, and a Dyck tile of size $n\ge1$ is 
a non-trivial tile.

In Figure \ref{fig:kDT}, we give an example of $2$-Dyck tiling of size $n=5$.
We have seven unit boxes, one $2$-Dyck tile of size one, and one $2$-Dyck tile of size $2$.
Therefore, we have two non-trivial tiles in this $2$-Dyck tiling.
\begin{figure}[ht]
\begin{tikzpicture}[scale=0.6]
\foreach \a in {0,1,2,3,4}
\draw(\a+\a,\a)--(\a+\a,\a+1)--(\a+\a+2,\a+1);
\draw(1,1)--(1,3)--(3,3)--(3,4)--(6,4);
\draw (0,1)--(0,2)--(1,2);
\draw(0,2)--(0,4)--(3,4);
\draw(2,4)--(2,5)--(8,5);
\foreach \a in {3,4,5,6,7}
\draw(\a,4)--(\a,5);
\draw[red](0,1.5)--(1.5,1.5)--(1.5,2.5)--(3.5,2.5)--(3.5,3.5)--(6,3.5);
\draw[red](0,2.5)--(0.5,2.5)--(0.5,3.5)--(3,3.5);
\draw[red](2,4.5)--(8,4.5);
\end{tikzpicture}
\caption{A cover-inclusive $2$-Dyck tiling of size $5$}
\label{fig:kDT}
\end{figure}
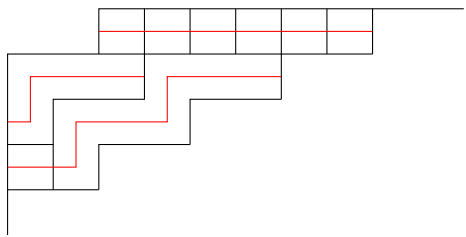

We will construct a bijection between a $k-1$-Dyck tiling $D$ and 
$k$-ary labeled binary tree $L(T)$.
Since we have a bijection between $h(L(T))$ satisfying Eq. (\ref{eq:condh})
and a labeled $k$-ary tree $L(T)$ in Section \ref{sec:kDktree}, it is enough 
to construct a bijection between $D$ and $h(L(T))$.
The bijection can be realized by use of trajectories on a Dyck tiling.

Recall that a $k-1$-Dyck tile $D$ is a ribbon. 
We put a red line in $D$ from east edge to the south-most west edge.
We draw red lines on all Dyck tiles which form a Dyck tiling.
We call the red line a trajectory. 
In the case of Figure \ref{fig:kDT}, we have three red lines 
corresponding to three trajectories.
The right end of a trajectory is associated to the up step $u$ in $\mu_{0}$ by 
translating it by $(p,-p)$ for some non-negative integer $p$.
Then, we say that the trajectory is associated with the up step $u$.
Let $\mathtt{Traj}(i)$ is the set of Dyck tiles where the trajectory 
is associated with the $i$-th up step (from left) in $\mu_{0}$.

A $k-1$-Dyck tile $d$ of size $n$ has a weight $\mathrm{wt}(d):=(k-1)n+1$.
Then, the integer sequence $h:=(h_1,\ldots,h_{n})$ is defined in 
terms of the weight of Dyck tiles:
\begin{align*}
h_{i}:=\sum_{d\in\mathtt{Traj}(i)}\mathrm{wt}(d).
\end{align*}

By construction, we always have $h_1=0$ since there is no trajectory associated 
to the left-most up step.
It is easily shown that $h_{i}$ satisfies the condition Eq. (\ref{eq:condh}), 
and there is a one-to-one correspondence between a sequence $h$ and a $k-1$-Dyck 
tiling $D$.

For example, the $2$-Dyck tiling in Figure \ref{fig:kDT} has 
a sequence $h$:
\begin{align*}
h=(0,0,3,6,6).
\end{align*}
Here, $h_2=0$ since there is no trajectory associated to the second up step in $\mu_{0}$,
and $h_{3}=3$ since the Dyck tile of size one is associated to the third up step in $\mu_{0}$.

\begin{lemma}
\label{lemma:hntDT}
Let $h=(h_1,\ldots,h_{n})$ be the integer sequence obtained from a $k-1$-th Dyck tiling.
The followings are equivalent:
\begin{enumerate}
\item The Dyck tiling contains at least one non-trivial Dyck tiles.
\item There exists $i\in[2,n-1]$ such that $h_{i}+k-1<h_{i+1}$.
\end{enumerate}
\end{lemma}
\begin{proof}
(2) $\Rightarrow$ (1). Suppose that $h_{i}+k-1<h_{i+1}$ and  let $t_{i}$ and $t_{i+1}$ be 
a trajectory associated to the $i$-th and $i+1$-st up steps in $\mu_{0}$.
In $\mu_{0}$, we have $k-1$ down steps between the $i$-th and $i+1$-st
up steps. Therefore, the condition that $h_{i}+k-1<h_{i+1}$ means that 
the $i+1$-st trajectory reaches far left to the $i$-th trajectory.
This is impossible if Dyck tiles are trivial in the $i+1$-st trajectory, and 
possible if we have non-trivial Dyck tiles in the $i+1$-st trajectory and 
the $i$-th trajectory is translated by $(-1,1)$.

(1) $\Rightarrow$ (2). 
Conversely, suppose that a Dyck tiling contains at least one non-trivial 
Dyck tiles and there exists a non-trivial Dyck tile in the $i+1$-st trajectory.
Then, by construction of trajectories, 
it is easily shown that we have $h_{i}+k-1<h_{i+1}$.

From these considerations, (1) and (2) are equivalent.
\end{proof}

We have the following proposition as a summary.

\begin{prop}
\label{prop:kDyckenu}
We have 
\begin{enumerate}
\item The number of $k$-Dyck tilings above $\mu_{0}$ is 
given by $\prod_{j=0}^{n-1}((k-1)j+1)$.
\item The number of $k$-Dyck tilings containing only trivial tiles 
is given by $n$-th Fuss--Catalan number.
\end{enumerate}
\end{prop}
\begin{proof}
(1) A $k$-Dyck tiling is bijective to a labeled $k$-ary tree. 
Thus, the number of $k$-Dyck tilings is equal to the number of labeled 
$k$-ary trees, i.e., $\prod_{j=0}^{n-1}((k-1)j+1)$ by Proposition \ref{prop:numlktree}

(2) 
Let $h:=(h_1,\ldots, h_{n})$ be the integer sequence associated to a $k$-Dyck tiling containing only trivial tiles.
Define $h':=(h'_1,\ldots, h'_{n})$ by $h'_{i}=(k-1)i-h_{i}$.
From Lemma \ref{lemma:hntDT}, $h'$ satisfies $h'_{i}\le h'_{i+1}$ and $h_{i}\le (k-1)i$.
It is obvious that the number of such $h'$ is equal to the number of (unlabeled) $k$-ary tree.
Recall that the number of unlabeled $k$-ary tree is equal to that of canonical labeled $k$-ary trees. 
Thus, the number of such tilings 
is equal to the number of canonical trees, i.e., the $n$-th 
Fuss--Catalan number by Proposition \ref{prop:numcktree}.
\end{proof}

\begin{remark}
Two remarks are in order:
\begin{enumerate}
\item
When $k=2$, we have a standard notion of Dyck tilings. 
Proposition \ref{prop:kDyckenu} implies that the number of 
Dyck tilings of $n$ nodes is $n!$.
Note that this number is the same as the number of noncommutative 
crossing partitions in $[n]$.
\item
Recall that the number of elements in $\mathtt{NCP}(n)$ is given by 
the Catalan number. The number of the Dyck tilings ($k=2$ case) with 
only trivial tiles is also equal to the Catalan number.
This fact can be obtained by the fact that we have a natural bijection 
between a Dyck tilings with only trivial tiles and a Dyck path.
A Dyck tiling with non-trivial tiles corresponds to a labeled binary 
tree which is not canonical, or crossing.
\end{enumerate}
\end{remark}

\section{Intervals of the lattice}
\label{sec:int}
Recall that we have the correspondence between $k-1$-chains in the lattice 
of non-crossing partitions and $k$-ary trees \cite{Edel80}.
This is realized by identifying an element in a $k-1$-chain with 
a labeled canonical binary tree, and successively by building up 
a $k$-ary tree from $k-1$ labeled canonical binary trees.
In this section, we generalize the correspondence to the case of 
non-commutative crossing partitions.
Recall that noncommutative crossing partitions are a generalization of non-crossing partitions.
Thus, if such a correspondence exists, the restriction of noncommutative crossing partitions 
to canonical non-crossing partitions has to give the same correspondence studied in \cite{Edel80}.  
This implies that a partition which is not canonical or crossing has to give 
a $k$-ary tree which is not canonical.	
For this purpose, we need to introduce combinatorial objects which contain 
$k$-ary trees as a subset.
This is achieved by introducing the labeled $k$-ary trees, or equivalently,
the $k$-Dyck tilings above the lowest path $\mu_{0}$.

\subsection{Decomposition of \texorpdfstring{$k$}{k}-ary trees by binary trees}
\label{sec:deckbt}
In this section, we give a recipe to decompose labeled $k$-ary trees by use of 
$k-1$ labeled binary trees.

Let $L:=L(T(n;k))$ be a labeled $k$-ary tree with $n$ internal nodes.
Since $L$ is a $k$-ary tree, every node has $k$ children, or equivalently 
$k$ edges below it.
Recall that each node in $k$-ary tree has $k$ edges below it.
We call the $p$-th edge from right below a node a $p$-edge for $1\le p\le k$.

We construct a set $\mathfrak{D}(L):=(D_{1},\ldots,D_{k-1})$ 
of $k-1$ labeled binary trees $D_{i}$, $1\le i\le k-1$, from $L$ 
by the following way:
\begin{enumerate}
\item Fix an integer $i\in[1,k-1]$. 
\item We divide $p$-edges with $1\le p\le k$ into two sets $RE(i)$ and $LE(i)$ where 
$RE(i)$ is the set of $p$-edges with $1\le p\le i$, and $LE(i)$ is the set of 
$p$-edges with $i+1\le p\le k$.
\item When $n=1$, $D_{i}$ is a labeled binary tree with the root labeled one.
We construct $D_{i}$ recursively as follows.
\begin{enumerate}
\item 
Suppose we have a labeled binary tree $D'_{i}$ corresponding to $L'$ where 
$L'$ is a labeled $k$-ary tree consisting of labels in $[1,n-1]$.
The labeled tree $L'$ is obtained from $L$ by deleting the node $\mathfrak{n}(n)$ labeled $n$.
Let $n'$ be the label of the parent node $\mathfrak{n}(n')$ of $\mathfrak{n}(n)$ in $L$, and 
$e(n,n')$ be the edge in $L$ connecting  the node $\mathfrak{n}(n')$ with the node $\mathfrak{n}(n)$.
By definition, we have $n'<n$.
\item 
Let $\mathfrak{m}$ be the node labeled $n'$ in $D'_{i}$.
\begin{enumerate}
\item
If $e(n,n')\in RE(i)$, recall we have a sequence of nodes in $D'_{i}$: 
a right-extended sequence $\overrightarrow{q}(\mathfrak{m})$ as in Definition \ref{defn:lereseq}.
We add the node labeled $n$ to $D'_{i}$ in such a way that 
the node labeled $n$ is connected to the node labeled the maximal integer in $\overrightarrow{q}(\mathfrak{m})$ 
by a right edge. The right-extended sequence is lengthened by one. 
\item
If $e(n,n')\in LE(i)$, take a right-most path $p_{l}(\mathfrak{m})$ from the node $\mathfrak{m}$ to a leaf 
such that the node $\mathfrak{m}$ is connected to a child by a left edge in $p_{l}(\mathfrak{m})$.
If such a path does not exist, then we define $p_{l}(\mathfrak{m}):=\mathfrak{m}$.
We add the node labeled $n$ to $D'_{i}$  by a left edge in such a way that 
the node labeled $n$ is connected to the node at the leaf of $p_{l}(\mathfrak{m})$.
\end{enumerate}
\item We define $D_{i}$ as the new labeled binary tree obtained from $D'_{i}$ and node labeled $n$.
\end{enumerate}
\end{enumerate}

Note that the adding a node labeled $n$ by a left edge in (3-b-ii) is well-defined 
even when $p_{l}(\mathfrak{m})=\mathfrak{m}$ since the node $\mathfrak{m}$ does not 
have a left child.
 
\begin{example}
\label{ex:LDD}
We consider the following canonical ternary labeled tree $L$ and its decomposition
of two labeled binary trees $\mathfrak{D}(L)=(D_1,D_2)$.
In Figure \ref{fig:LDD}, we recursively construct two labeled binary trees $D_1$ and $D_2$ 
from the canonical ternary labeled tree $L$.

\begin{figure}[ht]
\begin{align*}
\begin{tikzpicture}
\node at (0,0){$L:$};
\node at (1.5,0){
\tikzpic{-0.5}{
[grow'=up,level distance=0.5cm, 
level 1/.style={sibling distance=0.6cm},
level 2/.style={sibling distance=0.77cm},
level 3/.style={sibling distance=0.42cm}
]
\node {$1$} child{}child{}child{};
}
};
\node at (3,0){$\rightarrow$};
\node at (4.5,0) {
\tikzpic{-0.5}{
[grow'=up,level distance=0.5cm, 
level 1/.style={sibling distance=0.6cm},
level 2/.style={sibling distance=0.6cm},
level 3/.style={sibling distance=0.42cm}
]
\node {$1$} child{node{$2$}child{}child{}child{}}child{}child{};
}
};
\node at (6,0){$\rightarrow$};
\node at (7.5,0){
\tikzpic{-0.5}{
[grow'=up,level distance=0.8cm, 
level 1/.style={sibling distance=0.8cm},
level 2/.style={sibling distance=0.77cm},
level 3/.style={sibling distance=0.42cm}
]
\node {$1$} child{node{$2$}child{}child{}child{node{$3$}}}child{}child{};
}
};
\node at (9.5,0) {$\rightarrow$};
\node at (11,0){
\tikzpic{-0.5}{
[grow'=up,level distance=0.8cm, 
level 1/.style={sibling distance=0.8cm},
level 2/.style={sibling distance=0.77cm},
level 3/.style={sibling distance=0.42cm}
]
\node {$1$} child{node{$2$}child{}child{}child{node{$3$}}}child{node{$4$}}child{};
}
};
\node at (0,-2.5) {$D_{1}:$};
\node at (1.5,-2.5){
\tikzpic{-0.5}{
[grow'=up,level distance=0.5cm, 
level 1/.style={sibling distance=0.6cm},
level 2/.style={sibling distance=0.6cm},
level 3/.style={sibling distance=0.42cm}
]
\node {$1$} child{}child{};
}
};
\node at (3,-2.5){$\rightarrow$};
\node at (4.5,-2.5){
\tikzpic{-0.5}{
[grow'=up,level distance=0.7cm, 
level 1/.style={sibling distance=0.6cm},
level 2/.style={sibling distance=0.6cm},
level 3/.style={sibling distance=0.42cm}
]
\node {$1$} child{node{$2$}child{}child{}}child{};
}
};
\node at (6,-2.5){$\rightarrow$};
\node at (7.5,-2.5){
\tikzpic{-0.5}{
[grow'=up,level distance=0.7cm, 
level 1/.style={sibling distance=0.6cm},
level 2/.style={sibling distance=0.6cm},
level 3/.style={sibling distance=0.42cm}
]
\node {$1$} child{node{$2$}child{}child{node{$3$}}}child{};
}
};
\node at (9.5,-2.5){$\rightarrow$};
\node at (11,-2.5){
\tikzpic{-0.5}{
[grow'=up,level distance=0.7cm, 
level 1/.style={sibling distance=0.6cm},
level 2/.style={sibling distance=0.6cm},
level 3/.style={sibling distance=0.42cm}
]
\node {$1$} child{node{$2$}child{}child{node{$3$}child{node{$4$}}child{}}}child{};
}
};
\node at (0,-4.8) {$D_{2}:$};
\node at (1.5,-4.8){
\tikzpic{-0.5}{
[grow'=up,level distance=0.5cm, 
level 1/.style={sibling distance=0.6cm},
level 2/.style={sibling distance=0.6cm},
level 3/.style={sibling distance=0.42cm}
]
\node {$1$} child{}child{};
}
};
\node at (3,-4.8){$\rightarrow$};
\node at (4.5,-4.8){
\tikzpic{-0.5}{
[grow'=up,level distance=0.7cm, 
level 1/.style={sibling distance=0.6cm},
level 2/.style={sibling distance=0.6cm},
level 3/.style={sibling distance=0.42cm}
]
\node {$1$} child{node{$2$}child{}child{}}child{};
}
};
\node at (6,-4.8){$\rightarrow$};
\node at (7.5,-4.8){
\tikzpic{-0.5}{
[grow'=up,level distance=0.7cm, 
level 1/.style={sibling distance=0.6cm},
level 2/.style={sibling distance=0.6cm},
level 3/.style={sibling distance=0.42cm}
]
\node {$1$} child{node{$2$}child{}child{node{$3$}}}child{};
}
};
\node at (9.5,-4.8){$\rightarrow$};
\node at (11,-4.8){
\tikzpic{-0.5}{
[grow'=up,level distance=0.7cm, 
level 1/.style={sibling distance=0.6cm},
level 2/.style={sibling distance=0.6cm},
level 3/.style={sibling distance=0.42cm}
]
\node {$1$} child{node{$2$}child{}child{node{$3$}}}child{node{$4$}};
}
};
\end{tikzpicture}
\end{align*}
\caption{The decomposition of a ternary tree into two binary trees}
\label{fig:LDD}
\end{figure}
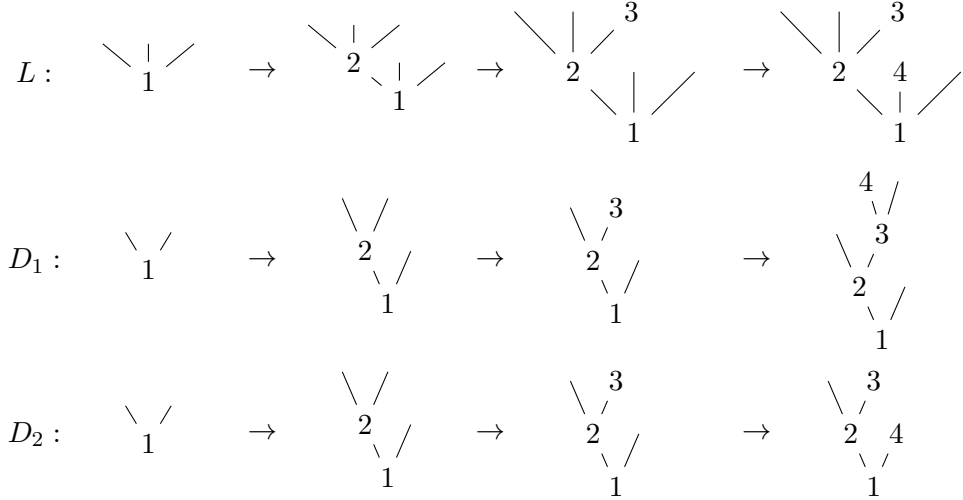
\end{example}

\begin{example}
\label{ex:ktobinary}
We consider a labeled $k$-ary tree $L$ with $k=4$ depicted as in Figure \ref{fig:4tree}.
Note that the tree is not canonical.
\begin{figure}[ht]
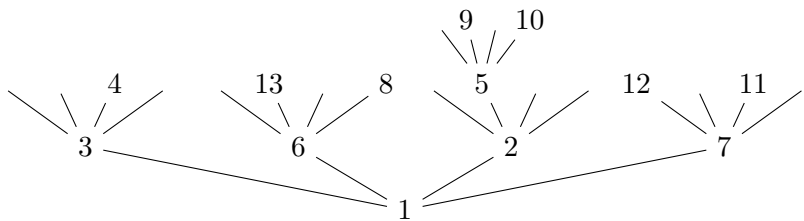

\tikzpic{-0.5}{
[grow'=up,level distance=0.84cm, 
level 1/.style={sibling distance=2.8cm},
level 2/.style={sibling distance=0.77cm},
level 3/.style={sibling distance=0.42cm}
]
\node {$1$}
	child{node{$3$}
		child{node{}}
		child{node{}}
		child{node{$4$}}
		child{node{}}
	}
	child{node{$6$}
		child{node{}}
		child{node{$13$}}
		child{node{}}
		child{node{$8$}}
	}
	child{node{$2$}
		child{node{}}
		child{node{$5$}
			child{node{}}
			child{node{$9$}}
			child{node{}}
			child{node{$10$}}
		}
		child{node{}}
		child{node{}}
	}
		child{node{$7$}
		child{node{$12$}}
		child{node{}}
		child{node{$11$}}
		child{node{}}
	};
}
\caption{A labeled $4$-ary tree}
\label{fig:4tree}
\end{figure}
The three labeled binary trees $\mathfrak{D}(L)$ are depicted in Figure \ref{fig:3btree}.
\begin{figure}[ht]
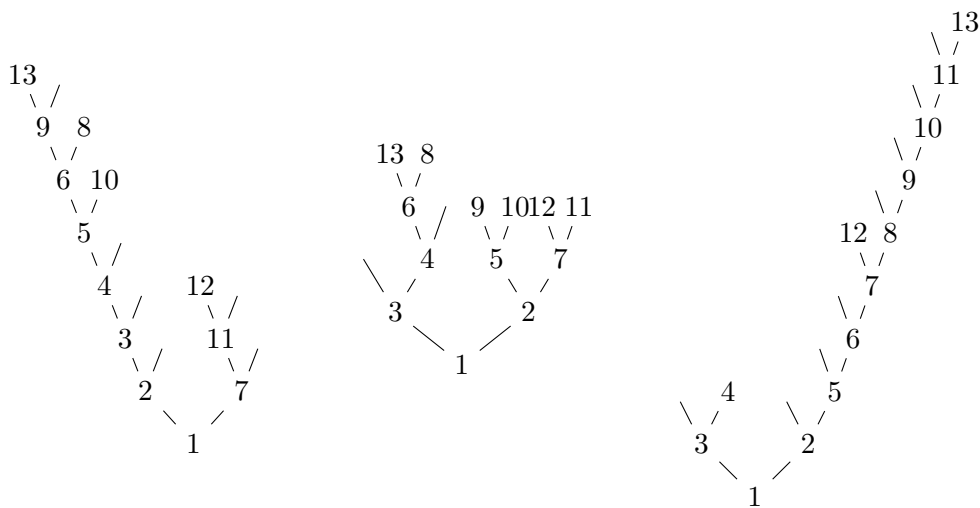

\tikzpic{-0.5}{
[scale=0.7,grow'=up,level distance=1cm, 
level 1/.style={sibling distance=1.8cm},
level 2/.style={sibling distance=0.77cm}
]
\node {$1$}
	child{node{$2$}
		child{node{$3$}
			child{node{$4$}
				child{node{$5$}
					child{node{$6$}
						child{node{$9$}child{node{$13$}}child{node{}}}
						child{node{$8$}}
					}
					child{node{$10$}}
				}
				child{node{}}
			}
			child{node{}}
		}
		child{node{}}
	}
	child{node{$7$}
		child{node{$11$}
			child{node{$12$}}
			child{node{}}
		}
		child{node{}}
	};
}\qquad
\tikzpic{-0.5}{
[scale=0.7,grow'=up,level distance=1cm, 
level 1/.style={sibling distance=2.5cm},
level 2/.style={sibling distance=1.2cm},
level 3/.style={sibling distance=0.7cm}
]
\node{$1$}
   child{node{$3$}child{}
   	child{node{$4$}
   		child{node{$6$}child{node{$13$}}child{node{$8$}}}
   		child{}
   	}
   }
   child{node{$2$}
   	child{node{$5$}
   		child{node{$9$}}
   		child{node{$10$}}
   	}
   	child{node{$7$}
   		child{node{$12$}}
   		child{node{$11$}}
   	}
   };	
}\quad
\tikzpic{-0.5}{
[scale=0.7,grow'=up,level distance=1cm, 
level 1/.style={sibling distance=2cm},
level 2/.style={sibling distance=1cm},
level 3/.style={sibling distance=0.7cm}
]
\node{$1$}
	child{node{$3$}
		child{node{}}
		child{node{$4$}}
	}
	child{node{$2$}
		child{node{}}
		child{node{$5$}
			child{node{}}
			child{node{$6$}
				child{node{}}
				child{node{$7$}
					child{node{$12$}}
					child{node{$8$}
						child{node{}}
						child{node{$9$}
							child{node{}}
							child{node{$10$}
								child{node{}}
								child{node{$11$}
									child{node{}}
									child{node{$13$}}
								}
							}
						}
					}
				}
			}
		}
	};
}

\caption{Three binary trees $\mathfrak{D}(L)=(D_1,D_2,D_3)$ corresponding to the $4$-ary tree $L$ in Figure \ref{fig:4tree}.
The trees $D_1$, $D_2$ and $D_3$ are depicted from left to right.}
\label{fig:3btree}
\end{figure}
\end{example}

\begin{defn}
\label{defn:chi}
We denote by $\chi$ the map from a labeled $k$-ary tree $L$ to $\mathfrak{D}(L)$ defined as above. 
\end{defn}

Recall a labeled binary tree corresponds to a noncommutative crossing partition 
as in Section \ref{sec:lbt}.
The next proposition connects a set $\mathfrak{D}(L)$ of labeled binary trees 
with $k-1$-chain in the Hasse diagram of $\mathcal{L}_{NCCP}$.

\begin{prop}
Let $\mathfrak{D}(L):=(D_1,\ldots,D_{k-1})$ be the set of $k-1$ labeled binary 
tree obtained from a labeled $k$-ary tree $L$ by $\chi$.
Then, $\mathfrak{D}(L)$ satisfies 
\begin{align}
\label{eq:Drising}
D_{1}\le D_{2}\le \ldots \le D_{k-1},
\end{align} 
in $\mathcal{L}_{NCCP}=(\mathtt{NCCP}(n),\le)$.
\end{prop}
\begin{proof}
By construction, we have $RE(i)\subseteq RE(i+1)$ and $LE(i)\supseteq LE(i+1)$ 
for $1\le i\le k-2$.	
If $RE(i)=RE(i+1)$, or equivalently $LE(i)=LE(i+1)$ for all labels in $L$, we obviously have $D_{i}=D_{i+1}$.
So, we assume that $RE(i)\subset RE(i+1)$ and $LE(i)\supset LE(i+1)$.
Let $n'\in[1,n]$ be the smallest integer such that $n'\in RE(i+1)$ and $n'\notin RE(i)$. 
Let $D'_{i}=D'_{i+1}$ be the two labeled binary trees whose labels are up to $n'-1$.
We add the node labeled $n'$ to $D'_{i}$ and $D'_{i+1}$.
Then, it is easy to see that if we add the node labeled $n'$ to $D'_i$ and $D'_{i+1}$, 
$D_{i}<D_{i+1}$ by construction of $D_{i}$ and $D_{i+1}$ and by definition of the rotation of $D_{i}$.
The node $n'$ is connected to its parent node by a left edge in $D'_{i}$ and by a right edge 
in $D'_{i+1}$. The position of the node $n'$ in $D'_{i+1}$ is right to that in $D'_{i}$.

Suppose that we have $D'_{i}<D'_{i+1}$ where the labels are up to $n'-1$.
We want to show that $D_{i}<D_{i+1}$ after the addition of the node labeled $n'$.
This is guaranteed by the compatibility of the definition of the rotation on $D_{i}$ 
and the recursive construction of $D_{i}$ and $D_{i+1}$.
The position of the node $n'$ in $D_{i+1}$ is always weakly right to that in $D_{i}$.	
Therefore, we have Eq. (\ref{eq:Drising}) for $\mathfrak{D}(L)$.
\end{proof}

\begin{example}
We consider the same example as in Example \ref{ex:ktobinary}.
The three noncommutative crossing partitions associated to the 
$4$-ary tree in Figure \ref{fig:4tree} are 
\begin{align*}
d/9/68/5a/4/3/2/1c/b/7 < 3d/68/4/19/5a/2c/7d < 34/1256c/789abd,
\end{align*}
from left to right where $(a,b,c,d)=(10,11,12,13)$.
See also Figure \ref{fig:3btree} for the three labeled binary 
trees.
\end{example}

The next theorem gives the correspondence between a labeled $k$-ary 
tree and a $k-1$-chain of labeled binary trees.
\begin{theorem}
\label{thrm:bijchi}
Let $L$ be a labeled $k$-ary tree.
The map $\chi: L \mapsto\mathfrak{D}(L)$ is a bijection, where 
$\mathfrak{D}(L)$ satisfies Eq. (\ref{eq:Drising}).
\end{theorem}

We explicitly give an inverse $\chi^{-1}:\mathfrak{D}(L)\mapsto L$ to prove 
Theorem \ref{thrm:bijchi}.
We assume that $\mathfrak{D}(L):=(D_1,\ldots,D_{k-1})$ satisfies Eq. (\ref{eq:Drising}).

Since $D_{i}\le D_{i+1}$, the size of right-extended sequence for any node in $D_{i}$
is longer than or equal to that in $D_{i+1}$.
We denote by $\overrightarrow{q}(\mathfrak{r})|_{i}$ the right-extended sequence 
of the root (node labeled one) $\mathfrak{r}$ in $D_{i}$.
Let $i\in[1,k-2]$.
We have two cases: 1) $\overrightarrow{q}(\mathfrak{r})|_{i}=\overrightarrow{q}(\mathfrak{r})|_{i+1}$,
and 2) $\overrightarrow{q}(\mathfrak{r})|_{i}\subset\overrightarrow{q}(\mathfrak{r})|_{i+1}$.

\paragraph{Case 1)}
The tree has no subtree whose root is connected to the root $\mathfrak{r}$ by 
an $i+1$-edge.

\paragraph{Case 2)}
Let $R(\mathfrak{r};i)$ be the set of labels which is right to the root $\mathfrak{r}$ 
in $D_{i}$. 
The tree has a subtree $T'$ whose root is connected to the root $\mathfrak{r}$
by an $i+1$-edge.
The subtree $T'$ consists of the labels in $R(\mathfrak{r};i+1)\setminus R(\mathfrak{r};i)$.

By the above procedures, one knows the labels in the subtree which is connected 
to the root by a $p$-edge in $L$.
By replacing the root by a node labeled $l$ in $[1,n]$ one-by-one in the above procedure, 
we extract the tree structure above the node labeled $l$ and reconstruct 
a $k$-ary tree $L$.

Further, we extract some simple combinatorial structures of a tree from two 
labeled binary trees $D_{1}$ and $D_{k-1}$:
\begin{enumerate}
\item If two labels $l_1$ and $l_2$ are connected by a right edge in $D_{1}$,
the nodes labeled $l_1$ and $l_2$ are connected by a $1$-edge in $L$. 
\item 
If two labels $l_1$ and $l_2$ are connected by a left edge in $D_{k-1}$,
the nodes labeled $l_1$ and $l_2$ are connected by a $k$-edge in $L$.
\end{enumerate}

\begin{proof}[Proof of Theorem \ref{thrm:bijchi}]
The theorem follows from the fact that the map $\chi$ has an inverse as above.
\end{proof}

The next proposition implies that the definitions of canonical trees in 
labeled $k$-ary trees and $k-1$ labeled binary trees are compatible
with each other.
\begin{prop}
\label{prop:kLTcano}
Let $L(T)$ be a labeled $k$-ary tree, and $\mathfrak{D}(L(T))$ be 
the set of $k-1$ labeled binary trees corresponding to $L(T)$.
Then, the followings are equivalent:
\begin{enumerate}
\item $L(T)$ is canonical.
\item $k-1$ labeled binary trees in $\mathfrak{D}(L(T))$ are all canonical.
\end{enumerate}
\end{prop}
\begin{proof}
Let $L(T')$ be a labeled $k$-ary tree whose labels are up to $n-1$, 
and $\mathfrak{D}(L(T'))$  be the set of $k-1$ 
labeled binary trees obtained from $L(T')$, i.e., $\mathfrak{D}(L(T')):=(D'_1,\ldots,D'_{k-1})$.
We prove the statement by induction on $n$.
When $n=1,2$, the conditions are trivial and we have (1) $\Leftrightarrow$ (2).
We assume $n\ge2$, and both $L(T')$ and $\mathfrak{D}(L(T'))$ are canonical.

(1) $\Rightarrow$ (2). 
Suppose that the node labeled $n$ is a child of the node labeled $n-1$ in $L(T)$.
In this case, the node labeled $n-1$ has neither right or left children 
in $D_{i}$, $1\le i\le k-1$. 
The node labeled $n$ is connected to the node 
labeled $n-1$ by a left or right edge in $D_{i}$ by the map $\chi$.
Then, it is obvious that all $D_{i}$'s are canonical.

Suppose that the node labeled $n$ is a child of the node labeled $n'$ with $n'\le n-2$
in $L(T)$. 
Suppose that two nodes labeled $n$ and $n'$ are connected by a $p$-edge. 
Since $L(T)$ is canonical and $n$ is a maximal label, 
there is no subtree such that its root is connected to the node $n'$ by 
a $j$-edge for some $1\le j\le p-1$.
Further, we have the following condition:
\begin{enumerate}[(A)]
\item
A node whose label is in $[n'+1,n-1]$ is in a subtree 
such that its root is connected to the node $n'$ by a $j$-edge for some $p+1\le j\le k-1$.
\end{enumerate}
These are because of the definition of being canonical.
If $1\le i\le p-1$, we add the node labeled $n$ to $D'_{i}$ by a left edge.
In $D'_{i}$, $n'$ has a left sub-tree, and we take a right-most path of nodes from $n'$, which starts 
by a left edge, to add the node labeled $n$. Such a path always exists by the condition (A).
Then, it is easy to see that $D_{i}$ is canonical.
If $p\le i\le k-1$, we add the node labeled $n$ to $D'_{i}$ by a right edge.
In this case, we lengthen the right-extended sequence of the node $n'$. 
Since all the labels in $[n'+1,n-1]$ are weakly left to the node $n$ in $D_{i}$,
$D_{i}$ is canonical. 

(2) $\Rightarrow$ (1). 
Suppose we have the node labeled $n$ is connected to the parent node by 
a left edge in $D_{i}$, $1\le i\le p-1$, and by a right edge in $D_{j}$, 
$p\le j\le k-1$.

Suppose that the node labeled $n$ is connected to the node labeled $n'$ 
in $D_{p}$. Note that $n'<n$.
If $n'=n-1$, the node labeled $n$ is connected to the node $n'$ by the $p$-edge in $L(T)$.
This implies that the node $n$ is weakly right to the node $n'$, and hence	 
$L(T)$ is canonical.
Suppose $n'<n-1$.
The condition that a binary tree $D_p$ is canonical implies that 
a node whose label is in $[n'+1,n-1]$ is in a left subtree whose root is $n'$.
We add the node labeled $n$ to $L(T')$ such that it is connected to the node 
labeled $n'$ by a $p$-edge in $L(T)$. The condition that all $D_{i}$'s are canonical 
implies that the nodes whose labels are in $[n'+1,n-1]$ are weakly left to the node 
$n$ in all $D_{i}$.
From these, a node whose label is in $[n'+1,n-1]$ is weakly right to the node $n'$
and weakly left to the node $n$ in $L(T)$.
Then, it is easy to see that $L(T)$ is canonical.

From above considerations, we have (1) $\Leftrightarrow$ (2).
\end{proof}

\begin{remark}
Since a canonical labeled $k$-ary tree corresponds to an element in $\mathtt{NCCP}(n;312)$,
Proposition \ref{prop:kLTcano} gives the same correspondence between 
a non-crossing partition and a $k-1$ (unlabeled) binary trees.
This is already observed in \cite{Edel80,Edel82}.
\end{remark}

Let $k\ge2$.
Let $T$ be a $k$-ary tree with $n$ nodes and $h(T)$ be the integer sequence 
satisfying the condition (\ref{eq:condh}).
By Theorem \ref{thrm:bijchi}, we have a bijection between a labeled $k$-ary 
tree $L$ and a set of $k-1$ binary trees $\mathfrak{D}(L):=(B_{1},\ldots,B_{k-1})$.

\begin{prop}
Suppose that, for all $1\le i\le n-1$, the node labeled $i$ is connected by an edge 
to the node labeled $i+1$ in $L$.
Then, we have 
\begin{align}
\label{eq:hlabel}
h(L)=\sum_{1\le i\le k-1}h(B_{i}).
\end{align} 
\end{prop}
\begin{proof}
By the construction of the bijection $\chi$ from $L$ to $\mathfrak{D}(L)$,
it is obvious that the node labeled $i$ is connected by an edge to the node 
labeled $i+1$ in $B_{p}$ for all $1\le p\le k-1$.
We prove Eq. (\ref{eq:hlabel}) by induction on $n$.
In the case of $n=1$ and $2$, Eq. (\ref{eq:hlabel}) is obviously satisfied.
Suppose that Eq. (\ref{eq:hlabel}).
We denote $h(L)=:(h_1,\ldots,h_{n})$ and $h(B_{i})=:(h'_{i,1},\ldots,h'_{i,n})$.
Equation (\ref{eq:hlabel}) is equivalent to $h_{j}=\sum_{1\le i\le k-1}h'_{i,j}$
for all $1\le j\le k-1$.

Suppose that the node labeled $n$ is connected to its parent node by a $p$-edge.
Then, by definition of $h(L)$, we have $h_{n}=h_{n-1}+p$.
In $\mathfrak{D}(L)$, the node labeled $n$ is connected to its parent edge 
by a left edge in $B_{i}$ for $1\le i\le p$ and by a right edge in $B_{i}$ for 
$p+1\le i\le k-1$.
The connection by a left (resp. right) edge in $B_{i}$ implies that $h'_{i,n}=h'_{i,n-1}+1$ 
(resp. $h'_{i,n}=h'_{i,n-1}$).
Therefore, we have $\sum_{1\le i\le k-1}h'_{i,n}=\sum_{1\le i\le k-1}h'_{i,n-1}+p$.
On the other hand, by induction assumption, we have $h_{n-1}=\sum_{1\le i\le k-1}h'_{i,n-1}$.
From these we have $h_{n}=\sum_{1\le i\le k-1}h'_{i,n}$, 
which completes the proof.
\end{proof}

\begin{remark}
The equality in Eq. (\ref{eq:hlabel}) does not hold for a general labeled tree $L$.
For example, consider the $3$-ary labeled tree $L$ with $h(L)=(0,2,1)$.
Then, we have two binary trees $B_{1}$ and $B_{2}$ such that 
$h(B_{1})=(0,1,2)$ and $h(B_{2})=(0,1,0)$.
From these we have $(0,2,1)\neq (0,2,2)=(0,1,2)+(0,1,0)$.
This comes from the operation (3-b-ii) above Example \ref{ex:LDD}.
\end{remark}

\subsection{Intervals of the lattice}
\label{sec:interlat}
Fix an integer $k\ge2$.

The following theorem is one of the main results. 
We count the number of chains of length $k-1$ by use of the bijection between 
a $k-1$-chain and a $k-1$-Dyck tilings.
Further, we translate the condition that all elements of a chain are 
contained in $\mathtt{NCCP}(n;312)$ into the condition on 
$k-1$-Dyck tilings.

\begin{theorem}
\label{thrm:interval}
Let $c(\pi,\pi';k-1):\pi=\pi_{1}\le \pi_{2}\le\ldots\le \pi_{k-1}=\pi'$ be a chain in the lattice $\mathcal{L}_{NCCP}$.
Then, we have   	
\begin{enumerate}
\item 
The number of chains of length $k-1$ in $\mathcal{L}_{NCCP}$ 
is given by 
\begin{align}
\label{eq:numck1}
\sum_{\pi}\sum_{\pi':\pi\le\pi'}|c(\pi,\pi';k-1)|
=\prod_{j=0}^{n-1}((k-1)j+1).
\end{align}
\item The number of chains of length $k-1$ which contain only elements in 
$\mathtt{NCCP}(n;312)$ is given by the $n$-th Fuss--Catalan number.
\end{enumerate}
\end{theorem}
\begin{proof}
(1) The bijection $\phi^{-1}$ defined in Section \ref{sec:lbt} gives $k-1$ labeled binary trees 
$D_1,\ldots,D_{k-1}$ where $D_{i}:=\phi^{-1}(\pi_{i})$.
Note that $\phi^{-1}$ is the order-preserving, which implies 
$D_{i}\le D_{i+1}$.
Then, by the bijection $\chi$ defined in Definition \ref{defn:chi} gives 
a labeled $k$-ary tree.
Therefore, we have a bijection from a chain $c(\pi,\pi';k-1)$ 
and a labeled $k$-ary tree.
From Proposition \ref{prop:numlktree}, the number of chains of length $k-1$
is given by Eq. (\ref{eq:numck1}).

(2) Since a chain contains only elements in $\mathtt{NCCP}(n;312)$, 
the labeled binary tree $D_{i}:=\phi^{-1}(\pi_{i})$ is canonical by Lemma \ref{lemma:312clbt}.
Denote by $\mathfrak{D}:=(D_1,\ldots,D_{k-1})$.
By Proposition \ref{prop:kLTcano}, the labeled $k$-ary tree obtained by 
$\chi(\mathfrak{D})$ is canonical.
Since both $\phi^{-1}$ and $\chi$ are bijections, the number of chains which contain
only elements in $\mathtt{NCCP}(n;312)$ is equal to the number of canonical 
$k$-ary trees, which is the $n$-th Fuss--Catalan number.
\end{proof}

\begin{remark}
Since the lattice $(\mathtt{NCCP}(n;312),\le)$ is isomorphic to the Kreweras lattice 
$\mathcal{L}_{Kre}=(\mathtt{NCP}(n),\subseteq)$, 
the number of chains of length $k-1$ in $\mathcal{L}_{Kre}$ is given by 
the Fuss--Catalan number by Theorem \ref{thrm:interval}.
We recover the results in \cite{Edel82} (see also \cite{Edel80}).
\end{remark}

The following proposition follows from the schematic correspondence as shown 
in Figure \ref{fig:cco}.
\begin{prop}
Let $c(\pi,\pi';k)$ be a chain in $\mathcal{L}_{NCCP}(n)$ as above.
\begin{enumerate}
\item We have a bijection between $c(\pi,\pi';k)$ and a $k$-Dyck tiling.
\item Let $D$ be a $k$-Dyck tiling corresponding to a chain $c(\pi,\pi';k)$. 
The followings are equivalent:
\begin{enumerate}
\item The chain $c(\pi,\pi';k)$ contains a partition $\pi''\notin\mathtt{NCCP}(n;312)$.
\item The $k$-Dyck tiling $D$ has non-trivial tiles. 
\end{enumerate}
\end{enumerate}
\end{prop}
\begin{proof}
(1) Given a labeled $(k+1)$-ary tree $L(T)$, we have a unique chain $c(\pi,\pi';k)$ in the 
lattice $\mathcal{L}_{NCCP}$ by decomposing $L(T)$ into $k$ labeled binary trees.
This decomposition is unique and gives a bijection between them.
Similarly, we have a bijection between a tree $L(T)$ and a $k$-Dyck tiling.
By composing the bijections, we have a bijection between $c(\pi,\pi';k)$ and 
a $k$-Dyck tiling.

(2) 
We consider the set of canonical labeled $k$-ary trees as in Figure \ref{fig:cco}.
On one hand, we have a bijection between $c(\pi,\pi';k)$ which contains only elements in $\mathtt{NCCP}(n;312)$
and  a canonical labeled $k$-ary tree.
On the other hand, we have a bijection between a canonical $k$-ary tree and a $k$-Dyck tiling without non-trivial 
tiles.
Thus, we have a bijection between a $k$-Dyck tiling without non-trivial tiles and $c(\pi,\pi';k)$ whose elements
are in $\mathtt{NCCP}(n;312)$.
By contraposition, (a) and (b) are equivalent.
\end{proof}

\bibliographystyle{amsplainhyper} 
\bibliography{biblio}

\end{document}